\documentclass[12pt,reqno]{amsart}

\usepackage{color}
\usepackage{esint,amssymb}
\usepackage{graphicx}
\usepackage{tikz}

\usepackage{hyperref}

\newtheorem{theorem}{Theorem}
\newtheorem{proposition}[theorem]{Proposition}
\newtheorem{lemma}[theorem]{Lemma}
\newtheorem{corollary}[theorem]{Corollary}

\theoremstyle{definition}
\newtheorem{remark}[theorem]{Remark}

\numberwithin{equation}{section}
\numberwithin{theorem}{section}

\newtheorem{zero}{Theorem}

\newtheorem{one}{Proposition}

\newtheorem{two}{Proposition}

\newtheorem{three}{Proposition}

\newtheorem{zerobdry}{Theorem}

\newtheorem{four}{Proposition}

\newtheorem{five}{Proposition}

\newtheorem{six}{Proposition}

\newtheorem{seven}{}

\newcommand{\Z}{\mathbb{Z}}

\newcommand{\R}{\mathbb{R}}

\newcommand{\HH}{\mathcal{H}}

\newcommand{\ep}{\varepsilon}

\newcommand{\J}{\mathcal A}
\newcommand{\la}{\lambda}

\newcommand{\abs}[1]{\left| #1 \right|}

\newcommand{\average}{{\mathchoice {\kern1ex\vcenter{\hrule height.4pt
width 6pt depth0pt} \kern-9.7pt} {\kern1ex\vcenter{\hrule
height.4pt width 4.3pt depth0pt} \kern-8pt} {} {} }}
\newcommand{\ave}{\average\int}

\renewcommand{\bar}{\overline}
\renewcommand{\tilde}{\widetilde}

\begin{document}

\title[H\"older regularity of stable solutions]
{H\"older regularity of stable solutions to  elliptic equations up to $\mathbf{\mathbb{R}^9}$: \\
full quantitative proofs}

\author[X. Cabr\'e]{Xavier Cabr\'e}
\address{﻿X.C.\textsuperscript{1,2,3} ---
\textsuperscript{1}ICREA, Pg.\ Lluis Companys 23, 08010 Barcelona, Spain \& 
\textsuperscript{2}Universitat Polit\`ecnica de Catalunya, Departament de Matem\`{a}tiques and IMTech, 
Diagonal 647, 08028 Barcelona, Spain \&
\textsuperscript{3}Centre de Recerca Matem\`atica, Edifici C, Campus Bellaterra, 08193 Bellaterra, Spain
}
\email{xavier.cabre@upc.edu}

\begin{abstract}
This article concerns the results obtained in [Cabr\'e, Figalli, Ros-Oton, and Serra, Acta Math.\ 224 (2020)], which established the H\"older regularity of stable solutions to semilinear elliptic equations in the optimal range of dimensions $n\leq 9$. For expository purposes, we provide self-contained proofs of all results. 
They involve only basic Analysis tools and are intended to be accessible to a broader mathematical audience beyond PDE specialists.

Two of the results in the 2020 article relied on compactness arguments.  Here we present, instead, quantitative proofs from the more recent paper [Cabr\'e, to appear in Amer. J. Math, arXiv:2211.13033]. They allow to quantify the H\"older regularity exponent and simplify significantly the treatment of boundary regularity. 

We also comment on similar progress and open problems for related equations.
\end{abstract}


\dedicatory{In memory of Haim Brezis and Ireneo Peral, in their debt}

\thanks{The author was supported by grants PID2021-123903NB-I00 and RED2022-134784-T funded by MCIN/AEI/10.13039/501100011033 and by ERDF ``A way of making Europe'', and by the Catalan grant 2021-SGR-00087, as well as by the Spanish State Research Agency through the Severo Ochoa and Mar\'{\i}a de Maeztu Program for Centers and Units of Excellence in R\&D (CEX2020-001084-M)}

\maketitle

\
\vspace{-1cm}

{\small
\tableofcontents
}

\section{Introduction \vspace{.15cm}}

The regularity of stable solutions to semilinear elliptic PDEs has been a topic of study since the 1960's, motivated by a combustion problem raised by Gel'fand~\cite{Gelf}. The topic concerns the equation $-\Delta u = f(u)$ posed in a bounded domain $\Omega$ of $\R^n$. The main question is to know in which dimensions stable energy solutions must be regular or, in the contrary, can be unbounded ---that is, singular.  It was quickly noticed that, in dimensions 10 and higher, there exist explicit stable energy solutions which are singular. 

The first regularity results appear in the seventies, in works of Joseph and Lundgren~\cite{JL} and of Crandall and Rabinowitz~\cite{CR} on an important class of stable solutions: the so called ``extremal solutions'', usually denoted by~$u^\star$.  In the mid-nineties, the theory experienced a revival after new progress was made by Brezis and collaborators~\cite{BCMR, BV}. Among many other questions on the topic, Brezis asked the following one, that we quote literally:
 
\vspace{1.5mm}

{\it \cite[Brezis; Open problem 1]{Brezis}. Is there something ``sacred'' about dimension 10? More precisely, is it possible in ``low'' dimensions to construct some $f$ (and some~$\Omega$) for which the extremal solution $u^\star$ is unbounded? Alternatively, can one prove in ``low'' dimension that $u^\star$ is
smooth for every $f$ and every $\Omega$? }
\vspace{1.5mm}

Note that the question did not exclude any of the two possible answers ---at that time, it was not clear which one would be correct. Although there had been some partial answers, the problem remained open until 2019, when Cabr\'e, Figalli, Ros-Oton, and Serra~\cite{CFRS} finally established the regularity of stable solutions up to the optimal dimension 9. 

In this article we will present self-contained proofs of the results of \cite{CFRS}. Following a more recent paper of the author, \cite{C22quant}, we will incorporate simplified and quantitative proofs of two of the results in~\cite{CFRS} which were proven, there, by contradiction-compactness arguments. These new tools have led, in works by Erneta~\cite{Ern1, Ern2,Ern3}, to analogue results for equations with variable coefficients.

We will also comment on similar progress and on open problems for related equations, including MEMS-type problems for the Laplacian, as well as the $p$-Laplacian and the fractional Laplacian cases.

For expository purposes, the presentation in the current article is self-contained, in the sense that all results are fully proven. It is intended to be accessible to a broader mathematical audience beyond PDE specialists, since we only use basic Analysis tools ---for instance, in the interior regularity proofs, we do not even use the Sobolev inequality.
The exposition is divided into two parts: we first treat interior regularity and later regularity up to a flat boundary. We treat boundary regularity only up to the flat boundary of a half-ball to make the exposition more transparent and shorter, and also since this case already contains the essential ideas in the boundary regularity results of~\cite{CFRS, Ern2, Ern3} for more general domains.

All results in this article are taken either from~\cite{CFRS} or \cite{C22quant}, with the exception of  simplified proofs of the $W^{1,2}$ estimates of Propositions~\ref{prop:2} and~\ref{prop:bdry-nabla2byL1}. These proofs appear here for the first time and are based on a new interpolation lemma, proved in Appendix~\ref{app:interp}.

Appendices \ref{app:morrey} and \ref{app:example} do not appear in~\cite{CFRS}, nor in~\cite{C22quant}. The first one is a quick proof of Morrey's embedding. Appendix~\ref{app:example}, instead, is not needed in our proofs but addresses an interesting Real Analysis question raised (already back to 2013) by the crucial lemma of~\cite{CFRS}, namely, Proposition~\ref{prop:1} below.

The book~\cite{Dup} and the short survey \cite{C17} present the general framework and previous regularity results on stable solutions, except for those from \cite{CFRS} and \cite{C22quant}. Besides regularity issues, \cite{Dup} treats many other questions on stable solutions.

\subsection{The setting and previous works}\label{subsec-plan}
\
\smallskip

Consider the semilinear elliptic equation
\begin{equation}\label{eq:PDE}
-\Delta u=f(u) \qquad \text{in }\Omega\subset \R^n,
\end{equation}
where $\Omega$ is a bounded domain, $f\in C^1(\R)$, and  $u:\overline\Omega \subset \R^n\to \R$.
Notice that \eqref{eq:PDE} is the Euler-Lagrange equation of the functional
\begin{equation}\label{energy}
\mathcal E(u):=\int_{\Omega}\Bigl(\frac{|\nabla u|^2}2-F(u)\Bigr)\,dx,
\end{equation}
where $F$ is a primitive of $f$. 
The second variation of $\mathcal E$ at $u$ is given by
$$
\frac{d^2}{d\varepsilon^2}\Big|_{\varepsilon=0}\mathcal E(u+\varepsilon\xi)
=\int_{\Omega}\Bigl(|\nabla \xi |^2-f'(u)\xi^2\Bigr)\,dx.
$$
We say that $u$ is a {\it stable} solution\footnote{At this point,  in this definition we may restrict ourselves to regular solutions, in which case all quantities in~\eqref{stabilityLip} are well defined. Later we will consider distributional solutions for nonlinearities with $f'\geq0$, in which case \eqref{stabilityLip} is also well defined.} of equation \eqref{eq:PDE} in $\Omega$ if the second variation at $u$ is nonnegative, namely, if 
\begin{equation}\label{stabilityLip} 
\int_{\Omega} f'(u) \xi^2\,dx\leq \int_{\Omega} |\nabla \xi|^2\,dx  \quad \mbox{ for all }
\xi\in C^1(\overline\Omega) \text{ with } \xi_{|_{\partial\Omega}}\equiv 0.
\end{equation}
Note that the stability of $u$ is defined within the class of functions $u+\varepsilon\xi$ agreeing with~$u$ on the boundary $\partial\Omega$. Stability is equivalent to assuming the nonnegativeness of the first
Dirichlet eigenvalue in~$\Omega$ for the linearized operator of \eqref{eq:PDE} at~$u$, namely, $-\Delta-f'(u)$.

If a {\it local minimizer} of $\mathcal E$ exists (that is, a minimizer for small perturbations having same boundary values), then it will be a stable solution as well.

This paper concerns the regularity of stable solutions to \eqref{eq:PDE}, a question that has been investigated since the seventies. It was motivated by the Gel'fand-type problems \eqref{eq:lambda} described below, for which $f(u)=e^u$ is a model case. Notice that for this nonlinearity, and more generally when $f$ satisfies \eqref{nonl-G} below, the energy functional \eqref{energy}, among functions with same boundary values  on $\partial\Omega$ as any given function $u$, is unbounded below.\footnote{\label{Foot1}To see this,  consider $\mathcal E(u+tv)$ for any $v \in C^1_c(\Omega)$ with $v\geq0$ and $v\not\equiv 0$, and let $t\to +\infty$.}  Thus, the functional admits no absolute minimizer. However, we will see that, in some relevant cases, it admits stable solutions. This class of solutions is the subject of this article.

Two other remarks may be helpful to settle the general framework and goals. First, note that once a solution is proven to be bounded, further regularity follows from classical linear elliptic theory, since the right-hand side $f(u)$ becomes bounded.\footnote{Hence, $f(u)\in L^p$ for all $p<\infty$ and one can apply first Calder\'on-Zygmund theory (to gain two derivatives in $L^p$ and get $u\in W^{2,p}\subset C^{1,\varepsilon}$ when $p>n$), and later Schauder theory (to gain two derivatives in $C^\varepsilon$ and get $u\in C^{3,\varepsilon}$ when $f\in C^{1,\varepsilon}$ ---and even further H\"older regularity if $f$ allows. Note that when $f\in C^1$ only, then $f(u)\in C^\alpha$ for all $\alpha\in (0,1)$, and thus $u\in C^{2,\alpha}$ for all $\alpha\in (0,1)$ is the best regularity we can expect). The statements of the Schauder and Calder\'on-Zygmund theories are well summarized in the beginning of \cite[Appendix B]{Struwe}; full proofs can be found in Chapters 4 and 9, respectively, of \cite{GT}.} Second, recall that within the class of all solutions (not being necessarily stable), $L^\infty$~bounds are available, in general domains, only for nonlinearities $f$ growing at most as the critical Sobolev power, i.e., $|f(u)|\leq C(1+|u|)^p$ with $p\leq (n+2)/(n-2)$ for $n\geq 3$. For a proof of this fact, see the application after Lemma B.3 in \cite{Struwe}. In the current paper the goal is to reach a larger class of nonlinearities when we restrict ourselves to stable solutions.

Our problem is a PDE analogue of another fundamental question: the regularity of stable minimal hypersurfaces in $\R^n$. These surfaces may not be absolute minimizers of area. It is known that they may be singular for $n\geq 8$ ---the Simons cone being an explicit example. In contrast, their regularity was established for surfaces in $\R^3$ in the late seventies (independently by Fischer-Colbrie and Schoen~\cite{Fischer-Schoen} and by Do Carmo and Peng~\cite{doCarmo-Peng}) and very recently (in the 2020's) in $\R^4$ by Chodosh and Li~\cite{CL}, in $\R^5$ by Chodosh, Li, Minter, and Stryker~\cite{CLMS}, and in $\R^6$ by Mazet~\cite{M}. The question remains open for $n= 7$. 

In our setting, as in minimal surfaces, there is an explicit singular stable solution (with finite energy) in large dimensions. Indeed, for $n\geq 3$,
\begin{equation}\label{logsoln}
u=\log\frac{1}{|x|^2} \quad \text{solves \eqref{eq:PDE} weakly, with } f(u)=2(n-2)e^u \text{ and } \Omega=B_1.
\end{equation}
Using Hardy's inequality,\footnote{See Proposition 1.20 in \cite{CP}, for instance, for its very simple proof.} we see that the linearized operator $-\Delta-2(n-2)e^u=-\Delta-2(n-2)|x|^{-2}$ is nonnegative when $2(n-2)\leq (n-2)^2/4$, i.e., when $n\ge 10$. Thus, there exist singular $W^{1,2}(B_1)$ stable solutions to equations of the form \eqref{eq:PDE} whenever $n\geq 10$. From our results we will see that this does not occur for $n\leq 9$.

Stable solutions exist and play an important role in the so called Gel'fand-type problems. The case $f(u)=e^u$ arose in~\cite{Gelf}, in the sixties, from an ignition problem in combustion theory. It motivated the subsequent literature on stable solutions. See Section~3 of \cite{C22} for the physical meaning, in combustion, of the elliptic problem~\eqref{eq:lambda} below and its relation with the corresponding nonlinear heat equation.
 
Assume that $f:[0,+\infty)\to\R$ is $C^1$ and that
\begin{equation}\label{nonl-G}
f\text{ is nondecreasing, convex, and satisfies } f(0)>0 \text{ and } 
\lim_{t\to+\infty}\frac{f(t)}{t}=+\infty. 
\end{equation}
Given a constant $\lambda > 0$ consider the problem\footnote{Problem \eqref{eq:lambda} can be written equivalently in the form \eqref{BVP} used later. Indeed, assuming \eqref{nonl-G}, extend $f$ to $(-\infty,0)$, no matter how but to have $f\in C^1(\R)$ and $f\geq 0$ in $\R$. Then, by the maximum principle, solutions of the problem $-\Delta u=\lambda f(u)$ in $\Omega$, $u=0$ on $\partial\Omega$, coincide (for $\lambda >0$) with those of \eqref{eq:lambda}.}
\begin{equation}
\label{eq:lambda}
\left\{
\begin{array}{cl}
-\Delta u=\lambda f(u) & \text{in }\Omega\\
u>0 & \text{in }\Omega\\
u=0 & \text{on }\partial\Omega.\end{array}
\right.
\end{equation}
Note first that $-\Delta u=\lambda f(u)$ admits no trivial solution ($u\equiv 0$ is not a solution since $\lambda f(0)>0$). At the same time, the functional $\mathcal E$ is unbounded from below (see Footnote~\ref{Foot1}) and thus admits no absolute minimizer. However, \eqref{eq:lambda} admits a unique {\it stable} classical solution $u_\lambda$ for every $\lambda \in (0,\lambda^\star)$, where $\lambda^\star\in (0,+\infty)$ is a certain parameter; see e.g.\ the book \cite{Dup} by Dupaigne.\footnote{See Section 2 of \cite{C22} for a related situation concerning minimal surfaces. It concerns catenoids (soap films formed between two coaxial parallel circular rings) that are stable but not absolute minimizers and still can be observed experimentally.} On the other hand, no weak solution exists for $\lambda>\lambda^\star$ (under various definitions of weak solution). The functions $u_\lambda$, which are bounded in $\overline\Omega$ for $\lambda<\lambda^\star$ and hence regular, form an increasing sequence in $\lambda$. They converge, as $\lambda\uparrow\lambda^\star$, towards an $L^1(\Omega)$ distributional {\it stable} solution $u^\star$ of \eqref{eq:lambda} for $\lambda=\lambda^\star$. It is called {\it the extremal solution} of the problem and, depending on the nonlinearity~$f$ and the domain~$\Omega$, it may be either regular or singular. 

\begin{remark}\label{monsters}
Let us explain, at this point, why in previous paragraphs we have made reference to stable {\it ``energy''} solutions, and also to $W^{1,2}$ solutions, as well as its relation to extremal solutions $u^\star$.

By energy solution we mean a weak (or, equivalently, distributional) solution $u\in W^{1,2}(\Omega)$ which has finite energy $\mathcal E(u)$, where $\mathcal{E}$ is defined in \eqref{energy}. A larger class considered in \cite{CFRS}, which is the most appropriate one when $f\geq 0$, is that of stable weak solutions belonging to $W^{1,2}(\Omega)$. By Proposition~4.2 and Corollary~4.3 in \cite{CFRS}, when $f$ is nonnegative, nondecreasing, and convex (as in the Gel'fand problems), every stable $W^{1,2}(\Omega)$ weak solution can be approximated by regular stable solutions of semilinear equations with nonlinearities in the same class.\footnote{This extends the approximation result \cite[Section 3.2.2]{Dup}) for $W^{1,2}_0(\Omega)$ solutions vanishing on the boundary. The results use a key technique introduced in Brezis, Cazenave, Martel, and Ramiandrisoa~\cite{BCMR}.} Therefore, once a priori estimates are proved for regular stable solutions (with such nonlinearities), they pass to the limit and establish the same estimates now for every  stable $W^{1,2}(\Omega)$ weak solution. 

The same fact also holds for every extremal solution $u^\star$, since it can be approximated (trivially, by definition) by the stable regular solutions $u_\lambda$ for $\lambda < \lambda^\star$ ---even if one does not know a priori if $u^\star \in W^{1,2}(\Omega)$. This fact (i.e., $u^\star \in W^{1,2}(\Omega)$ in every dimension and domain) was another open question posed by Brezis and V\' azquez~\cite[Open problem 1]{BV} in the nineties and is one of the main results of~\cite{CFRS}; see Subsection~\ref{subsec-bdry-results} here below.

Instead, there are {\it singular} distributional solutions $u\in L^1(\Omega)$ (having, thus, $\Delta u\in L^1(\Omega))$, but with $u\notin W^{1,2}(\Omega)$, which are stable in as low dimensions as any $n\geq 3$. These solutions are ``strange'' objects, in the sense that cannot be approximated by regular stable solutions ---otherwise they would be locally bounded for $n\leq 9$, as a consequence of Theorem~\ref{thm:0} below. In particular, they cannot be extremal solutions. A simple example of such singular distributional solution is $U(x):= |x|^{2-n+\varepsilon}-1$ for $n\geq 3$ and $\varepsilon \in (0,1)$ small enough. Indeed, $U$~solves the semilinear equation with $f(u)= \varepsilon (n-2-\varepsilon) (1+u)^{(n-\varepsilon)/(n-2-\varepsilon)}$, a nonnegative, nondecreasing, and convex nonlinearity. Since $f'(U)= \varepsilon (n-\varepsilon)/|x|^2$, a comparison with the Hardy constant ---as done right after \eqref{logsoln}--- shows that $U$ is a stable distributional solution if $\varepsilon>0$ is small enough.
\end{remark}

Back to the Gel'fand problem and its extremal solution, in the seventies the seminal paper of Crandall and Rabinowitz~\cite{CR} established the boundedness of $u^\star$ when $n \leq 9$ and $f$ is either the exponential $f(u)=e^u$ or a power $f(u)=(1+u)^p$ with $p>1$.\footnote{Previously, these same nonlinearities had been studied in all detail in the radial case by Joseph and Lundgren~\cite{JL}.} In the mid-nineties Brezis asked for a similar result in the larger class of nonlinearities \eqref{nonl-G} (see, e.g., \cite{Brezis}), a question that gave rise to the following works, among others, cited in chronological order.

Nedev~\cite{Ned00} proved the regularity of $u^\star$ for $n\leq 3$. In the radial case,
Cabr\'e and Capella~\cite{CC} established the boundedness
of every stable radial $W^{1,2}(B_1)$  solution to \eqref{eq:PDE} in the unit ball
whenever $n\leq 9$, for every nonlinearity $f$.
Back to the general nonradial case, an interior $L^\infty$ a priori bound for stable regular solutions holding  for all nonlinearities was established for $n\leq 4$ by Cabr\'e~\cite{C10} in 2010. The author gave a different proof of this result more recently, in~\cite{C19}. This led to the regularity of $u^\star$ up to the boundary of convex domains when $n\leq 4$. The convexity assumption of the domain was nicely removed by Villegas in~\cite{Vil13}.

The optimal dimension $n\leq 9$ was finally reached in 2019, by Cabr\'e, Figalli, Ros-Oton, and Serra~\cite{CFRS}, assuming $f\ge 0$ for interior regularity, and $f\ge 0$, $f'\ge 0$, and $f''\ge 0$ for boundary regularity. More recently, \cite{C22quant} has provided quantitative proofs of two results of~\cite{CFRS} proved there by contradiction-compactness arguments.

The book~\cite{Dup} contains detailed presentations of these developments, except for those of \cite{CFRS} and \cite{C22quant}, which are more recent. The survey \cite{C17}, which is a shorter presentation of the topic, contains also most of the proofs of the above mentioned results ---the ones prior to \cite{CFRS}.

\subsection{The interior result}\label{subsec-interior-results}
\
\smallskip

The following is the interior regularity result of \cite{CFRS}. It provides an interior H\"older a priori bound for stable solutions  when $n\leq 9$. It also establishes, in every dimension, an interior $W^{1,2}$ estimate (indeed a better $W^{1,2+\gamma}$ bound). All quantities are controlled in terms of only the $L^1$ norm of the solution. Regarding the nonlinearity, all what we need is $f\geq 0$. 

Notice that the following estimates look like the typical ones for linear equations. This is a remarkable consequence of the stability of the solution, since our equations are nonlinear.

Here, and throughout the paper, by ``dimensional constant'' we mean a constant that depends only on $n$. 

\begin{zero}
\label{thm:0}\hspace{-1mm}(\cite[Theorem 1.2]{CFRS}) {\it
Let $u\in C^\infty(\overline B_1)$ be a stable solution of $-\Delta u=f(u)$ in $B_1\subset \R^n$, for some nonnegative function $f\in C^{1}(\R)$.

Then,
\begin{equation}
\label{eq:W12 L1 int}
\|\nabla u\|_{L^{2+\gamma}({B}_{1/2})} \le C\|u\|_{L^1(B_1)}
\end{equation}
for some dimensional constants $\gamma>0$ and $C$.
In addition,
\begin{equation}
\label{eq:Ca L1 int}
 \|u\| _{C^\alpha(\overline{B}_{1/2})}\leq C\|u\|_{L^1(B_1)} \quad \text{ if } n \leq 9,
 \end{equation}
where $\alpha>0$ and $C$ are dimensional constants.
}\end{zero}

The theorem is stated as an a priori bound for smooth stable solutions. From it one may deduce, by the approximation method mentioned in Remark~\ref{monsters},\footnote{Here one must approximate $f$, in addition,  by smooth nonlinearities, since we assume $u\in C^\infty$ in Theorem~\ref{thm:0} to simplify the exposition. This was not necessary in \cite{CFRS}, which only assumed $u\in C^2$.} interior regularity for $W^{1,2}$ stable weak solutions to \eqref{eq:PDE} whenever $f$ is nonnegative, nondecreasing, and convex. The same fact holds for the extremal solution~$u^\star$ of~\eqref{eq:lambda} when $f$ satisfies~\eqref{nonl-G}.\footnote{This is the case since, as mentioned before, $u^\star$  can be approximated (trivially, by definition) by the stable regular solutions $u_\lambda$ for $\lambda < \lambda^\star$ ---even if one does not know a priori if $u^\star \in W^{1,2}(\Omega)$.}

For $n\geq 10$ (and still $f\geq 0$), \cite{CFRS} proved that stable solutions belong, in the interior of the domain, to certain Morrey spaces. This was an almost optimal result that has been improved by Peng, Zhang, and Zhou \cite{PZZ1}, reaching optimality. For instance, for $n=10$, \cite{PZZ1} shows that stable solutions are BMO in the interior, a result also found, independently, by Figalli and Mayboroda (personal communication). See also the forthcoming work of Figalli and Franceschini cited in 1.3.4 of Subsection~\ref{subsec-open-related}.

Also after \cite{CFRS}, another interesting result by Peng, Zhang, and Zhou~\cite{PZZ2} proves that, for $n=5$, the hypothesis $f\geq 0$ is not needed to conclude interior H\"older regularity. The H\"older control, however, is given by the $W^{1,2}$ norm of the solution instead of the $L^1$ norm. This establishes in dimension~5 the previous analogue result of Cabr\'e~\cite{C10,C19} for $n\leq 4$. More details on this are given below in Subsection~\ref{subsec-open-related}, Open problem 1.3.1.  

 \begin{remark}\label{fbddbelow} 
The quantitative proof of Theorem~\ref{thm:0} allows for an easy extension of the result to the case $f\geq -K$ for some nonnegative constant $K$. The resulting estimates in this case are those of Theorem~\ref{thm:0} with the right-hand sides replaced by $C(\|u\|_{L^1(B_1)}+K)$, where $C$ is still a dimensional constant. The necessary changes in the proofs will be explained in Section~\ref{sect:fbddbelow}. 
 
Here, we will follow the quantitative proofs of \cite{C22quant}, since it is not obvious how to obtain the result for $f\geq -K$ by extending the proof in \cite{CFRS} ---due to the compactness argument used in that paper. However, this has been accomplished recently by Fa Peng \cite{P}, who has weakened the hypothesis  $f\geq 0$ in \cite{CFRS} to allow $f\geq A\min(0,t)-K$ for positive constants $A$ and~$K$. This improves our requirement $f\geq -K$. He needs, however, to replace the~$L^1$ norm of $u$ on the right-hand side of~\eqref{eq:Ca L1 int} by the $W^{1,2}$ norm of $u$.
\end{remark}
  
\subsection{The boundary result}\label{subsec-bdry-results}
\
\smallskip

Let us now turn into boundary regularity.
It is well known that the interior bound of Theorem~\ref{thm:0} and the moving planes method yield an $L^\infty$ bound up to the boundary when the domain~$\Omega$ is convex and we deal with vanishing Dirichlet boundary conditions; see \cite{C10} and Corollary~1.4 of \cite{CFRS}. Since this procedure does not work for nonconvex domains, \cite{CFRS} undertook the study of boundary regularity in more general domains and proved the following result.

\begin{seven}
\label{thm:7}{\it
Let  $\Omega\subset \R^n$ be a bounded domain of class $C^3$ and let $u\in C^0(\overline\Omega) \cap C^2(\Omega)$ be a stable solution of 
\begin{equation}\label{BVP}
\left\{
\begin{array}{cl}
-\Delta u=f(u) & \text{in }\Omega\\
u=0 & \text{on }\partial\Omega.
\end{array}
\right.
\end{equation}
Assume that $f: \R \to \R$ is nonnegative, nondecreasing, and convex.

Then,
\begin{equation}\label{eq:w12}
\|\nabla u\|_{L^{2+\gamma}(\Omega)}\le C_\Omega \,\|u\|_{L^1(\Omega)} \quad \text{ for every } n\geq 1,
\end{equation}
and
\begin{equation*}
 \|u\| _{C^\alpha(\overline\Omega)} \leq C_\Omega\, \|u\|_{L^1(\Omega)} \quad \text{ for } n\leq 9,
\end{equation*}
where $\gamma>0$ and $\alpha>0$ are dimensional constants, while $C_\Omega$ depends only on $\Omega$.}
\end{seven}

Erneta~\cite{Ern1, Ern2, Ern3} has improved the previous result to allow for~$C^{1,1}$ domains.
 
From the a priori estimates of the theorem, one deduces the corresponding regularity results for the extremal solution $u^\star$ to \eqref{eq:lambda}. This is accomplished by using the above estimates for the family of regular stable solutions $u_\lambda$, where $\lambda<\lambda^\star$. 

The $W^{1,2}$ regularity result \eqref{eq:w12} solved~\cite[Open problem 1]{BV}, a question posed by Brezis and V\'azquez. Previously, it had been proven, under the stronger hypotheses \eqref{nonl-G} on the nonlinearity, for $n\leq 5$ by Nedev~\cite{Ned00} and for $n=6$ by Villegas~\cite{Vil13}. 

To prove the estimates of the theorem, one first covers $\partial\Omega$ by small enough balls so that $\partial\Omega$ is  almost flat (in the sense of \cite{CFRS}) inside each of the balls. Thus, it is sufficient to prove a local-boundary result\footnote{See the statement of Theorem~\ref{thm:0bdry} to understand the expression ``local-boundary result''. This terminology reflects that, in the case of half-balls, one proves a bound in a concentric half-ball smaller than the starting half-ball. They share part of their flat boundaries, which is where the vanishing boundary datum is assumed.\label{foot5}} for almost flat boundaries. 
To simplify the exposition, in the current paper we place ourselves in the simpler setting of a flat boundary ---which, anyway, contains the essential ideas in the proofs of~\cite{CFRS, Ern2, Ern3} for non-flat boundaries. Notice here that Erneta~\cite{Ern2, Ern3} has extended the quantitative proof from~\cite{C22quant} to almost flat boundaries. This allows him to avoid the proof of Theorem~6.1 in \cite{CFRS}, which is a delicate blow-up and compactness argument used together with a Liouville theorem in the half-space. It was needed there since the contradiction-compactness argument (Step~2 of the proof of Proposition~6.3 in \cite{CFRS}) required the boundary to be flat. The quantitative proof also allows Erneta to avoid a delicate (and strong) compactness result (Theorem~4.1 in \cite{CFRS}) needed to treat boundary regularity with the approach of~\cite{CFRS}.

Having two different proofs of the regularity result has been already of interest in the context of other equations. Indeed, the compactness proof of \cite{CFRS} has been extended in~\cite{CMS}, for interior regularity,  to semilinear equations involving the $p$-Laplacian. Instead, the quantitative proofs of~\cite{C22quant} have been extended to stable solutions for operators with variable coefficients by Erneta~\cite{Ern1, Ern2, Ern3}.  As mentioned above, this has significantly simplified the boundary regularity proof of~ \cite{CFRS} even for the Laplacian.

We introduce the notation
$$
\R^n_+=\{x\in\R^n\, :\, x_n>0\}, \quad B_\rho^+=\R^n_+\cap B_\rho, \quad\text{and}\quad  \partial^0 B^+_\rho = \{x_n=0\}\cap \partial B^+_\rho,
$$
and state the result of \cite{CFRS} in the case of flat boundaries, which is the one we will fully prove in this expository article.

\begin{zerobdry}
\label{thm:0bdry} \hspace{-1mm}(\cite[Propositions 5.2 and 5.5, and Theorem 6.1]{CFRS}) {\it 
Let $u\in C^\infty(\overline{B^+_1})$ be a nonnegative stable solution of 
$-\Delta u=f(u)$ in $B^+_1\subset\R^n$, with $u=0$ on $\partial^0 B^+_1$.
Assume that $f\in C^1(\R)$ is nonnegative, nondecreasing, and convex.
 
Then,
\begin{equation} 
\label{higher-bdry}
 \|\nabla u\|_{L^{2+\gamma}(B^+_{1/2})}  \le C \|u\|_{L^{1}(B^+_{1})}
\end{equation}
for some dimensional constants $\gamma>0$ and $C$.
In addition,
\begin{equation}
\label{holder-bdry}
\|u\| _{C^\alpha (\overline{B^+_{1/2}})}\leq C\|u\|_{L^1(B^+_1)}  \quad \text{ if } n \leq 9,
\end{equation}
where $\alpha>0$ and $C$ are dimensional constants.
}\end{zerobdry}

It is important here to assume the solution $u$ to be nonnegative ---an assumption that will hold when applying the theorem to the global problem \eqref{BVP}, since $f\geq0$. Note also that the result requires more assumptions on the nonlinearity than the interior theorem.

\subsection{The three ingredients for the interior estimate}\label{subsec-interior}
\
\smallskip

We first note that, by approximation, the stability inequality \eqref{stabilityLip},
$$
\int_{\Omega} f'(u) \xi^2\,dx\leq \int_{\Omega} |\nabla \xi|^2\,dx,
$$ 
holds not only for $C^1(\overline\Omega)$ functions $\xi$ vanishing on $\partial\Omega$, but also for all Lipschitz functions $\xi$ in $\overline\Omega$ which vanish on $\partial\Omega$. Hence, we can consider a test function of the form $\xi=\mathbf{c}\eta$,
where $\mathbf{c}\in W^{2,\infty}(\Omega)$, $\eta$ is a Lipschitz function in $\overline\Omega$, and $\mathbf{c}\eta$ vanishes on~$\partial\Omega$. 
Then, 
since
\begin{equation}\label{eq:parts}
\int_{\Omega} |\nabla \xi|^2\,dx = \int_{\Omega} \bigl(  |\nabla \mathbf{c}|^2 \eta^2 + \mathbf{c}\nabla\mathbf{c} \nabla \eta^2+  \mathbf{c}^2 |\nabla \eta|^2\bigr) \, dx,
\end{equation}
integrating by parts the second term of the last integral we see that 
\begin{equation}\label{eq:07}
\int_{\Omega} \bigl( \Delta \mathbf{c}+f'(u)\mathbf{c}\bigr) \mathbf{c}\,\eta^{2}\, dx \leq 
\int_{\Omega} \mathbf{c}^{2}|\nabla \eta|^{2} \, dx.
\end{equation}

The key point here is that, if we take $\eta$ vanishing on the boundary, then we have the freedom to make choices for the function $\mathbf{c}$ having arbitrary values on $\partial\Omega$.
Note also that expression \eqref{eq:07} brings the linearized operator $\Delta+f'(u)$, acting on $\mathbf{c}$, into play. The rest of the paper will make three different crucial choices for the function~$\mathbf{c}$: \eqref{test1} and \eqref{test2} for interior regularity and \eqref{test3} for flat-boundary regularity.

The proof of H\"older regularity is based on putting together three estimates. In the interior case, the first one, and where the restriction $n\leq 9$ appears, is the following. Here $f\geq 0$ is not needed. 

Throughout the paper we will use the notation
\begin{equation}\label{notation r}
r=|x| \qquad\text{and}\qquad u_r=\frac{x}{|x|}\cdot \nabla u.
\end{equation}

\begin{one}\label{prop:1} \hspace{-1mm}(\cite[Lemma 2.1]{CFRS}) {\it
Let $u\in C^\infty (\overline B_1)$ be a stable solution of $-\Delta u=f(u)$ in $B_1\subset \R^n$, for some function $f\in C^{1}(\R)$.

If $3 \leq n\leq 9$, then 
\begin{equation}\label{ineq1}
\int_{B_\rho}r^{2-n} u_{r}^2\,dx \leq C\rho^{2-n}\int_{B_{3\rho/2}\setminus B_\rho}|\nabla u|^2\,dx
\end{equation}
for all $\rho< 2/3$, where $C$ is a dimensional constant.
}\end{one}

This result will be proven using \eqref{eq:07} with 
\begin{equation}\label{test1}
\mathbf{c}= x\cdot\nabla u = ru_r \quad\text{ and } \quad \eta=r^{(2-n)/2}\zeta,
\end{equation}
where $\zeta$ is a cut-off function. Recall that $r=|x|$.

Note that the proposition requires $n\geq 3$. However, adding superfluous independent variables to the solution, as in \cite{CFRS}, we will see that it can be used to establish Theorem~\ref{thm:0} also in dimensions one and two.

The following remark, that can be avoided in a first reading, motivates the second and third ingredients presented in this section.

\begin{remark}\label{rem:hole}
If the right-hand side of estimate \eqref{ineq1} had $u_r^2$ as integrand, instead of $|\nabla u|^2$, then H\"older regularity of $u$ would be a simple consequence of it, by the ``hole filling technique''. This will be shown in detail later in the paper. Instead, in Appendix~\ref{app:example} we will show that in the actual form \eqref{ineq1}, Proposition~\ref{prop:1} (even when used with radial derivatives taken with respect to any other base point, and not only the origin) is not enough to deduce an interior $L^\infty$ bound for~$u$. This is true even assuming that $u$ is an unstable\footnote{Instead, the $L^\infty$ bound holds among stable solutions, by Theorem~\ref{thm:0}, when $n\leq 9$.}  solution to \eqref{eq:PDE} for some nonnegative nonlinearity $f$ and $n\leq 9$. We emphasize this fact since we ---and some other experts--- initially thought  that the estimate of Proposition~\ref{prop:1}, used with radial derivatives taken with respect to any base point and not only the origin, could suffice. This is not the case indeed, and therefore, still having Proposition~\ref{prop:1} at hand, to get H\"older regularity one needs to exploit again the stability of the solution~$u$.

Notice that the quantities in \eqref{ineq1} are rescale invariant. That is, they are adimensional quantities ---like the $L^\infty$ norm of $u$ in $B_\rho$ is, as well. At the same time, the class of all stable solutions for some nonnegative nonlinearity is also rescale invariant. Hence, we may assume $\rho=1$ in \eqref{ineq1} ---in which case $\rho^{2-n}$ is both bounded from above and below in the annulus $B_{3/2}\setminus B_1$.
In \cite{CFRS} the proof of H\"older regularity continued from \eqref{ineq1} by controlling $|\nabla u|$ in $L^2$ by $u_r$ in $L^2$, under a {\it doubling assumption} on $|\nabla u|^2$. More precisely, \cite[Lemma 3.1]{CFRS} proved that
\begin{equation}\label{step3CFRS}
\Vert \nabla u\Vert^2_{L^{2}({B_{3/2})}}\leq C_\delta \Vert u_{r}\Vert^2_{L^{2}({B_{3/2}\setminus B_{1})}} \quad\text{ whenever }
\Vert \nabla u\Vert^2_{L^{2}({B_{1})}}\geq \delta \Vert \nabla u\Vert^2_{L^{2}({B_{2})}}
\end{equation}
and $u$ is a stable solution in $B_2$. This was accomplished through a contradiction-compactness argument. Compactness in the $W^{1,2}$ norm followed from the higher integrability bound \eqref{eq:W12 L1 int} (obtained in $B_{3/2}$ instead of $B_{1/2}$). The contradiction came from analyzing the validity of \eqref{step3CFRS} in the limiting case $u_r\equiv 0$. This is simple, since if $u_r\equiv 0$, then $u$ is $0$-homogeneous and, thus, $u$ is a superharmonic function for the Laplace-Beltrami operator on the unit sphere (recall that $f\geq 0$). This yields $u$ to be constant, and hence $|\nabla u|\equiv 0$.

In the current paper, we will avoid this contradiction-compactness proof and present a more recent quantitative proof from \cite{C22quant}. It allows to quantify the H\"older exponent~$\alpha$ of Theorem~\ref{thm:0}.
\end{remark}

Proposition \ref{prop:1} above is the first ingredient towards the interior H\"older regularity. The second one, already contained in \cite{CFRS}, uses the choice 
\begin{equation}\label{test2}
\mathbf{c} = |\nabla u|
\end{equation}
in \eqref{eq:07}, which was first introduced by Sternberg and Zumbrun \cite{SZ} while studying the Allen-Cahn equation. It will lead, in every dimension, to certain $L^2$ and $L^1$ second derivative estimates. From them, we will deduce the $L^{2}$ estimate for $\nabla u$ of the following proposition, which holds in all dimensions.  

Before stating it, let us mention that the two key test functions \eqref{test1} and \eqref{test2} used in the stability inequality correspond, respectively, to the following perturbations of the solution $u$ (one in the radial direction, the other in the normal direction to the level sets):
$$
u\left( x+ \ep r^{(2-n)/2}\zeta (x) x \right)
\quad\text{ and }\quad
u\Big( x+ \ep \eta (x)\frac{\nabla u (x)}{|\nabla u(x)|}  \Big),
$$
where $\zeta$ and $\eta$ are cut-off functions. Indeed, making the derivative of these expressions with respect to $\varepsilon$, at $\varepsilon =0$, we obtain the test functions $\xi=\mathbf{c} \eta$ in  \eqref{test1} and \eqref{test2}.

\begin{two}\label{prop:2}  \hspace{-1mm}(\cite[Proposition 2.5]{CFRS}) {\it
Let $u\in C^\infty (\overline B_1)$ be a stable solution of $-\Delta u=f(u)$ in $B_1\subset \R^n$, for some nonnegative function $f\in C^{1}(\R)$.

Then,
\begin{equation}\label{ineq2}
\|\nabla u\|_{L^{2}(B_{1/2})}  \le C \|u\|_{L^{1}(B_{1})}
\end{equation}
for some dimensional constant $C$.
}\end{two}

Proposition \ref{prop:2} provides the second step towards the H\"older regularity proof. Indeed, one starts from estimate \eqref{ineq1} of Proposition~\ref{prop:1} with, for instance,  $\rho=5/8$. Now, one covers the annulus in the right-hand side of \eqref{ineq1} by a dimensional number of balls with appropriate small radius. Next, in each of these balls one uses \eqref{ineq2}, properly rescaled, with $u$ replaced by the stable solution $u-t$, a solution of $-\Delta v= f(v+t)$, where $t$ is a constant to be chosen later. Adding the right-hand sides of  \eqref{ineq2} in all the balls, one concludes that the integral of $r^{2-n} u_{r}^2$ in $B_{1/2}$ can be controlled by the $L^1$ norm of $u-t$ in the annulus $B_1\setminus B_{1/2}$.

To control this last quantity, we will use the final third ingredient, namely, estimate \eqref{new-int-ur} in an annulus of Proposition~\ref{prop:3} below, first established in~\cite{C22quant}.\footnote{Instead, we will not use estimate \eqref{new-int-ur-ball} in balls, which will follow easily from \eqref{new-int-ur} and is stated here only for completeness.}  We replace the estimate \eqref{step3CFRS} used in \cite{CFRS} by the new bound \eqref{new-int-ur} in annuli. 

\begin{three}\label{prop:3}\hspace{-1mm}(\cite[Theorem 1.4]{C22quant}) {\it
Let $u\in C^\infty (\overline B_1)$ be superharmonic in the ball $B_1\subset \R^n$. 

Then, 
\begin{equation}\label{new-int-ur} 
\inf_{t\in\R} \Vert u-t\Vert_{L^{1}({B_{1}\setminus B_{1/2})}}\leq C\Vert u_{r}\Vert_{L^{1}({B_{1}\setminus B_{1/2})}}
\end{equation}
and
\begin{equation}\label{new-int-ur-ball} 
\inf_{t\in\R} \Vert u-t\Vert_{L^{1}({B_{1})}}\leq C\Vert u_{r}\Vert_{L^{1}({B_{1})}}
\end{equation}
for some dimensional constant $C$.
}\end{three}

As pointed out in \cite{C22quant}, the estimates of the proposition can be equivalently stated as
$$
\Vert u-\textstyle{\ave_{E}}u\Vert_{L^{1}(E)}\leq C\Vert u_{r}\Vert_{L^{1}(E)}
$$
with $E=B_{1}\setminus B_{1/2}$ and $E=B_{1}$, where $\ave$ denotes the average in a set.

Notice that with respect to \eqref{step3CFRS}, in Proposition~\ref{prop:3} we have replaced $L^2$ norms by $L^1$ norms, as well as $|\nabla u|$ by $u$ minus constants. In addition, here we do not require a doubling assumption, unlike \eqref{step3CFRS}. 

One could wonder about the validity of \eqref{new-int-ur} and  \eqref{new-int-ur-ball} for all functions~$u$, and not only superharmonic ones. For $n\geq2$, both inequalities would be false. Indeed, this can be seen with the family of smooth functions 
$$
u^\ep(x):=\varphi(|x|/\ep)\, x_n/|x|, 
$$
where $\ep\in (0,1)$, $\varphi$ is smooth and nonnegative, vanishes in $(0,1/4)$, and is identically~1 in $(1/2,\infty)$. Then, $u^\ep_r$ has support in $B_{\ep/2}\setminus B_{\ep/4}$ and its $L^1(B_1)$ norm is bounded above by $C\ep^{n-1}$. Instead, to control the $L^1(B_1\setminus B_{1/2})$ norm of $u^\ep-\textstyle{\ave_{B_1\setminus B_{1/2}}} u$ (or, similarly, of $u^\ep-\textstyle{\ave_{B_1}} u$) by below, if $\textstyle{\ave_{B_1\setminus B_{1/2}}} u\leq0$ (the case $\textstyle{\ave_{B_1\setminus B_{1/2}}} u\geq 0$ is treated similarly) we may integrate only in $\{1/2<|x|<1, x_n>0\}$. In this set $|u^\ep-\textstyle{\ave_{B_1\setminus B_{1/2}}} u|=u^\ep -\textstyle{\ave_{B_1\setminus B_{1/2}}} u\geq u^\ep=x_n/|x|$, and thus the $L^1(B_1\setminus B_{1/2})$ norm is larger than a positive constant independent of $\ep$.

\begin{remark}\label{vanish-on-bdry}
Instead, it is easy to check (this will be done in \cite{CSi}) that the radial Poincar\'e inequalities \eqref{new-int-ur} and \eqref{new-int-ur-ball} hold among all functions that vanish on $\partial B_1$ (in this case, there is no need to subtract a constant to the function $u$ for the inequalities to hold). They also hold for general harmonic functions, not necessarily vanishing on  $\partial B_1$.
\end{remark}

\begin{remark}\label{Lp-version}
In the forthcoming paper \cite{CSi}, Cabr\' e and Saari determine the range of exponents $p\in [1,\infty)$ for which the radial Poincar\'e inequalities \eqref{new-int-ur} and \eqref{new-int-ur-ball} hold, for superharmonic functions, with $L^1$ replaced by $L^p$ in both their left and right-hand sides. It turns out that inequality \eqref{new-int-ur} in an annulus holds if and only~if
$$
\text{either } n\leq 3, \text{ or } n\geq 4 \text{ and } 1\leq p <(n-1)/(n-3).
$$
Instead, \eqref{new-int-ur-ball} in a ball holds if and only if
$$
\text{either } n\leq 3, \text{ or } n\geq 4 \text{ and either } 1\leq p <(n-1)/(n-3) \text{ or } p>n.
$$

On the other hand, the statements of Remark~\ref{vanish-on-bdry} hold not only for $p=1$, but also for every $p\in [1,\infty)$ (see \cite{CSi}).
\end{remark}

Putting together Propositions \ref{prop:1}, \ref{prop:2} (after a covering  of the annulus by balls), and \ref{prop:3}, one concludes that the integral of $r^{2-n} u_{r}^2$ in $B_{1/2}$ is controlled by its integral on the annulus $B_1\setminus B_{1/2}$. This estimate and the ``hole filling technique" ---which refers to the hole of the annulus--- , used at all scales and iterated, will easily lead to the H\"older regularity of the stable solution $u$.

The proof of Proposition \ref{prop:3} uses a harmonic replacement together with the maximum principle ---here it will be crucial that we work with $L^1$ instead of $L^2$ norms. In this way, one reduces the problem to the case when $u$ is harmonic. But when $\Delta u=0$, we confront a Neumann problem for the homogeneous Laplace equation (since $u_r$, restricted to any sphere, may be thought as the Neumann data), for which the estimates of Proposition~\ref{prop:3} are not surprising and, in fact, easy to prove.

\subsection{The three ingredients for the flat boundary estimate}\label{subsec-bdry}
\
\smallskip

Let us now switch to estimates up to $\partial\R^n_+$, that is, regularity up to flat boundaries. Theorem~\ref{thm:0bdry} will follow from three results, which are the analogues of the interior ones described above. The first one is the following.

\begin{four}\label{prop:1bdry}  \hspace{-1mm}(\cite[(6.6)]{CFRS}) {\it
Let $u\in C^\infty(\overline{B^+_1})$ be a stable solution of $-\Delta u=f(u)$ in $B^+_1\subset\R^n$, with $u=0$ on $\partial^0 B^+_1$, for some nonlinearity $f\in C^1(\R)$.
 
If $3 \leq n\leq 9$, then 
\begin{equation*}
\int_{B^+_\rho}r^{2-n} u_{r}^2\,dx \leq C_\lambda\, \rho^{2-n}\int_{B^+_{\lambda\rho}\setminus B^+_{\rho}}|\nabla u|^2\,dx
\end{equation*}
for all $\lambda>1$ and $\rho< 1/\lambda$, where $C_\lambda$ is a constant depending only on $n$ and $\lambda$.
}\end{four}

For the proof of this result, as well as of the following one, the stability condition is used with a test function $\xi=\mathbf{c}\eta$ where it is important that $\eta$ does not vanish on~$\partial^0 B^+_1$. Thus, we will need to take $\mathbf{c}$ vanishing on the flat boundary. For the previous proposition, one makes the same  choice $\mathbf{c}=x\cdot\nabla u$ as in the interior case, which does vanish on the flat boundary.  

Instead, for the following result one cannot use the choice $\mathbf{c}=|\nabla u|$ (the one made in the analogue interior result), since this function does not vanish on $\partial^0 B^+_1$. In \cite{CFRS} we found that
\begin{equation}\label{test3}
\mathbf{c}=|\nabla u|- u_{x_n},
\end{equation}
which vanishes on $\partial^0 B^+_1$, is a suitable replacement. This allowed to establish the following $W^{1,2+\gamma}$ estimate. Note that while the previous proposition required no sign condition on $f$ and $u$, we now require $f$ to be nonnegative and nondecreasing ---but not necessarily convex--- and $u$ to be nonnegative.

\begin{five}\label{prop:2bdry}  \hspace{-1mm}(\cite[Propositions 5.2 and 5.5]{CFRS}) {\it
Let $u\in C^\infty(\overline{B^+_1})$ be a nonnegative stable solution of $-\Delta u=f(u)$ in $B^+_1\subset\R^n$, with $u=0$ on $\partial^0 B^+_1$.
Assume that $f\in C^1(\R)$ is nonnegative and nondecreasing.

Then,
\begin{equation*}
 \|\nabla u\|_{L^{2+\gamma}(B^+_{1/4})}  \le C \| u\|_{L^{1}(B^+_{1})}, 
\end{equation*}
for some dimensional constants $\gamma>0$ and $C$.
}\end{five}

To use this result as second ingredient towards the H\"older estimate, it would suffice to have the exponent $2+\gamma$ to be 2 (as it was the case in the interior regularity). However, as a novelty with respect to the interior theory, Proposition~\ref{prop:2bdry} will be used also to establish the control of $u$ by $u_r$ (that is, Proposition~\ref{prop:3bdry}). For this, having the integrability exponent $2+\gamma$ larger than 2 will be useful.

The last ingredient, on control up to the boundary of the solution $u$ by its radial derivative $u_r$, is where \cite{CFRS} used a contradiction-compactness argument (Step~2 of the proof of Proposition 6.3 in that paper) in addition to a new delicate result. To explain the new result, it is interesting to first look at the extreme case $u_r\equiv 0$ ---this can be seen as a first validity test for a boundary version of estimate \eqref{step3CFRS}. In fact, such case had to be analyzed also in the contradiction-compactness proof of \cite{CFRS} ---as well as in the interior case, as described above in Remark~\ref{rem:hole}. In contrast with the interior case, ``bad news'' arise here. Indeed, there exist nonzero nonnegative superharmonic functions $u$ in $B_1^+$ which vanish on $\partial^0  B_{1}^+$ and are 
0-homogeneous (i.e., for which $u_r\equiv 0$) ---obviously, they are not smooth at the origin, but belong to $W^{1,2}(B_{1}^+)$ when $n\ge 3$. An example is given by the function 
$$
u(x)= \frac{x_n}{|x|}.
$$ 
This is the key reason behind the fact that the new boundary result, Proposition~\ref{prop:3bdry} below, does not hold in the class of superharmonic functions which are smooth up to the flat boundary, as we will show in all detail in Remark~\ref{rk:not-superh}.

As a consequence, the contradiction in \cite{CFRS} could not come from the limiting function being superharmonic ---as in the interior case---, but from a stronger property. Namely, Theorem~4.1 of \cite{CFRS} established that the class of $W^{1,2}$ functions~$u$ for which there exists a nonnegative, nondecreasing, and convex nonlinearity~$f$ (with $f$ blowing-up, perhaps, at the value $\sup u$)
 such that $u$ is a stable solution of $-\Delta u=f(u)$, is closed under the $L^1$-convergence of the functions~$u$  in compact sets. Note that the class does not correspond to stable solutions of only one equation, but of all equations with such nonlinearities. This is a very strong result, of interest by itself, which will not be needed in the quantitative proof from~\cite{C22quant} presented in the current paper. The work \cite{CFRS} used this compactness result to find that, in the extreme case discussed above, the limiting function $u$ would be not only superharmonic, but also a solution of an equation  $-\Delta u=g(u)$ for some nonlinearity~$g$ with the properties listed above. This gave a contradiction, since then $g(u)$ would be $0$-homogeneous (being $u_r$ identically 0), while $-\Delta u$ would be $-2$-homogeneous.

Instead, the simpler quantitative proof of boundary regularity will use estimate \eqref{introbdryradial} below, which is the main result of \cite{C22quant}.

\begin{six}\label{prop:3bdry}\hspace{-1mm}(\cite[Theorem 1.9]{C22quant}){\it
Let $u\in C^\infty(\overline{B_{1}^+})$ be a nonnegative stable solution of $-\Delta u=f(u)$ in $B_{1}^+\subset\R^n$, with $u=0$ on $\partial^0  B_{1}^+$. Assume that $f\in C^1(\R)$ is nonnegative, nondecreasing, and convex.
 
Then,
\begin{equation}\label{introbdryradial}
\| u\|_{L^1(B_{1}^+\setminus B_{1/2}^+)} \le C \|u_r\|_{L^1({B_{1}^+\setminus B_{1/2}^+)}}
\end{equation}
and
\begin{equation}\label{introbdryradial-ball}
\| u\|_{L^1(B_{1}^+)} \le C \|u_r\|_{L^1({B_{1}^+)}}
\end{equation}
for some dimensional constant $C$.
}
\end{six}

Several comments are in order.

We will show in Remark~\ref{rk:not-superh} that, unlike the interior case of Proposition~\ref{prop:3}, estimate \eqref{introbdryradial} does not hold in the larger class of nonnegative superharmonic functions in~$B_1^+$ which are smooth up to the boundary of $B_1^+$ and vanish on $\partial^0 B^+_1$. 

In Remark~\ref{rk:bdry-ngeq3} we will prove that, at least for $n\geq3$, Proposition \ref{prop:3bdry} also holds when replacing \eqref{introbdryradial}, or  \eqref{introbdryradial-ball}, by the stronger bound
$$
\| u\|_{L^1(B_{1}^+)} \le C \|u_r\|_{L^1({B_{1}^+\setminus B_{1/2}^+)}} \quad\text{ (if $n\ge 3$)}.
$$

Proving Proposition~\ref{prop:3bdry} was rather delicate and required new ideas, described below. It uses not only the semilinear equation for $u$, but also the stability condition.  Furthermore, it requires the nonlinearity to satisfy $f\ge 0$, $f'\ge 0$, and $f''\ge 0$. Note that these are the exactly same assumptions needed in the contradiction-compactness proof in \cite{CFRS}. This fact came partly as a surprise, since the quantitative proof of \cite{C22quant} and the one of  \cite{CFRS} are based on very different ingredients, as described next.

The proof of Proposition \ref{prop:3bdry} originates from the identity
\begin{equation}\label{origin:bdry}
-2 \Delta u + \Delta (x\cdot\nabla u)  = -f'(u) \, x\cdot\nabla u=-f'(u) r u_r,
\end{equation}
which is easily checked by computing $\Delta (x\cdot\nabla u)$.\footnote{This is the same identity as \eqref{basic lin} in the interior proof. It is the crucial expression leading to the requirement $n\leq 9$ for interior regularity.}  
After multiplying \eqref{origin:bdry} by an appropriate test function and integrating by parts in a half-annulus, it is easily seen that the left-hand side bounds the $L^1$ norm of $u$ by below, modulus an error which is admissible for  \eqref{introbdryradial}: the $L^1$ norm of $x\cdot \nabla u=ru_r$. The main difficulty is how to control the right-hand side of \eqref{origin:bdry} by above, in terms of only $u_r$. This is delicate, due to the presence of $f'(u)$. We will explain the details in Section~\ref{sect:boundary-L1radial}. The proof involves introducing the functions $u_\lambda(x):=u(\lambda x)$, interpreting the radial derivative $x\cdot \nabla u (\lambda x)$ as $\frac{d}{d\lambda} u_\lambda (x)$, considering a ``$u_\lambda$-version'' of \eqref{origin:bdry}, and the key idea of averaging such version  in $\lambda$ to reduce the number of derivatives falling on the solution $u$ through the proof arguments.

\subsection{Open problems and related equations}\label{subsec-open-related}
\

\smallskip
1.3.1. The assumption $f\geq 0$ is not needed in the interior $L^\infty$ estimates of \cite{C10,C19} for $n\leq 4$, nor in the  interior H\"older bound of \cite{PZZ2} for $n\leq 5$. 
It is neither needed in the bounds for $n\leq 9$ of the radial case, \cite{CC}. In the general nonradial case, it is not known if an interior $L^\infty$ estimate could hold for $6\leq n\leq 9$ without the hypothesis $f\geq 0$.

For future developments, it is of interest to keep in mind the test functions used in the stability inequality \eqref{stabilityLip}. Up to trivial cut-off functions $\zeta$, \cite{C10} used $\xi= |\nabla u| \varphi (u)\zeta$ for some appropriate function $\varphi$. Both \cite{CFRS} and the current paper use, for interior regularity, the two test functions $\xi=\mathbf{c}\eta$ given by \eqref{test1} and \eqref{test2} (this last one with $\eta$ being a cut-off). Instead, \cite{C19} and the recent \cite{PZZ2} use
$$
\xi= |\nabla u| r r^{-\beta} \zeta
$$
for some positive exponents $\beta$ ---\cite{PZZ2} takes the optimal exponent $\beta=(n-2)/2$.

\smallskip
1.3.2. Our boundary regularity results require $f\geq 0$, $f'\geq 0$, and $f''\geq 0$. Both the proof in \cite{CFRS} and the one from \cite{C22quant} in the current paper, which are completely different, turn out to need these three conditions. It is not known if these hypotheses could be relaxed.

\smallskip
1.3.3. Bruera and Cabr\'e~\cite{BC} have extended some of the techniques of \cite{CFRS} to stable solutions of MEMS problems (see, e.g.\ \cite{EGG, Lin-Yang}) to reach a regularity result in the optimal range $n\leq 6$. Here one has the equation $-\Delta u =\lambda f(u)$ for positive nonlinearities $f:[0,1)\to\R$ which blow-up at $u=1$ in such a way that $\int_0^1 f=+\infty$. In this new setting, a solution is said to be regular when $\Vert u\Vert_{L^\infty(\Omega)}< 1$.\footnote{Note that the results of \cite{CFRS} apply to this class of blow-up nonlinearities ---in fact, they appeared naturally while proving the results of  Section~4 in that paper. Thus, Corollary 4.3 of \cite{CFRS} gives H\"older regularity of stable solutions~$u$ to MEMS problems when $n\leq 9$. However, this does not prevent the solution to be singular in the new meaning, that is, to have $u=1$ at some point (and thus $-\Delta u =f(u)=+\infty$ at that point, from what we see $u\notin C^2$).} The result requires, however, a Crandall-Rabinowitz type assumption on the nonlinearity, more precisely, that $\liminf_{t\uparrow 1} \left(ff''/(f')^2\right) (t) >1$ ---or, equivalently, that $f^{-a}$ is concave near~1 for some $a>0$.

\smallskip
1.3.4. A Crandall-Rabinowitz type assumption on $f$ (a weaker one) is also used by Figalli and Franceschini~\cite{FF}, where they prove partial regularity results for stable solutions in arbitrary dimensions. This substantially improves the Morrey regularity results for $n \geq 10$ of~\cite[Theorem 1.9]{CFRS}.

\smallskip
1.3.5. The results of \cite{CFRS} have been extended by Figalli and Zhang~\cite{Fig-Zhang} to finite Morse index solutions, a larger class than stable solutions. Here the nonlinearity must be required to be supercritical, in an appropriate sense.  It would be interesting to make further progress on finite Morse index solutions, both in bounded domains and in the entire space. See also the presentation and open problems in~\cite{Wang}.

\smallskip
1.3.6. Dupaigne and Farina~\cite{Dup-Far} have used the results of \cite{CFRS} to establish Liouville theorems for entire stable solutions. A very general one, for dimensions $n\leq 10$ and all nonlinearities $f\geq 0$, applies to all stable solutions which are bounded by below. Still, \cite{Dup-Far} and \cite{Wang} raise some open questions in this topic.

\smallskip
1.3.7. The interior result of \cite{CFRS} has been extended to the $p$-Laplacian by Cabr\'e, Miraglio, and Sanch\'on~\cite{CMS}.\footnote{For general $p>1$, optimal bounds for the extremal
solution  when $f(u)=\lambda e^u$ (extending the  results of Crandall and Rabinowitz~\cite{CR} for $p=2$) were obtained by Garc\'{\i}a-Azorero, Peral, and Puel~\cite{GarPe92,GarPePu94}.} This work reaches the optimal dimension for interior regularity of stable solutions when $p>2$. It remains an open problem to accomplish the expected optimal dimension for $1<p<2$.

\smallskip
1.3.8. For the equation $(-\Delta)^su=f(u)$, with $0<s<1$,  involving the fractional Laplacian, the optimal dimension for regularity of stable solutions has only been reached when $f(u)=\lambda e^u$. This was accomplished in a remarkable paper by Ros-Oton~\cite{R-O}. The dimension found in that paper, which depends on~$s\in (0,1)$, is expected to also give the correct range of dimensions for other nonlinearities. This is however an open problem, even in the radial case. The most recent progress is the paper~\cite{CSz} by Cabr\' e and Sanz-Perela in the general nonradial case, where the techniques of \cite{CFRS} are extended for the half-Laplacian ($s=1/2$) to prove interior regularity in dimensions $n\leq 4$ for all nonnegative convex nonlinearities ---here, with $s=1/2$, one expects $n\leq 8$ to be the optimal result.

\bigskip\medskip
\centerline{\large{\sc Part I: Interior regularity}}
\addtocontents{toc}{ \vspace{.2cm}\textsc{\hspace{.2cm} Part I: Interior regularity  \vspace{.1cm}}}
\medskip

This first part of the article concerns interior regularity. Thanks to the new result of Proposition~\ref{prop:3}, which is proven in Section \ref{sect:L1byRad}, we avoid the contradiction-compactness argument of \cite{CFRS}. 

In addition, our proof of the $W^{1,2}$ estimate of Proposition \ref{prop:2} (given in Section~\ref{sect:W12}) is more elementary than that of \cite{CFRS}, which needed the higher integrability result~\eqref{eq:W12 L1 int} for $|\nabla u|$. While this last result was obtained through the Sobolev inequality and some estimates on each level set of the solution (as presented in Section~\ref{sect:higher} below, for completeness), we will instead use a new and simple interpolation inequality.

\section{The weighted $L^{2}$ estimate for radial derivatives}
\label{sect:interior-weighted}

The following is the key estimate towards the interior regularity result. In particular, it is here (and only here) where the condition $n\leq 9$ arises. It is proven by using the test function $\xi=(x\cdot \nabla u) \eta = ru_r\eta$ in the stability condition and by later taking $\eta=r^{(2-n)/2}\zeta$ for some cut-off function $\zeta$. Recall the notation \eqref{notation r} for $r$ and~$u_r$.

Throughout the paper, some subscripts denote partial derivatives, that is, $u_i=\partial_iu,$ $u_{ij}=\partial_{ij}u$, etc.

\begin{lemma}{\rm(\cite[Lemma 2.1]{CFRS})}\label{conseqestab2}
Let $u\in C^\infty (\overline B_1)$ be a stable solution of $-\Delta u=f(u)$ in $B_1\subset \R^n$,  for some function $f\in C^{1}(\R)$.

Then, 
\begin{equation}\label{eq:firstest}
 \int_{B_1} \Big( |\nabla u|^2  \big\{(n-2)\eta +2 x\cdot \nabla \eta \big\}\eta
 - 2(x\cdot \nabla u) \nabla u\cdot \nabla(\eta^2) 
 - |x\cdot \nabla u|^2 |\nabla \eta|^2 \Big)\,dx \le 0
\end{equation}
for all Lipschitz functions $\eta$ vanishing on $\partial B_1$.
As a consequence, 
\begin{equation}\label{firstest-n-2}
\begin{split}
\frac{(n-2)(10-n)}{4}\int_{B_1} & r^{2-n}  u_r^2  \zeta^2\, dx \\ 
&\hspace{-2cm} \leq \int_{B_1}  (-2) r^{3-n} |\nabla u|^2\zeta \zeta_r\, dx   
+\int_{B_1} 4 r^{3-n} u_r \zeta\,  \nabla u\cdot\nabla\zeta \, dx \\
& \hspace{-1cm}
+\int_{B_1} (2-n) r^{3-n} u_r^2\zeta \zeta_r  \, dx
+\int_{B_1} r^{4-n} u_r^2 |\nabla\zeta|^2 \, dx
\end{split}
\end{equation}
for all Lipschitz functions $\zeta$ vanishing on $\partial B_1$. 
\end{lemma}

\begin{proof}
We first establish \eqref{eq:firstest}.
We choose $\mathbf{c}(x):=x\cdot \nabla u(x)$. By a direct computation 
we have that
\begin{equation}\label{basic lin}
\Delta \mathbf{c}=x\cdot \nabla \Delta u +2\sum_{i=1}^n u_{ii}=-f'(u)\mathbf{c}+2\Delta u\qquad\text{in } B_1.
\end{equation}
Substituting this identity in \eqref{eq:07} we get\footnote{The same computation arises in the proof of the classical Pohozaev identity.}
\begin{align*}
\int_{B_1} |x\cdot \nabla u|^{2}&\left|\nabla \eta\right|^{2} \,dx \geq 
\int_{B_1}\bigl( \Delta \mathbf{c}+f'(u)\mathbf{c}\bigr) \mathbf{c}\,\eta^{2}\, dx=2\int_{B_1} (x\cdot \nabla u)\Delta u\,\eta^{2}\, dx\\
&\hspace{-.6cm} =\int_{B_1}\Big({\rm div}\bigl(2(x\cdot \nabla u)\nabla u - |\nabla u|^2x\bigr)+(n-2)|\nabla u|^2\Bigr)\eta^2\,dx\\
&\hspace{-.6cm} =\int_{B_1}\Big(-2(x\cdot \nabla u) \nabla u\cdot \nabla (\eta^2) +|
\nabla u|^2x\cdot \nabla (\eta^2)+(n-2)|\nabla u|^2\eta^2\Bigr)\,dx,
\end{align*}
which establishes \eqref{eq:firstest}.

Next, we deduce \eqref{firstest-n-2}.
Given $a\in (0,n)$,
we would like to take $\eta:=r^{-a/2}\zeta$, with~$\zeta$ Lipschitz and vanishing on $\partial B_1$, as test function in \eqref{eq:firstest}.
Since $\eta$ might not be Lipschitz at the origin, we approximate it by the Lipschitz function
\[\eta_\varepsilon:= \min\{r^{-a/2},\varepsilon^{-a/2}\}\zeta\]
for $\varepsilon\in(0,1)$, which agrees with $\eta$ in $B_1\setminus B_\varepsilon$.
We have that $\eta_\varepsilon\to\eta$ and $\nabla \eta_\varepsilon\to \nabla \eta$ a.e.\ in $B_1$ as $\varepsilon\downarrow0$.
At the same time, when $\eta$ is chosen to be $\eta_\varepsilon$, every term in \eqref{eq:firstest}  is bounded in absolute value by $Cr^{-a}|\nabla u|^2\leq \tilde Cr^{-a}\in L^1(B_1)$, since $u\in C^\infty (\overline B_1)$.
Hence, the dominated convergence theorem gives that \eqref{eq:firstest} also holds with $\eta:=r^{-a/2}\zeta$.

Now, noticing that
\begin{equation*}
x\cdot \nabla \eta=-\frac{a}2r^{-a/2}\zeta+r^{-a/2}x\cdot \nabla \zeta,\quad \nabla (\eta^2)=-ar^{-a-2}\zeta^2x+2r^{-a}\zeta \nabla \zeta,
\end{equation*}
and
\begin{equation*}
|\nabla \eta|^2=\Big|-\frac{a}2r^{-a/2-2}\zeta x+r^{-a/2}\nabla \zeta \Big|^2=\frac{a^2}4r^{-a-2}\zeta^2 + r^{-a}|\nabla \zeta|^2-ar^{-a-2}\zeta (x\cdot \nabla \zeta),
\end{equation*}
\eqref{firstest-n-2} follows from \eqref{eq:firstest} by choosing $a=n-2$.
\end{proof}

From the lemma, we easily deduce our key estimate in dimensions~$3\leq n\leq9$.

\begin{proof}[Proof of Proposition \ref{prop:1}]
Given $\rho \in (0,2/3)$, consider a Lipschitz function $\zeta$ satisfying $0\leq \zeta\leq 1$, $\zeta_{|B_{\rho}}=1$, $\zeta_{|\R^n\setminus B_{3\rho/2}}=0$,
and  $|\nabla \zeta| \leq C/\rho$. Using it in \eqref{firstest-n-2}
and noticing that $r=|x|$ is comparable to $\rho$ inside ${\rm supp}(\nabla \zeta)\subset \overline B_{3\rho/2}\setminus B_\rho$, the result follows.
\end{proof}

\section{Hessian and $W^{1,2}$ estimates}
\label{sect:W12}

To establish the $W^{1,2}$ estimate of Proposition \ref{prop:2} we use again the stability of~$u$, but choosing now another test function: $\xi=|\nabla u|\eta$ with $\eta$ a cut-off.
This choice was introduced by Sternberg and Zumbrun~\cite{SZ} in 1998 while studying the Allen-Cahn equation. It leads to the following result.

Define
\begin{equation}\label{defAAA}
\J :=  
\begin{cases} \left( \sum_{ij} u_{ij}^2  - \sum_{i} \left(\sum_{j} u_{ij} \frac{u_j}{|\nabla u|} \right)^2  \right)^{1/2} \quad  \quad &\mbox{if  } \nabla u \neq 0
\\
0 &\mbox{if  } \nabla u=0.
\end{cases}
\end{equation}

\begin{lemma}[\cite{SZ}]\label{conseqestab}
Let $u\in C^\infty (\overline B_1)$ be a stable solution of $-\Delta u=f(u)$ in $B_1\subset \R^n$,  for some function $f\in C^{1}(\R)$. 

Then, 
\begin{equation}\label{stab-real-c}
\int_{B_1}  \J^2 \eta ^2\, dx \le \int_{B_1} |\nabla u|^2 |\nabla \eta|^2\, dx
\end{equation}
for all Lipschitz functions $\eta$ vanishing on $\partial B_1$. 
\end{lemma}

\begin{proof}
We would like to take $\mathbf{c}:=|\nabla u|$ in the alternative form \eqref{eq:07} of stability.
A simple computation shows that 
\begin{equation}\label{real-c}
|\nabla u|\bigl(\Delta|\nabla u|+f'(u)|\nabla u|\bigr)=\mathcal A^2\quad \text{ in } \{\nabla u\neq0\}.
\end{equation}
This, together with \eqref{eq:07},  suggests that \eqref{stab-real-c} should hold. However a complete argument is needed, since we used integration by parts to derive \eqref{eq:07}, $\mathbf{c}=|\nabla u|$ may not be differentiable  where the gradient vanishes, and \eqref{real-c} in principle only holds in $ \{\nabla u\neq0\}$.

The following is the most elementary proof of the lemma.\footnote{See \cite{CFRS} for an alternative proof
which does not require to compute~$\Delta |\nabla u|$ and hence applies to~$C^{2}$ solutions $u$ ---it uses, however, Stampacchia's theorem (\cite[Theorem~6.19]{LiebLoss}) to ensure that $D^2u=0$ a.e.\ in $\{\nabla u=0\}$. A third proof, similar to the one we give here and that does not rely on Stampacchia's result, can be given following an argument for the $p$-Laplacian from~\cite{CMS}. It consists of using \eqref{eq:07} with $\mathbf{c}= |\nabla u| \phi (\abs{\nabla u}/\ep)$,
where $\phi\in C^\infty(\R)$ is nondecreasing, $\phi(t)=0$ for $t\leq1$, and $\phi(t)=1$ for $t\geq2$,  and letting $\ep\downarrow 0$ at the end.}
As in \cite{C10}, for $\varepsilon >0$ we consider 
\begin{equation}\label{cep1}
\mathbf{c}_\ep:=\left( |\nabla u|^2+\varepsilon^{2}\right)^{1/2}
\end{equation}
and introduce
\begin{equation}\label{cep2}
\J_\ep :=   \left( \sum\nolimits_{ij} u_{ij}^2  - \sum\nolimits_{i} \left(\sum\nolimits_{j} u_{ij} \frac{u_j}{\mathbf{c}_\ep} \right)^2  \right)^{1/2}.
\end{equation}
Note that we are taking the square root of positive numbers and that now $\mathbf{c}_\ep$ is a~$C^\infty$ function in all of $B_1$. Since $\Delta u_j+f'(u)u_j=0$ and $\J=0$ in $\{\nabla u=0\}$, it is simple to verify that 
\begin{equation}\label{cep3}
\left(\Delta \mathbf{c}_\ep +f'(u) \mathbf{c}_\ep\right)\mathbf{c}_\ep  = \ep^2 f'(u) + \J_\ep^2 \geq  \ep^2 f'(u) + \J^2 \quad\text{ in } B_1.
\end{equation}
Now, from \eqref{eq:07} used with $\mathbf{c}$ replaced by $\mathbf{c}_\ep$ we deduce that
\begin{equation*}
\int_{B_1}\left( \ep^2 f'(u) + \J^2 \right) \eta^2\, dx \le
\int_{B_1}\left(\left|\nabla u\right|^2
+\varepsilon^{2}\right)
\left|\nabla \eta\right|^2\,dx.
\end{equation*}
Letting $\varepsilon\downarrow0$, we conclude the lemma.
\end{proof}

We can now deduce two $L^1$ estimates for second derivatives. One of them, \eqref{ahgiohwiob1bis}, is weighted by the gradient and will lead to our $W^{1,2}$ a priori bound. We now assume $f\geq 0$ and use crucially the superharmonicity of the stable solution.

\begin{lemma}\label{lem:hessian}
Let $u\in C^\infty (\overline B_1)$ be a stable solution of $-\Delta u=f(u)$ in $B_1\subset \R^n$,  for some nonnegative function $f\in C^{1}(\R)$. 

Then,
\begin{equation}\label{estdivbis}
\Vert D^2 u \Vert_{L^1(B_{3/4})} \leq  C \Vert \nabla u \Vert_{L^{2}(B_{1})}
\end{equation}
and
\begin{equation}\label{ahgiohwiob1bis}
\Vert\, |\nabla u| \, D^2 u \, \Vert_{L^1(B_{3/4})} \leq C \Vert \nabla u \Vert^2_{L^{2}(B_{1})}
\end{equation}
for some dimensional constant $C$.
\end{lemma}

\begin{proof} We proceed as in Step 1 of the proof of Proposition 2.4 in \cite{CFRS}.
Define $\nu=\nabla u/|\nabla u|$  in the set $\{\nabla u\neq 0\}$ and $\nu=0$ in $\{\nabla u=0\}$. We claim that
\begin{equation}
\label{hwuighwiu2}
\big| D^2 u - (D^2u[\nu, \nu]) \nu\otimes\nu\big| \leq C\J \quad\text{ a.e.\ in } B_1
\end{equation}
and, as a consequence,
\begin{equation}
\label{HessbyLapl}
|D^2u|\leq |\Delta u|+C\J \quad\text{ a.e.\ in } B_1,
\end{equation}
for some dimensional constants $C$. 

Indeed, at points $x$ for which $\nabla u(x)\neq 0$, \eqref{hwuighwiu2} follows from the fact that $\J^2$, as defined in \eqref{defAAA}, is larger or equal than half the squared Hilbert-Schmidt norm of the matrix $D^2 u - (D^2u[\nu, \nu])\,\nu\otimes\nu$. This is easily seen by writing the symmetric matrix $D^2 u(x)$ in any orthonormal basis having $\nu(x)$ as last (say) basic vector. Note that  $D^2 u - (D^2u[\nu, \nu])\,\nu\otimes\nu$ contains all the Hessian matrix except for its last column. But all the entries in this last column with the exception of $u_{\nu \nu}$ also appear in the last row of $D^2 u - (D^2u[\nu, \nu])\,\nu\otimes\nu$. We conclude the validity, almost everywhere, of \eqref{hwuighwiu2} using Stampacchia's result (see for instance~\cite[Theorem~6.19]{LiebLoss}), which ensures that $D^2u=0$ a.e.\ in $\{\nabla u =0\}$.

Now, to prove \eqref{HessbyLapl}, notice that by \eqref{hwuighwiu2} it suffices to control $|D^2u[\nu, \nu]|$. To accomplish this, we simply rely on 
$$
D^2u[\nu, \nu] = \Delta u -  {\rm tr}\big( D^2 u - (D^2u[\nu, \nu]) \nu\otimes\nu\big)  
$$
and use again \eqref{hwuighwiu2}.

Now, choose a nonnegative function $\zeta\in C^\infty_c(B_1)$ with $\zeta\equiv1$ in $B_{3/4}$.
Then, since $-\Delta u \geq 0$, we have
\begin{equation}\label{LapL1}
\int_{B_{3/4}}|\Delta u|\,dx
\leq -\int_{B_1} \Delta u \,\zeta\,dx=\int_{B_1} \nabla u \cdot \nabla \zeta\,dx \leq C\|\nabla u\|_{L^2(B_1)}.
\end{equation}
This bound, together with \eqref{HessbyLapl} and Lemma~\ref{conseqestab}, gives \eqref{estdivbis}. 

To establish \eqref{ahgiohwiob1bis}, we begin from the pointwise identity
$$
{\rm div}(|\nabla u|\, \nabla u)  =  |\nabla u| \biggl( \sum\nolimits_{ij} \frac{u_{ij}u_i u_j}{|\nabla u|^2} + \Delta u\biggr) \quad\text{in } \{\nabla u\neq0\},
$$
from which we deduce
\begin{equation} \label{hwuighwiu}
 {\rm div}(|\nabla u|\, \nabla u) =    -|\nabla u|\,{\rm tr}\big( D^2 u - (D^2u[\nu, \nu]) \nu\otimes\nu\big)   +  2|\nabla u|\Delta u \quad\text{ a.e.\ in } B_1,
\end{equation}
where tr denotes the trace. Here we have used that $ {\rm div}(|\nabla u|\, \nabla u) = 0$ a.e.\ in $\{\nabla u=0\}$, which is a consequence of Stampacchia's theorem (\cite[Theorem~6.19]{LiebLoss}). Indeed, since $|\nabla u| u_j$ has weak derivatives locally in $L^1$ for every $j$,\footnote{This can be done replacing $|\nabla u|$ by  $(|\nabla u|^2+\varepsilon^2)^{1/2}$ in the definition of weak derivative, and letting $\varepsilon\downarrow 0$ at the end.}
Stampacchia's theorem ensures that $\partial_k (|\nabla u|u_j) = 0$ a.e.\ in $\{u_j=0\}$ for all $k$.

Next, let $\eta\in C^\infty_c(B_1)$ be a cut-off function with $\eta\equiv 1$ in $B_{3/4}$. We use \eqref{hwuighwiu}, \eqref {hwuighwiu2}, and Lemma \ref{conseqestab} to obtain
\begin{eqnarray}\label{pass0}
 \int_{B_{3/4}}  2|\nabla u|\,|\Delta u|\, dx &\le &  
 -\int_{B_{1}} 2|\nabla u|\,\Delta u \, \eta^2\,dx 
 \\  &&\hspace{-3cm}=-\int_{B_1} |\nabla u|\,{\rm tr}\big( D^2 u - (D^2u[\nu, \nu]) \nu\otimes\nu\big)\,  \eta^2\,dx -   \int_{B_1} {\rm div}(|\nabla u|\, \nabla u)\,\eta^2\,dx
 \nonumber\\&&\hspace{-3cm}
 \le  C\left(\int_{B_1} |\nabla u|^2\eta^2\,dx\right)^{1/2} \left(\int_{B_1} \J^2 \eta^2\,dx\right)^{1/2} +   \int_{B_1} \,|\nabla u|\, \nabla u \cdot \nabla (\eta^2)\,dx
 \nonumber\\&&\hspace{-3cm}\label{pass1}
 \le C \|\nabla u\|_{L^{2}(B_{1})}^2.
\end{eqnarray}
Using this bound and \eqref{HessbyLapl} (together with a similar argument as the previous one, to control the integral of $|\nabla u|\J$ in $B_{3/4}$), we conclude \eqref{ahgiohwiob1bis}.
\end{proof}

We now prove that the Dirichlet energy of a stable solution in a ball can be controlled by its $L^1$ norm in a larger ball. To get this result, we combine the previous weighted Hessian estimate with two interpolation lemmas from  Appendix~\ref{app:interp}. This proof is more elementary than the one of \cite{CFRS}, which was based on the more delicate $W^{1,2+\gamma}$ bound of Section~\ref{sect:higher} below.

\begin{proof}[Proof of Proposition \ref{prop:2}]
We cover $B_{1/2}$ (except for a set of measure zero) with a family of disjoint open cubes $Q_j$ of side-length small enough (depending only on $n$) such that $Q_j\subset B_{3/4}$.
We now combine the interpolation inequalities of  Propositions~\ref{prop5.2} and \ref{Nash} in each cube $Q_j$ (we rescale them to pass from the unit cube to the cubes $Q_j$), used with $p=2$, a given $\ep\in(0,1)$, and $\tilde\ep= \ep^{3/2}$ in  Proposition~\ref{Nash}.  We obtain that
\begin{equation*}
\int_{Q_j}\abs{\nabla u}^{2}dx \leq C\varepsilon \int_{Q_j}\abs{\nabla u}\lvert D^2u\rvert\,dx+ C\varepsilon \int_{Q_j}\abs{\nabla u}^2dx+C\varepsilon^{-2-\frac{3n}{2}}\left( \int_{Q_j}\abs{u}\,dx\right)^2
\end{equation*} 
and thus, using $Q_j\subset B_{3/4}$ and estimate \eqref{ahgiohwiob1bis}, that
\begin{equation*}
\int_{Q_j}\abs{\nabla u}^{2}dx \leq C\varepsilon \int_{B_1}\abs{\nabla u}^2dx+C\varepsilon^{-2-\frac{3n}{2}}\left( \int_{B_1}\abs{u}\,dx\right)^2.
\end{equation*} 
Adding up all these inequalities (note that the number of cubes $Q_j$ depends only on~$n$), we get
\begin{equation}\label{zhang}
\|\nabla u\|_{L^{2}(B_{1/2})}^2 \le C\varepsilon  \|\nabla u\|_{L^{2}(B_{1})}^2 + C\varepsilon^{-2-\frac{3n}{2}} \|u\|_{L^{1}(B_{1})}^2.
\end{equation}

Applying this estimate to the functions $u_{\rho,y}(\bar x):=u(y+\rho \bar x)$, where $\bar x\in B_1$ and $B_\rho(y)\subset B_1$ (note that $u_{\rho,y}$ is a stable solution to the semilinear equation $-\Delta u_{\rho,y}=f_{\rho}(u_{\rho,y})$ in~$B_1$ with $f_\rho(t)=\rho^2f(t)$, and thus all the previous results apply to $u_{\rho,y}$ as well), we conclude that
\begin{align*}
\rho^{n+2}\int_{B_{\rho/2}(y)}|\nabla u|^2\,dx&\leq 
C\varepsilon \rho^{n+2}\int_{B_{\rho}(y)}|\nabla u|^2\,dx
+ C\varepsilon^{-2-\frac{3n}{2}}\left(\int_{B_{\rho}(y)}|u|\,dx\right)^2\\
& \leq C\varepsilon \rho^{n+2}\int_{B_{\rho}(y)}|\nabla u|^2\,dx
+C\varepsilon^{-2-\frac{3n}{2}}\left(\int_{B_{1}}|u|\,dx \right)^2
\end{align*}
for every $\varepsilon\in(0,1)$.
By Lemma \ref{lem_abstract}, applied with
$\sigma(B):=\|\nabla u\|_{L^2(B)}^2$ and $\varepsilon$ sufficiently small, the result follows.
\end{proof}

\section{The radial derivative controls the function in $L^{1}$}
\label{sect:L1byRad}

In the previous section we have controlled the $W^{1,2}$ norm of a stable solution by its $L^1$ norm. We now prove, as claimed in Proposition~\ref{prop:3}, that the $L^1$ norm of the solution (in fact of any superharmonic function) in an annulus (respectively, in a ball) can be controlled, up to an additive constant, by the $L^1$ norm of its radial derivative, taken also in the annulus (respectively, in the ball). For this, we first need a similar result for harmonic functions.

\begin{lemma}{\rm (\cite[Lemma 2.1]{C22quant})}\label{lem:RadcontrolsHarm}
Let $v\in C^\infty(\overline B_1)$ solve $\Delta v=0$ in $B_1\subset \R^n$. Then,
\begin{equation}\label{Linfty}
\Vert v-v(0)\Vert_{L^{\infty}(\partial B_{1})}\leq 2n^{3/2}\Vert v_{r}\Vert_{L^{\infty}(\partial B_{1})}
\end{equation} 
and
\begin{equation}\label{L1}
\Vert v-t\Vert_{L^{1}(\partial B_{1})}\leq 2n^{3/2}\Vert v_{r}\Vert_{L^{1}(\partial B_{1})},
\end{equation} 
where $t:=\inf \left\{ \overline{t} : |\{v>\overline{t}\}\cap\partial B_{1}|\leq |\partial B_{1}|/2\right\}$ is the median of $v$.
\end{lemma}

The previous estimates should not be surprising. Indeed, we are controlling the solution to a boundary value problem for an homogeneous equation by its Neumann data ${v_r}_{|_{\partial B_1}}$. Alternatively, the flux ${v_r}_{|_{\partial B_1}}$ can be considered as a first-order elliptic integro-differential operator acting on the function $v_{|_{\partial B_1}}$. This operator is a kind of ``half-Laplacian'', since we extend the data $v_{|_{\partial B_1}}$ harmonically in the ball ---the half-Laplacian corresponds to the harmonic extension in a half-space. Anyhow, these comments will not be needed in the following elementary proof.

\begin{proof}[Proof of Lemma \ref{lem:RadcontrolsHarm}]
We follow \cite{C22quant}. We start proving \eqref{Linfty}. From it, we will easily deduce \eqref{L1} by duality. 

Let
$$
w(x):= x\cdot \nabla v (x) = r v_r \qquad\text{ for } x\in \overline B_ 1.
$$
We claim that
\begin{equation}\label{pointw}
|w(x)|\leq 2n^{3/2} \Vert v_{r}\Vert_{L^{\infty}(\partial B_{1})} \, r \qquad\text{ for } x\in \overline B_ 1.
\end{equation}
From this, \eqref{Linfty} will clearly follow since, for all $\sigma\in \partial B_1$ we have
$$
v(\sigma)-v(0)=\int_0^1 v_r(r\sigma) \,dr= \int_0^1 r^{-1} w(r\sigma)\,dr.
$$ 

Now, to prove \eqref{pointw}, note that $w$ is harmonic in $B_1$ and agrees with $v_r$ on $\partial B_1$. Consequently,
\begin{equation}\label{Linftyw}
\Vert w\Vert_{L^{\infty}(B_{1})} \leq \Vert v_r\Vert_{L^{\infty}(\partial B_{1})}.
\end{equation}
It follows that \eqref{pointw} holds whenever $|x|=r\geq 1/2$. 

Assume now that $x\in B_{1/2}$. Note that
\begin{equation}\label{wdiff}
|w(x)|= |w(x)-w(0)|\leq \Vert \nabla w \Vert_{L^{\infty}(B_{1/2})} \, r.
\end{equation}
Using that the partial derivatives $w_i$ of $w$ are harmonic and also \eqref{Linftyw}, we see that
\begin{align*}
|w_i(x)|&=\frac{1}{|B_{1/2}|}\left| \int_{B_{1/2}(x)} w_i \, dx \right|= 
\frac{1}{|B_{1/2}|}\left| \int_{\partial B_{1/2}(x)} w \nu^i \, d\HH^{n-1} \right| 
\\
&\leq \frac{|\partial B_{1/2}|}{|B_{1/2}|} \Vert v_r\Vert_{L^{\infty}(\partial B_{1})}
= 2n \Vert v_r\Vert_{L^{\infty}(\partial B_{1})}.
\end{align*}
Hence $|\nabla w(x)|\leq 2n^{3/2} \Vert v_r\Vert_{L^{\infty}(\partial B_{1})}$, which combined with \eqref{wdiff}, gives \eqref{pointw} when  $r< 1/2$.

Next, we establish \eqref{L1} by a duality argument. First note that the value $t$ is finite and well-defined. Replacing $v$ by $v-t$ we may assume that $t=0$, and therefore that $0=\inf \left\{ \overline{t} : |\{v>\overline{t}\}\cap\partial B_{1}|\leq |\partial B_{1}|/2\right\}$. It follows that $|\{v>0\}\cap\partial B_{1}|\leq |\partial B_{1}|/2$ and that $|\{v>\overline{t}\}\cap\partial B_{1}|> |\partial B_{1}|/2$ for all $\overline{t}<0$. As a consequence, $|\{v<\overline{t}\}\cap\partial B_{1}|\leq |\partial B_{1}|/2$ for all $\overline{t}<0$, and thus $|\{v<0\}\cap\partial B_{1}|\leq |\partial B_{1}|/2$. 

Let now $\text{sgn} (v):=v/|v|$ where $v\neq 0$. Since $|\{v>0\}\cap\partial B_{1}|\leq |\partial B_{1}|/2$ and $|\{v<0\}\cap\partial B_{1}|\leq |\partial B_{1}|/2$, it is possible to extend $ \text{sgn} (v)$ to $\{v=0\}\cap\partial B_{1}$, still taking values $\pm 1$ and in such a way that we have $\int_{\partial B_1} \text{sgn} (v)\, d\HH^{n-1} =0$.\footnote{This follows from the fact that, given a measurable set $A\subset \partial B_1$ (which we take here to be $\{v=0\}\cap\partial B_{1}$) and $\theta\in (0,1)$, there exists a measurable subset $B\subset A$ with $|B|=\theta |A|$. This follows from the continuity with respect to $\rho$ of the quantity $|A\cap B_\rho(1,0,\ldots,0)|$.} In addition, it will also hold that $|v|=v\, \text{sgn} (v)$ on $\partial B_1$. 

We now define the functions
$$
g_k (x) := \int_{\partial B_1} \text{sgn} (v)(y)  \, \rho_k (|x-y|)\, d\HH^{n-1} (y),
$$
where $\{\rho_k=\rho_k(|\cdot | )\}$ is a sequence of smooth mollifiers on $\partial B_1$. We have that $g_k\in C^\infty (\partial B_1)$, $|g_k|\leq 1$, and $\int_{\partial B_1} g_k\, d\HH^{n-1} =0$ since $\text{sgn} (v)$ has zero average on~$\partial B_1$. In addition, since $|v|=v\, \text{sgn} (v)$ a.e.\ on $\partial B_1$, it holds that
\begin{equation}\label{limgk}
\int_{\partial B_1} |v|\, d\HH^{n-1} = \lim_k \int_{\partial B_1} v g_k\, d\HH^{n-1}.
\end{equation}

Now, since $g_k$ is smooth and has zero average on $\partial B_1$, we can uniquely solve the problem
$$
\left\{
\begin{array}{cl}
\Delta \varphi_k=0 & \text{in }B_1\\
\partial_r \varphi_k=g_k & \text{on }\partial B_1
\end{array}
\right.
$$
by imposing, additionally, $\varphi_k(0)=0$. By \eqref{Linfty}, we have $\Vert \varphi_k\Vert_{L^{\infty}(\partial B_{1})}\leq 2n^{3/2}$. We conclude that
\begin{align*}
\left| \int_{\partial B_{1}} v g_k \, d\HH^{n-1} \right| & =
\left| \int_{\partial B_{1}} v \, \partial_r \varphi_k \, d\HH^{n-1} \right| =\left| \int_{B_{1}} \nabla v \cdot \nabla\varphi_k \, dx \right|
=\left| \int_{\partial B_{1}} v_r \, \varphi_k \, d\HH^{n-1} \right| 
\\
&
\leq 2n^{3/2}  \int_{\partial B_{1}} |v_r|\, d\HH^{n-1}.
\end{align*}
This, together with \eqref{limgk}, concludes the proof of \eqref{L1}.
\end{proof}

We can now give the

\begin{proof}[Proof of Proposition \ref{prop:3}]
We follow \cite{C22quant}. Since 
$$
\Vert u_{r}\Vert_{L^{1}({B_{1}\setminus B_{1/2})}}= \ave_{1/2}^1 \left( \int_{\partial B_r} \frac{1}{2} |u_{r}|\, d\HH^{n-1} \right) dr,
$$ 
there exists $\rho\in [1/2,1]$ such that $ \Vert u_{r}\Vert_{L^{1}({B_{1}\setminus B_{1/2})}}= \frac{1}{2}\int_{\partial B_\rho}  |u_{r}|\, d\HH^{n-1}$.

Let $v$ be the harmonic function in $B_\rho$ which agrees with $u$ on $\partial B_\rho$. We have $v\in C^\infty (\overline B_\rho)$. Decompose the radial derivatives $u_r=u_r^+-u_r^-$ and $v_r=v_r^+-v_r^-$ in their positive and negative parts. Since $v\leq u$ in $B_\rho$ and the functions agree on the boundary, we have $v_r\geq u_r$ on $\partial B_\rho$. As a consequence, $v_r^-\leq u_r^-$ on $\partial B_\rho$. Since, in addition 
$
0=\int_{B_\rho}  \Delta v\, dx = \int_{\partial B_\rho}  v_{r}\, d\HH^{n-1} = \int_{\partial B_\rho}  v_{r}^+\, d\HH^{n-1}-\int_{\partial B_\rho}  v_{r}^-\, d\HH^{n-1},
$
it follows that
\begin{align*}
2\Vert u_{r}\Vert_{L^{1}({B_{1}\setminus B_{1/2})}} &= \int_{\partial B_\rho}  |u_{r}|\, d\HH^{n-1}
\geq \int_{\partial B_\rho}  u_{r}^-\, d\HH^{n-1} \\
&\geq \int_{\partial B_\rho}  v_{r}^-\, d\HH^{n-1}=\int_{\partial B_\rho}  v_{r}^+\, d\HH^{n-1}= \frac{1}{2}\int_{\partial B_\rho}  |v_{r}|\, d\HH^{n-1}.
\end{align*}

Now, since $\rho\in [1/2,1]$, we can rescale the estimate \eqref{L1} in Lemma~\ref{lem:RadcontrolsHarm} and apply it to the harmonic function $v$ to deduce
$$
C\int_{\partial B_\rho}  |v_{r}|\, d\HH^{n-1} \geq \int_{\partial B_\rho}  |v-t|\, d\HH^{n-1}
=  \int_{\partial B_\rho}  |u-t|\, d\HH^{n-1}
$$
for some value $t$ and a dimensional constant $C$, where in the last equality we used that $v=u$ on $\partial B_\rho$. Therefore, by the two previous estimates, the proposition will be proven once we establish that, for any $t\in\R$,
\begin{equation}\label{conclsuper}
\Vert u-t\Vert_{L^{1}({B_{1}\setminus B_{1/2})}} \leq C\left( 
\int_{\partial B_\rho}  |u-t|\, d\HH^{n-1}+
\Vert u_r\Vert_{L^{1}({B_{1}\setminus B_{1/2})}}
\right)
\end{equation}
and
\begin{equation}\label{conclsuper-ball}
\Vert u-t\Vert_{L^{1}({B_{1/2})}} \leq C\left( 
\int_{\partial B_\rho}  |u-t|\, d\HH^{n-1}+
\Vert u_r\Vert_{L^{1}({B_{1})}}
\right).
\end{equation}

This is simple. For both estimates, we will use that
$$
(u-t)(s\sigma)=(u-t)(\rho\sigma)-\int_s^\rho u_r (r\sigma) \,dr
$$
for every $s\in (0,1)$ and $\sigma\in S^{n-1}$. 

Now, to check \eqref{conclsuper} we take $s\in (1/2,1)$ and note that
$$
s^{n-1}|(u-t)(s\sigma)|\leq 2^{n-1}\rho^{n-1} |(u-t)(\rho\sigma)|+2^{n-1}\int_{1/2}^1 r^{n-1} |u_r (r\sigma)| \,dr.
$$
Integrating in $\sigma\in S^{n-1}$, and then in $s\in (1/2,1)$, we conclude  \eqref{conclsuper}.

Finally, to prove \eqref{conclsuper-ball} we take $s\in (0,1/2)$ and note that (since $\rho\geq 1/2>s$)
\begin{align*}
s^{n-1}|(u-t)(s\sigma)| &\leq \rho^{n-1} |(u-t)(\rho\sigma)|+\int_{s}^\rho r^{n-1} |u_r (r\sigma)| \,dr \\
& \leq \rho^{n-1} |(u-t)(\rho\sigma)|+\int_{0}^1 r^{n-1} |u_r (r\sigma)| \,dr.
\end{align*}
Integrating in $\sigma\in S^{n-1}$, and then in $s\in (0,1/2)$, we conclude  \eqref{conclsuper-ball} by using also~\eqref{conclsuper}.
\end{proof}

\section{$C^{\alpha}$ estimate}
\label{sect:Holder}

Here we prove the interior H\"older estimate for $n\leq 9$. It yields, in particular, an interior $L^\infty$ bound for stable solutions.

\begin{proof}[Proof of the H\"older estimate \eqref{eq:Ca L1 int} in Theorem \ref{thm:0}]
We may assume $3 \leq n \leq 9$ ---~note that $n\geq 3$ is needed to use Proposition~\ref{prop:1}. Indeed, in case $n\leq 2$ one can easily reduce the problem to the case $n=3$ by adding artificial variables.\footnote{While writing this work, we have found a direct, alternative, proof for $n=2$ which does not require going to $\R^3$. It is as follows. 
One uses the higher integrability estimate \eqref{eq:W12 L1 int} of Theorem~\ref{thm:0}, which is proven in Section~6 for $n\geq 2$. From this, one concludes by using Morrey's embedding (see Theorem~\ref{thmC2} or \cite[Theorem 7.19]{GT}). Alternatively, instead of our results of Section~6, one can use Gehring's lemma (as described in the following Footnote~\ref{foot-Gehring}) after 
using Propositions~\ref{prop:2}  and \ref{prop:3} to get that the $L^2$ norm of the gradient in $B_{1/2}$ is controlled by the $L^1$ norm of the gradient in the larger ball $B_1$.} This simply means looking at the function of two variables (respectively, of one variable) as a function of three variables which does not depend on the third Euclidean variable (restpectively, on the last two variables). Note that both the equation and the stability condition are preserved under this procedure. 

Using Proposition~\ref{prop:1} with $\rho=5/8$, we have
\begin{equation}\label{eq:11proof}
\int_{B_{1/2}}r^{2-n} u_{r}^2\,dx \leq \int_{B_{5/8}}r^{2-n} u_{r}^2\,dx 
\leq C\int_{B_{15/16}\setminus B_{5/8}}|\nabla u|^2\,dx.
\end{equation}

Now, since $1/2<5/8<15/16<1$, we can cover the annulus $B_{15/16}\setminus B_{5/8}$ by a finite number of balls $B_j$, all of them with the same radius, that we choose small enough for the balls $2B_j$ (with same center as $B_j$ and twice its radius) to satisfy $2B_j\subset B_{1}\setminus B_{1/2}$. Now, let $t\in\R$.  In each ball $2B_j$ we can apply Proposition~\ref{prop:2}, rescaled, with $u$ replaced by $u-t$ ---since $u-t$ is a stable solution of $-\Delta v= f(v+t)$. We deduce that
$$
\int_{B_j}|\nabla u|^2\,dx\leq C\Vert u-t\Vert_{L^1(2B_j)}^2 \leq C\Vert u-t\Vert_{L^1(B_{1}\setminus B_{1/2})}^2 .
$$
Adding up in $j$, this yields
\begin{equation}\label{eq:12proof}
\int_{B_{15/16}\setminus B_{5/8}}|\nabla u|^2\,dx\leq  C\Vert u-t\Vert_{L^1(B_{1}\setminus B_{1/2})}^2 \qquad\text{ for all } t\in\R,
\end{equation}
where $C$ always denotes dimensional constants.

We now use estimate \eqref{new-int-ur} from Proposition \ref{prop:3} to deduce
\begin{equation*}
\inf_{t\in\R} \Vert u-t\Vert_{L^1(B_{1}\setminus B_{1/2})}^2\le C\Vert u_r\Vert_{L^1(B_{1}\setminus B_{1/2})}^2 \le 
 C\Vert u_r\Vert_{L^2(B_{1}\setminus B_{1/2})}^2\le
C\int_{B_{1}\setminus B_{1/2}} r^{2-n} u_r^2\,dx.
\end{equation*}
This, together with \eqref{eq:11proof} and \eqref{eq:12proof}, leads to
$$
\int_{B_{1/2}}r^{2-n} u_{r}^2\,dx \leq C\int_{B_{1}\setminus B_{1/2}} r^{2-n} u_r^2\,dx.
$$
We can write this inequality, in an equivalent way, as
$$
\int_{B_{1/2}}r^{2-n} u_{r}^2\,dx \leq  \theta \int_{B_{1}} r^{2-n} u_r^2\,dx
$$ 
for the dimensional constant $\theta= \frac{C}{1+C}\in (0,1)$.

For every $\rho \in (0,1)$, this estimate applied to the stable solution $u(\rho\, \cdot)$ yields
\begin{equation}\label{recurrence}
\int_{B_{\rho/2}}r^{2-n} u_{r}^2\,dx \leq \theta\int_{B_{\rho}} r^{2-n} u_r^2\,dx
\end{equation}
(note that the integrals are rescale invariant).
This inequality can be iterated, in the balls of radius $2^{-j}$ centered at $0$, for $j\ge 1$ an integer, to obtain
\begin{equation*}
\int_{B_{2^{-j}}}r^{2-n} u_{r}^2\,dx \leq \theta^{j-1} \int_{B_{1/2}}r^{2-n} u_{r}^2\,dx.
\end{equation*}
Since $\theta\in (0,1)$, it follows that, for some dimensional $\alpha\in (0,1)$,
\begin{equation}\label{finaltoholder}
\int_{B_\rho} r^{2-n}u_r^2\,dx\leq C\rho^{2\alpha}\int_{B_{1/3}} r^{2-n} u_r^2\,dx \le C\rho^{2\alpha}\Vert\nabla u\Vert_{L^2(B_{1/2})}^2 \quad\text{ for all } \rho \leq 1/3,
\end{equation}
where we have used Proposition~\ref{prop:1} with $\rho=1/3$ in the last inequality. In particular, since $\rho^{2-n}\leq r^{2-n}$ in $B_\rho$, we conclude that
\begin{equation*}
\int_{B_\rho} u_r^2\,dx  \le C\rho^{n-2+2\alpha}\Vert\nabla u\Vert_{L^2(B_{1/2})}^2
\le C\rho^{n-2+2\alpha}\Vert u\Vert_{L^1(B_{1})}^2  \quad\text{ for all } \rho \leq 1/3,
\end{equation*}
where we have used Proposition \ref{prop:2} in the last inequality.

Next, given $y\in \bar B_{1/12}$, we can apply the last estimate to the function $u_y(x):=u(y+\frac{x}{2})$, defined for $x\in B_{1}$, since $y+\frac{1}{2}B_{1}\subset B_{1}$. Denoting by
$$
u_{r_y}(x):= \frac{x-y}{|x-y|}\cdot \nabla u (x)
$$
the radial derivative based at the point $y$, we get
\begin{equation}\label{radial2}
\int_{B_\rho(y)} u_{r_y}^2\,dx\leq C\rho^{n-2+2\alpha}\Vert u\Vert_{L^1(B_1)}^2\quad \mbox{ for all }\,y \in \bar B_{1/12} \text{ and } \rho \leq 1/6.
\end{equation}
From \eqref{radial2}, we immediately get
\begin{equation}\label{radial1}
\int_{B_\rho(y)}|u_{r_y}|\,dx\leq C\rho^{n-1+\alpha}\Vert u\Vert_{L^1(B_1)}\quad \mbox{ for all }\,y \in \bar B_{1/12}\text{ and } \rho \leq 1/6.
\end{equation}

At this point we use Morrey's estimate, as stated and proven in Appendix \ref{app:morrey}. By Theorem~\ref{thmC2} applied to $u(\frac{\cdot}{3})$ with $\bar C= C\Vert u\Vert_{L^1(B_1)}$, \eqref{radial1} yields\footnote{In \cite{CFRS} we employed an alternative way to conclude the H\"older estimate. It consists of averaging \eqref{radial2}  with respect to $y$ to get the analogue estimate with the radial derivative in the integrand replaced by the full gradient. From this, by classical estimates on Morrey spaces (see for instance \cite[Theorem 7.19]{GT}) one deduces the desired bound. Instead, in this paper we see that there is no need to average in $y$ since, as we show in Appendix \ref{app:morrey}, Morrey's proof involves, in fact,  radial derivatives ---and not the full gradient.}
$$
\|u\|_{C^\alpha(\overline B_{1/12})}\leq C\|u\|_{L^1(B_{1})}.
$$

Finally, from this estimate, \eqref{eq:Ca L1 int} follows by a standard covering and scaling argument.
\end{proof}

\section{$W^{1,2+\gamma}$ estimate}
\label{sect:higher}

Here we establish a higher $L^{2+\gamma}$ integrability result for the gradient of the solution, as stated in Theorem~\ref{thm:0}. In \cite{CFRS} this result was used crucially to get compacity in $W^{1,2}$, locally in the  interior, for any sequence of stable solutions that are uniformly bounded in $L^1$ ---here the solutions may correspond to different nonnegative nonlinearities. This was an essential tool in the non-quantitative proof of H\"older continuity in \cite{CFRS} to control the full gradient by the radial derivative. Instead, in the current paper we have not used the higher integrability result to establish interior H\"older continuity. We present it here, however, because of the interest it has by itself and also since an analogue higher integrability result will be needed in our boundary regularity proof; see Lemma~\ref{lem:bdry-gamma-2} and the comments before it.

In \cite{CFRS} (and here), the proof of the $W^{1,2+\gamma}$ estimate\footnote{\label{foot-Gehring}Yi Zhang, in an interesting personal communication, has pointed out that a weaker version of this higher integrability result can also be directly deduced from  our $W^{1,2}$ estimate and Gehring's lemma. More precisely, one uses the version of Gehring's lemma in Proposition 1.1 of  \cite[Chapter~5]{Giaq}. One must also use standard covering and scaling arguments to pass from balls to cubes, and viceversa. One starts from \eqref{zhang}, then use the cited version of Gehring's lemma (here, in the notation of~\cite{Giaq},  $p=2+\gamma>2=q$), later employ Proposition~\ref{Nash} to control the $L^p$ norm of~$u$, and finally Simon's Lemma \ref{lem_abstract}. However, notice that while the exponent of higher integrability depends only on~$n$ in \cite{CFRS},  after the use of Gehring's lemma the exponent also depends on the constants in the control of the $L^2$ norm by the $L^1$ norm. This gives a particularly weaker result in the case of equations with variable coefficients, as treated in \cite{Ern1}. Indeed, proceeding in this way, the constants in the control of the $L^2$ norm by the $L^1$ norm (and hence the higher integrability exponent $2+\gamma$) will depend also on the ellipticity constants of the coefficients ---instead of only on~$n$ as in \cite{Ern1, Ern2, Ern3}.}
is based in the following bound on every level set of the solution.\footnote{Since level sets appear here, it is worth noticing (even though we will not use this in the paper) that the quantity $\mathcal A$ in \eqref{defAAA} controls the second fundamental form of every level set of $u$. This was crucially used in \cite{C10}, in combination with the Sobolev-type inequality of  Michael-Simon and Allard (applied on every level set of $u$), to prove regularity of stable solutions up to dimension $n \leq 4$ for all nonlinearities $f$ (including sign-changing nonlinearities).}

\begin{lemma}{\rm (\cite[(2.13)]{CFRS})}\label{levelsets}
Let $u\in C^\infty (\overline B_1)$ be a stable solution of $-\Delta u=f(u)$ in $B_1\subset \R^n$,  for some nonnegative function $f\in C^{1}(\R)$. 

Then, for almost every $t\in\R$, we have
\begin{equation}\label{ahgiohwiob1}
\int_{\{u=t\}\cap B_{1/2}} |\nabla u|^2 \, d\HH^{n-1}  \le  C\Vert \nabla u\Vert^{2}_{L^{2}(B_{1})}
\end{equation}
for some dimensional constant $C$.
\end{lemma}

\begin{proof}
Let  $\eta \in C^\infty_c(B_{3/4})$ with $\eta\equiv1$ in $B_{1/2}$.
Combining \eqref{hwuighwiu} with \eqref{hwuighwiu2} and \eqref{pass0}-\eqref{pass1},
and using again Lemma \ref{conseqestab}, we get
\begin{equation}\label{absdiv}
\begin{split}
\int_{B_{1}} \big| {\rm div}(|\nabla u|\, \nabla u) \big| \eta^2 \,dx &\le  C \int_{B_{1}} |\nabla u| \J\, \eta^2\,dx + \int_{B_{1}} 2|\nabla u|\,|\Delta u|\, \eta^2\,dx \\
&\hspace{-3cm} \leq C\left(\int_{B_1} |\nabla u|^2\eta^2\,dx\right)^{1/2} \left(\int_{B_1} \J^2 \eta^2\,dx\right)^{1/2} + C \|\nabla u\|_{L^{2}(B_{1})}^2\\
&\hspace{-3cm}\le C \|\nabla u\|_{L^{2}(B_{1})}^2.
\end{split}
\end{equation}

Now, by Sard's theorem, $\{u=t\}$ is a smooth hypersurface for almost every $t\in\R$. For such values of $t$, we can apply the divergence theorem and get
\begin{align*}\label{claim-12}
 \int_{\{u=t\}\cap B_{1/2}} |\nabla u|^2 \, d\HH^{n-1} & \leq \int_{\{u=t\}\cap B_{1}} |\nabla u|^2\eta^2 \,d\HH^{n-1}
 \\
 & =- \int_{\{u>t\} \cap B_1 }  {\rm div}\big(|\nabla u|\, \nabla u \,\eta^2\big) \,dx.
\end{align*}
This bound, combined with \eqref{absdiv}, gives \eqref{ahgiohwiob1}.
\end{proof}

\begin{proof}[Proof of the $W^{1,2+\gamma}$ estimate \eqref{eq:W12 L1 int} in Theorem \ref{thm:0}]
Since $u/ \|\nabla u\|_{L^{2}(B_{1})}$ is a stable solution of a new equation, we may assume $\|\nabla u\|_{L^{2}(B_{1})}=1$. This normalization will simplify the following proof.

First note that, with $\bar u := \ave_{B_1} u$, we have
\begin{equation*}
\left(\int_{B_{1}} |u-\bar u|^p \, dx\right)^{\frac 1 p } \le C\left( \int_{B_1} |\nabla u|^2\,dx\right)^{\frac 1 2 }=C
\end{equation*}
for some dimensional exponent $p>2$ and constant $C$ ---this is the Sobolev inequality with average for $W^{1,2}(B_1)$ functions.\footnote{For $n\geq 3$ this follows combining the Sobolev embedding (\cite[Theorem 7.26]{GT}) (see also \cite[Theorem 7.10]{GT}) applied to $u-\bar u$ with Poincar\'e's inequality with average and exponent 2 (\cite[(7.45)]{GT}). For $n=2$ the claim follows from \cite[Theorem 7.21]{GT}.} 
Using the previous bound and the coarea formula, we obtain 
\begin{equation}\label{ahgiohwiob2}
 \int_{\R} dt \int_{\{u=t\}\cap B_{1}\cap \{|\nabla u|\neq0\}}  d\HH^{n-1} |t-\bar u|^p \,|\nabla u|^{-1}  =   \int_{B_1} |u-\bar u|^p\chi_{\{|\nabla u| \neq 0\}}  \, dx  \le C.
\end{equation}
Also, since $p>2$, we may choose dimensional constants $q>1$ and $\theta\in (0,1/3)$ such that $p/q = (1-\theta)/\theta$.  Thus, 
defining 
\[
h(t) : =  \max\big\{1, |t-\bar u|\big\}
\]
and using the coarea formula and H\"older's inequality (note that $p\theta -q(1-\theta)=0$),  we obtain
\[
\begin{split}
\int_{B_{1/2}} |\nabla u|^{3-3\theta} \,dx&=\int_{\R} dt\int_{\{u=t\}\cap B_{1/2}\cap \{|\nabla u|\neq0\}} d\HH^{n-1} h(t)^{p\theta -q(1-\theta)}  |\nabla u|^{-\theta + 2(1-\theta)} 
\\
&\le   \left(\int_{\R} dt\int_{\{u=t\}\cap B_1\cap \{|\nabla u|\neq0\}} d\HH^{n-1} h(t)^{p}  |\nabla u|^{-1} \right)^\theta  \cdot \\
& \hspace{3cm} \cdot \bigg(\int_{\R} dt\int_{\{u=t\}\cap B_{1/2}} \hspace{-3mm}  h(t)^{-q} |\nabla u|^2  \bigg)^{1-\theta}.
\end{split}
\]

Observe now that, thanks to the definition of $h(t)$ and  \eqref{ahgiohwiob2}, we have
\begin{align*}
\int_{\R} dt & \int_{\{u=t\}\cap B_1\cap \{|\nabla u|\neq0\}} d\HH^{n-1}   h(t)^{p}  |\nabla u|^{-1} \\
& \leq \int_{\bar u-1}^{\bar u+1} dt\int_{\{u=t\}\cap B_1\cap \{|\nabla u|\neq0\}}  d\HH^{n-1}    |\nabla u|^{-1} +C\\
&\leq |B_1|+C\leq  C.
\end{align*}
Also, since $q>1$ it follows that  $\int_{\R} h(t)^{-q}dt$ is finite, and thus  \eqref{ahgiohwiob1} leads to
\[
\int_{\R} dt \, h(t)^{-q} \int_{\{u=t\}\cap B_{1/2}} d\HH^{n-1} |\nabla u|^2\le C\int_{\R} h(t)^{-q}\,dt \le C.
\] 
Therefore, we have proven that
$\int_{B_{1/2}} |\nabla u|^{3-3\theta}  \,dx \le C$
for some dimensional constants  $\theta\in(0,1/3)$  and $C$, as desired.
\end{proof}

\section{The case $f \geq -K$, with $K$ a nonnegative constant\vspace{.15cm}}
\label{sect:fbddbelow}

In this section we point out the necessary changes to treat the case $f\geq -K$, with $K$ being a nonnegative constant, as claimed in Remark~\ref{fbddbelow}.  We need to revise the proofs of H\"older regularity in Section~\ref{sect:Holder} and of the $W^{1,2+\gamma}$ estimate in Section~\ref{sect:higher}. 

First, Lemmas~\ref{conseqestab2} and \ref{conseqestab} did not require $f\geq 0$. Now,
the estimates of Lemma~\ref{lem:hessian} and Proposition~\ref{prop:2} still hold after adding the constant  $K$, in their right-hand sides, to the $L^2$ norm of $\nabla u$ and to the $L^1$ norm of $u$. These changes come from the arguments in \eqref{LapL1} and \eqref{pass0} ---the only places where we used $-\Delta u= f(u) \geq 0$. To see this, we consider the superharmonic function 
$$
w:= u-K |x|^2/(2n)
$$ 
and, in \eqref{LapL1}, we proceed as 
\begin{eqnarray*}
\int_{B_{3/4}}|\Delta u|\,dx & \leq & CK+ \int_{B_{3/4}}|\Delta w|\,dx \leq CK+\int_{B_1} (-\Delta w)\,\zeta\,dx \\
&=& C K+\int_{B_1} (-\Delta u+K) \,\zeta\,dx \leq CK+\int_{B_1} \nabla u \cdot \nabla \zeta\,dx
\\&\leq & C( \|\nabla u\|_{L^2(B_1)}+K).
\end{eqnarray*}
In \eqref{pass0}, we proceed as 
\begin{eqnarray}\label{estuu}
 \nonumber \int_{B_{3/4}}  2|\nabla u|\,|\Delta u|\, dx &\le &  
  \nonumber CK \|\nabla u\|_{L^{1}(B_{1})} + \int_{B_{3/4}} 2|\nabla u|\,|\Delta w| \,dx 
 \\  &&\hspace{-3cm}\leq
  \nonumber  CK  \|\nabla u\|_{L^{2}(B_{1})} - \int_{B_{1}} 2|\nabla u|\, \Delta w \, \eta^2\,dx \\
 &&\hspace{-3cm}
 \leq C(\|\nabla u\|_{L^{2}(B_{1})} +K)^2 - \int_{B_{1}} 2|\nabla u|\, \Delta u \, \eta^2\,dx 
 \end{eqnarray}
and finish the estimate as in \eqref{pass0}-\eqref{pass1}.

Second, we apply Proposition~\ref{prop:3} to the superharmonic function $w$. We deduce that the estimates of the proposition also hold for $u$ after adding $K$, in its right-hand sides, to the $L^1$ norm of $u_r$. 

Now, one proceeds to the proof of the H\"older estimate of Theorem~\ref{thm:0}, as given in Section~\ref{sect:Holder}. Here one must replace inequality~\eqref{recurrence}, to be iterated, by
$$
\int_{B_{\rho/2}}r^{2-n} u_{r}^2\,dx \leq \theta \int_{B_{\rho}} r^{2-n} u_r^2\,dx + C K^2 \rho^4.
$$
Now, using Lemma~8.23 of \cite{GT} (taking $\mu=1/2$, for instance), we see that \eqref{finaltoholder} still holds when adding $K$ to the $L^2$ norm of $\nabla u$. 

Finally, for the $W^{1,2+\gamma}$ estimate in Section~\ref{sect:higher}, thanks to \eqref{estuu} we see that 
\eqref{ahgiohwiob1} also holds after adding $K$ to the $L^2$ norm of $\nabla u$ in its right-hand side. Now, in the beginning of the proof of the $W^{1,2+\gamma}$ estimate, given in Section~\ref{sect:higher}, we simply consider $u/ (\|\nabla u\|_{L^{2}(B_{1})}+K)$ instead of $u/ \|\nabla u\|_{L^{2}(B_{1})}$, and the remaining of the proof is the same.

\bigskip\medskip
\centerline{\large{\sc Part II: Boundary regularity}}
\addtocontents{toc}{\textsc{\hspace{.2cm} Part II:  Boundary regularity \vspace{.1cm}}}
\medskip

In this second part of the article, we establish the boundary regularity results. Sections \ref{sect:boundary-L1radial} and \ref{sect:boundary-conclusion} are the main novelties. Proposition~\ref{prop:3bdry}, proved in Section~\ref{sect:boundary-L1radial}, significantly simplifies the proofs in \cite{CFRS}, which used delicate compactness, blow-up, and Liouville-type arguments.

With $\R^n_+=\{x\in\R^n\, :\, x_n>0\}$ and $r=|x|$, we use the notation
$$
B_\rho^+=\R^n_+\cap B_\rho, \quad A_{\rho, \overline\rho} := \{\rho< r< \overline\rho\},\quad A^+_{\rho, \overline\rho} := \{x_n>0, \rho< r< \overline\rho\},
$$
and
$$
\partial^0  \Omega= \{x_n=0\}\cap \partial \Omega \quad\text{and}\quad \partial^+  \Omega=\R^n_+\cap \partial \Omega
$$
for an open set $\Omega\subset \R^n_+$.

In the stability inequality we will consider (as in the interior case) test functions of the form $\xi=\mathbf{c}\eta$,
with $\mathbf{c}\in W^{2,\infty}(B_1^+)$, $\eta$ a Lipschitz function in $\overline{B_1^+}$, and $\mathbf{c}\eta$ vanishing on the full boundary~$\partial  B_1^+$. In particular, when $\mathbf{c}$ vanishes on $\partial^0B_1^+$, we do not need to require $\eta$ to vanish on this set.  Now, the integration by parts argument in~\eqref{eq:parts}-\eqref{eq:07} leads to
\begin{equation}\label{stabbdry}
\int_{B_1^+} \bigl( \Delta \mathbf{c}+f'(u)\mathbf{c}\bigr) \mathbf{c}\,\eta^{2}\, dx \leq 
\int_{B_1^+} \mathbf{c}^{2}\left|\nabla \eta\right|^{2} \, dx.
\end{equation}

\section{The boundary weighted $L^{2}$ estimate for radial derivatives}
\label{sect:bdry-weighted}

The following is the boundary analogue of the key lemma that yield an interior weighted $L^2$ bound for the radial derivative. The proof will be essentially the same as in the interior case. Note that, in the following statement, the test function $\eta$ does not necessarily vanish on $\partial^0 B_1^+$.
\begin{lemma}{\rm (\cite[Lemma 6.2]{CFRS})} \label{lem:x Du def}
Let $u\in C^\infty(\overline{B_1^+})$ be a stable solution of $-\Delta u=f(u)$ in~$B_1^+\subset\R^n$, with $u=0$ on $\partial^0  B_1^+$, for some nonlinearity $f\in C^1(\R)$. 

Then,
\begin{equation}\label{basic}
\int_{B_1^+} \Big( |\nabla u|^2\big\{(n-2)\eta + 2 x\cdot\nabla \eta \}\,\eta - 2(x\cdot \nabla u) \nabla u\cdot \nabla(\eta^2) - |x\cdot \nabla u|^2 |\nabla \eta|^2 \Big)\,dx \le 0
\end{equation}
for all Lipschitz functions $\eta$ vanishing on $\partial B_1$. As a consequence, given $\lambda >1$ and $\rho \in (0,1/\lambda)$, we have
\begin{equation}\label{bdry-notweighted}
\int_{B_\rho^+} |\nabla u|^2 \,dx \le C_\lambda \int_{A^+_{\rho,\lambda\rho}} |\nabla u|^2\,dx \qquad \text{if } n\ge 3,
\end{equation}
and
\begin{equation*}
\int_{B_\rho^+} r^{2-n}  u_r^2 \,dx \le C_\lambda \, \rho^{2-n}\int_{A^+_{\rho,\lambda\rho}} |\nabla u|^2\,dx  \qquad \text{if } 3\le n\le 9,
\end{equation*}
where $C_\lambda$ is a constant depending only on $n$ and $\lambda$. 
\end{lemma}

\begin{proof}[Proof of Lemma \ref{lem:x Du def} and of Proposition \ref{prop:1bdry}]

As in the interior case, we use $\xi=\mathbf{c}\eta =(x\cdot \nabla u)\,\eta$ as test function in the stability inequality \eqref{stabbdry}. From this, repeating the same computations of the interior case (in which now two boundary integrals on $\partial^0B_1^+$ arise, but they vanish), we deduce \eqref{basic}.

Next, let $\lambda>1$ and $\psi\in C^\infty_c (B_\lambda)$ be a nonnegative  radially nonincreasing function with $\psi\equiv 1$ in $B_{1}$. For $\rho\in (0,1/\lambda)$, set $\psi_\rho(x) := \psi(x/\rho)$. Note that $|\nabla \psi_\rho|\leq C_\lambda/\rho$ (where $C_\lambda$ depends only on $\lambda$) and that $\nabla \psi_\rho$ vanishes outside of the annulus $A_{\rho,\lambda\rho}$.
Choosing $\eta=\psi_\rho$ in \eqref{basic}, we immediately deduce
\eqref{bdry-notweighted}.

To derive the last estimate of the lemma, for $a<n$ and $\varepsilon\in(0,\rho)$ we use the Lipschitz function $\eta_\varepsilon(x):=\min\{r^{-a/2},\varepsilon^{-a/2}\}\psi_\rho(x)$ as a test function in \eqref{basic}.
Throwing away the integral $\int_{B^+_\varepsilon}(n-2)\eta_\varepsilon^2|\nabla u|^2dx$ over $B^+_\varepsilon$, we obtain
\begin{eqnarray*}
 \int_{B_{\lambda\rho}^+\setminus B_\varepsilon^+}   \Big\{(n-2-a)    |\nabla u|^2 + \Big( 2a -\frac{a^2}{4} \Big)  u_r^2  \Big\}  r^{-a} \psi_\rho^2\,dx
& &\\
& & \hspace{-8cm} \le C_{a,\lambda}\,  \rho^{-a}\int_{B_{\lambda\rho}^+\setminus B_\rho^+} |\nabla u|^2\,dx
\end{eqnarray*}
for some constant $C_{a,\lambda}$ depending only on $n$, $a$, and $\lambda$.
Choosing $a : = n-2$, since   $ 2a -\frac{a^2}{4}   = (n-2)\bigl(2-\frac{n-2}{4}\bigr)=\frac14(n-2)(10-n) >0$ for $3\le n\le 9$ we deduce
\[  
\int_{B_{\lambda\rho}^+\setminus B_\varepsilon^+} r^{2-n}  u_r^2 \, \psi_\rho^2\,dx \le C_\lambda \,\rho^{2-n}\int_{B_{\lambda\rho}^+\setminus B_\rho^+} |\nabla u|^2\,dx.
\]
Recalling that  $\psi_\rho\equiv 1$ in $B_\rho$ and
letting $\varepsilon\downarrow0$, we conclude the proof.
\end{proof}

\section{Boundary Hessian and $W^{1,2+\gamma}$ estimates}
\label{sect:boundary-higher}

The following lemma, which is based on a Pohozaev identity, will be crucial and used several times in what follows. From now on, we will assume the solution to be nonnegative and the nonlinearity to satisfy some hypotheses.

\begin{lemma}{\rm (\cite[Lemma 5.3]{CFRS})}\label{lem:auxbdry2}
Let $u\in C^\infty(\overline{B^+_1})$ be a nonnegative stable solution of $-\Delta u=f(u)$ in~$B^+_1\subset\R^n$, with $u=0$ on $\partial^0 B^+_1$.
Assume that $f\in C^1(\R)$ is nonnegative and nondecreasing.

Then,
\begin{equation*} 
\| u_\nu\|_{L^2(\partial^0 B^+_{7/8})}  \leq C\|\nabla u\|_{L^2(B^+_1)},
\end{equation*}
where $u_\nu=-u_{n}$ is the exterior normal derivative of $u$ on $\partial^0 B_1^+$ and $C$ is a dimensional constant.
\end{lemma}

\begin{proof}
Take a cut-off function  $\eta\in C^\infty_c(B_1)$ such that $\eta=1$ in $B_{7/8}$,
and consider the vector-field $\mathbf{X}(x):= x+\boldsymbol e_n$. Multiplying  the identity 
\[ 
{\rm div}\big( | \nabla u|^2\mathbf{X}- 2(\mathbf{X}\cdot \nabla u) \nabla u \big) =  (n-2) |\nabla u|^2  - 2(\mathbf{X}\cdot \nabla u) \Delta u
\]
by $\eta^2$, integrating in $B_1^+$,
and taking into account that $u_\nu^2=|\nabla u|^2$ on $\partial^0 B^+_1$ since $u=0$ on this set, we obtain
\[\begin{split}
\int_{\partial^0 B^+_{1}}  u_\nu^2 \, \eta^2\,d\mathcal H^{n-1}  - &\int_{B^+_{1}} \big( | \nabla u|^2\mathbf{X}- 2(\mathbf{X}\cdot \nabla u) \nabla u \big)\cdot \nabla\eta^2 \,dx
\\
&\qquad=\int_{B^+_{1}} \big((n-2) |\nabla u|^2  - 2(\mathbf{X}\cdot \nabla u) \Delta u\big) \eta^2 \,dx.
\end{split}\]
Since, for $F(t) : = \int_0^t f(s)ds$ we have $\mathbf{X}\cdot \nabla (F(u))= (\mathbf{X}\cdot \nabla u)  f(u)= - (\mathbf{X}\cdot \nabla u)  \Delta u$, we deduce
\[\begin{split}
\int_{\partial^0B^+_{1}}  u_\nu^2\,  \eta^2\, d\mathcal H^{n-1}  &
\le C\int_{B^+_{1}} |\nabla u|^2\,dx + 2\int_{B^+_{1}}  \mathbf{X}\cdot \nabla  (F(u)) \eta^2\,dx
\\
&=  C\int_{B^+_{1}} |\nabla u|^2\,dx - 2 \int_{B^+_{1}}  F(u) \,{\rm div}(\eta^2\mathbf{X})\,dx.
\end{split}
\]

Observe now that, since $f$ is nondecreasing, $0\leq F(t) \leq tf(t)$ for all $t \geq 0$. 
Hence, noticing that the function $g:=|{\rm div}(\eta^2\mathbf{X})|$ is Lipschitz and that $u$ and $f$ are nonnegative,
we have
\begin{align*}
- \int_{B^+_{1}}  F(u) \,{\rm div}(\eta^2\mathbf{X})\,dx &\le \int_{B^+_{1}} u\,f(u)\,g\,dx
 = -\int_{B^+_{1}} u\,\Delta u\,g\,dx\\
 & = \int_{B^+_{1}}\big(|\nabla u|^2g + u\,\nabla u\cdot \nabla g\big)\,dx \\
 &\leq C\int_{B^+_{1}} \big(u^2+|\nabla u|^2\big)\,dx.
\end{align*}
We conclude using Poincar\'e's inequality for functions $u$ vanishing on $\partial^0 B^+_{1}$.\footnote{\label{poincare-0}It is simple to establish this Poincar\'e inequality. Given any point $x$ in the half-ball, one considers the sphere centered at 0 containing $x$. Using geodesic arcs from its north pole, one joints~$x$ with a point in the equator of the sphere, where $u$ vanishes. In this way, one expresses $u(x)$ as an integral of a derivative of $u$ on the geodesic arc. One concludes by using Cauchy-Schwarz inequality on this integral, and then integrating (with the corresponding Jacobians) with respect to the other angles of $S^{n-1}$ and the distance $r$ to the origin.}
\end{proof}

The following is the boundary analogue of the Sternberg and Zumbrun~\cite{SZ} inequality. It bounds the quantity $\J$ defined in \eqref{defAAA}. Since the test function $|\nabla u|$ used in the interior case does not vanish on $\partial^0 B_1^+$, the proof in~\cite{CFRS} required a new key idea. It consists of using the function \eqref{cbdry}, properly regularized, in the boundary stability inequality.

\begin{lemma}{\rm (\cite[Step 2 of the proof of Proposition 5.2]{CFRS})}\label{lem:bdry-CalA}
Let $u\in C^\infty(\overline{B^+_{1}})$ be a nonnegative stable solution of $-\Delta u=f(u)$ in~$B^+_1\subset\R^n$, with $u=0$ on $\partial^0 B^+_1$.
Assume that $f\in C^1(\R)$ is nonnegative and nondecreasing.

Then,
\begin{equation}\label{bdry-CalA}
 \|\J\|_{L^{2}(B^+_{7/8})}  \le C \| \nabla u\|_{L^{2}(B^+_{1})}
\end{equation}
for some dimensional constant $C$.
\end{lemma}

\begin{proof}
The new idea is to use 
\begin{equation}\label{cbdry}
\mathbf{c}=|\nabla u|-u_{n},
\end{equation}
properly regularized, in the stability inequality \eqref{stabbdry}. Then, by taking $\eta$ with compact support in $B_1$,  $\mathbf{c}\eta$ vanishes on $\partial B^+_{1}$, since $|\nabla u|-u_{n}=0$ on $\partial^0 B_1^+$. 

\vspace{2mm}\noindent
{\it Step 1: We prove that
\begin{equation}\label{hwioghwoih} 
 \int_{B^+_{1/2}} \J^2\,dx \le C\int_{B^+_{3/4}}    |\nabla u| \, |D^2 u|\, dx + C \int_{B^+_{1}}  |\nabla u|^2\,dx,
\end{equation}
where $\J$ is defined by \eqref{defAAA}.}
\vspace{2mm}

For $\delta >0$, we set
\begin{equation*}
\phi_\delta(p) :=  |p|\, \chi_{\{|p|>\delta\}} + \Big( (|p|^2+\delta^2)/(2\delta) \Big) \chi_{\{|p|<\delta\}},
\end{equation*}
a convex $C^{1,1}$ regularization of the Euclidean norm. Notice that $\phi_\delta(\nabla u)\in W^{2,\infty}(B^+_{1})$, since the first derivatives of $\phi_\delta$ match on $\{|p|=\delta\}$.
Moreover, since $u$ is nonnegative and superharmonic,
the Hopf lemma yields
$|\nabla u|\ge c>0$ on $\partial^0 B^+_{3/4}$ for some constant~$c$ ---unless $u\equiv 0$, in which case there is nothing to prove. 
Hence, for $\delta>0$ small enough we have
\begin{equation}\label{inaneighborhood}
\phi_\delta(\nabla u)  = |\nabla u| \quad \mbox{in a neighborhood of }\partial^0 B^+_{3/4}\mbox{ inside }\overline{B^+_{3/4}}.
\end{equation}

Choosing $\delta>0$ small enough such that \eqref{inaneighborhood} holds, we set 
\[
\mathbf{c_\delta} := \phi_\delta(\nabla u) -u_{n}
\]
and we take $\eta\in C^\infty_c (B_{3/4})$ satisfying $\eta\equiv1$ in $B_{1/2}$. Now, since $\mathbf{c_\delta} $ vanishes on $\partial^0 B^+_{3/4}$, we are allowed to take $\xi = \mathbf{c_\delta} \eta$ as a test function in the stability inequality~\eqref{stabbdry}. We obtain that
\begin{equation}\label{ajgpen1}
\int_{B^+_{1}} \big(\Delta \mathbf{c_\delta} +f'(u)\mathbf{c_\delta} \big)\,\mathbf{c_\delta} \, \eta^2 \,dx\le \int_{B^+_{1}} \mathbf{c}_\delta^2 |\nabla \eta|^2 \,dx.
\end{equation}

Note now that, since $\Delta u_{n} + f'(u)u_{n}=0$,
\begin{equation}\label{ajgpen2}
\begin{split}
\bigl(\Delta \mathbf{c_\delta} + f'(u)\mathbf{c_\delta}\bigr) \,\mathbf{c_\delta} &=  \bigl({\Delta[\phi_\delta(\nabla u)]} +f'(u)\phi_\delta(\nabla u)\bigr) \phi_\delta(\nabla u)
\\
& \hspace{25mm}  - \big( {\Delta[\phi_\delta(\nabla u)]} +f'(u) \phi_\delta(\nabla u)\big) u_{n}.
\end{split}
\end{equation}
Next, since $\Delta \nabla  u  = -f'(u) \nabla u$, we also have
\begin{align}
\bigl({\Delta[\phi_\delta(\nabla u)]}  +f'(u)&\phi_\delta(\nabla u)\bigr)  \phi_\delta(\nabla u) 
 \nonumber\\
 & = f'(u)\phi_\delta(\nabla u)\Bigl(\phi_\delta(\nabla u)-\sum\nolimits_j (\partial_j\phi_\delta)(\nabla u)u_j\Bigr) 
 \label{may19}\\
 & \hspace{.5cm} +\,\phi_\delta(\nabla u)\sum\nolimits_{i,j,k} (\partial_{jk}^2\phi_\delta)(\nabla u)u_{ij}u_{ik}.\label{may19b}
\end{align}
Notice that, inside the set $\{|\nabla u|\leq \delta\}$, the term \eqref{may19b} is nonnegative since $\phi_\delta$ is convex, while the term \eqref{may19} is equal to $f'(u)\phi_\delta(\nabla u)(\delta-|\nabla u|^2/\delta)$
and, therefore, it is also nonnegative ---since all three factors are nonnegative.
On the other hand, inside the set $\{|\nabla u|>\delta\}$, the term \eqref{may19} vanishes, while the term \eqref{may19b} equals~$\J^2$ ---recall that $\J$ is defined by \eqref{defAAA}.
Therefore, we conclude that
\begin{equation}\label{ajgpen3}
\bigl({\Delta[\phi_\delta(\nabla u)]} +f'(u)\phi_\delta(\nabla u)\bigr) \phi_\delta(\nabla u)  \ge  \J^2 \, \chi_{\{|\nabla u|>\delta\}}.
\end{equation}

Coming back to \eqref{ajgpen2}, since  $\eta\in C^\infty_c (B_{3/4})$, {integrating by parts} and recalling  \eqref{inaneighborhood} we have
\begin{eqnarray}\label{ahoighwioh}
\hspace{-2.1mm} \int_{B^+_{1}}  {\Delta[\phi_\delta(\nabla u)]}\, u_{n} \,\eta^2\,dx &
\\
\nonumber&\hspace{-40.5mm} =  \int_{B^+_{1}}    \phi_\delta(\nabla u) \,  \Delta u_{n}  \,\eta^2\,dx 
+\int_{B^+_{1}}  \left( 2 \phi_\delta(\nabla u) \, \nabla u_{n}\cdot  \nabla (\eta^2) + \phi_\delta(\nabla u) \,u_{n} \,\Delta (\eta^2)\right)dx
\\
\nonumber&\hspace{-20mm} -\int_{\partial^0 B^+_{1}}  \left( \partial_{n} (|\nabla u|) \,u_{n}\, \eta^2 - |\nabla u|  \partial_{n} (u_{n}\, \eta^2) \right) d\mathcal H^{n-1}.
\end{eqnarray}
Now, on $\partial^0 B^+_{1}$ it holds 
 $\partial_n (|\nabla u|)\, u_{n}  = u_{nn}u_{n}=|\nabla u|  u_{nn}$ and therefore, thanks to Lemma \ref{lem:auxbdry2},
\begin{equation}\label{may19c}\begin{split}
& \bigg| \int_{\partial^0 B^+_{1}}  \left( \partial_{n} (|\nabla u|) \,u_{n}\, \eta^2 - |\nabla u|  \partial_{n} (u_{n}\, \eta^2) \right) d\mathcal H^{n-1}  \bigg| \\
 & \hspace{3cm} \le C\int_{\partial^0 B^+_{3/4}}  |u_\nu|^2\,d\mathcal H^{n-1} \le C\int_{B^+_{1}}|\nabla u|^2\, dx.
\end{split}\end{equation}
Thus, by \eqref{ahoighwioh} and \eqref{may19c}, we conclude that
\begin{align*}
& \left| \int_{B^+_{1}} \big( {\Delta[\phi_\delta(\nabla u)]} +f'(u) \phi_\delta(\nabla u)\big)u_{n}\,  \eta^2\,dx\right|\\ & \hspace{2cm}\le   C\int_{B^+_{3/4}} (|\nabla u|+\delta)\left( |D^2u|  + |\nabla u|\right)\,dx +C\int_{B^+_{1}}|\nabla u|^2\, dx.
\end{align*}

Combining this last bound with \eqref{ajgpen1},  \eqref{ajgpen2}, and \eqref{ajgpen3}, we finally obtain 
\begin{align*}
 &\int_{B^+_{1}} \J^2\eta^2 \chi_{\{|\nabla u|>\delta\}}\,dx \\ & \qquad
 \le C\int_{B^+_{3/4}} \left\{ (|\nabla u|+\delta)^2+\big( |\nabla u|+\delta \big) \left( |D^2u|  + |\nabla u|\right)\right\} \,dx+C\int_{B^+_{1}}|\nabla u|^2\, dx.
\end{align*}
Since $\eta\equiv1$ in $B_{1/2}$, by letting $\delta\downarrow 0$ we conclude \eqref{hwioghwoih}, 
as desired.

\vspace{2mm}\noindent
{\it Step 2: We show that
\[
 \int_{B^+_{{1/2}}} \J^2\,dx \le \ep  \int_{B^+_{1}} \J^2\,dx + \frac{C}{\ep}\|\nabla u\|_{L^2(B^+_{1})}^2
\]
for all $\ep\in (0,1)$.}
\vspace{2mm}

Let  $\eta \in C^\infty_c(B_{7/8})$ satisfy $\eta\equiv1$ in $B_{{3/4}}$. 
From \eqref{hwuighwiu} and 
\eqref{hwuighwiu2} ---note that these are pointwise relations that also hold for solutions in half-balls instead of full balls---, we have that
\begin{equation}\label{ahigohwiowb}
\int_{B^+_{1}}   -2 \Delta u |\nabla u|\,\eta^2\,dx \leq
-\int_{B^+_{1}} {\rm div}\big(  |\nabla u| \nabla u\big) \eta^2\,dx  + C \int_{B^+_{1}}  \J |\nabla u|\,\eta^2\,dx.
\end{equation}
This yields, since $|D^2u|\leq -\Delta u +C\J$ a.e.\ by \eqref{HessbyLapl} and the fact that $\Delta u \leq 0$, \eqref{ahigohwiowb} yields
\begin{equation}\label{ahigohwiowb2}
\int_{B^+_{1}}|\nabla u|\, |D^2u|\,\eta^2\,dx \leq 
\biggl|\frac12\int_{B^+_{1}} {\rm div}\big(  |\nabla u| \nabla u\big) \eta^2\,dx  \biggr|+C  \int_{B^+_{1}} \J\,|\nabla u|\,\eta^2\,dx.
\end{equation}
On the other hand, 
using Lemma \ref{lem:auxbdry2} we obtain
\begin{equation}\label{ahigohwiowb3}
\begin{split}
\hspace{-1mm}\biggl|\int_{B^+_{1}} {\rm div}\big(  |\nabla u| \nabla u\big) \eta^2\,dx\biggr|  &=
\biggl|-\int_{\partial^0 B^+_{1}}   (u_\nu)^2 \eta^2 \,d\mathcal H^{n-1}   - \int_{B_1^+}  |\nabla u| \nabla u \cdot \nabla (\eta^2)\,dx\biggr|\\
&\leq C \int_{B^+_{1}}   |\nabla u|^2\,dx.
\end{split}
\end{equation}
Thus, by \eqref{ahigohwiowb2}
and \eqref{ahigohwiowb3},
we see that
\begin{equation}
\label{eq:Du D2u}
\int_{B^+_{1}} |\nabla u|\, |D^2u| \,\eta^2\,dx \le C\int_{B^+_{1}} \J \,|\nabla u|\,\eta^2\,dx + C\int_{B^+_{1}} |\nabla u|^2\,dx.
\end{equation}

Recalling that $\eta\equiv1$ in $B_{{3/4}}$,
\eqref{hwioghwoih} and \eqref{eq:Du D2u} yield, for every $\ep\in(0,1)$,
\[\begin{split}
 \int_{B^+_{{1/2}}} \J^2\,dx 
&\le C \int_{B^+_{{3/4}}} |\nabla u|\, |D^2u| \,dx+C\|\nabla u\|_{L^2(B^+_{1})}^2 
\\
&\le   C \int_{B^+_{1}} \J\, |\nabla u|\,dx + C\|\nabla u\|_{L^2(B^+_{1})}^2\\
&\le \ep  \int_{B^+_{1}} \J^2\,dx + \frac{C}{\ep}\|\nabla u\|_{L^2(B^+_{1})}^2.
\end{split}\]

\vspace{2mm}\noindent
{\it Step 3: Conclusion.} 
\vspace{2mm}

We first claim that
\begin{equation}\label{angoiwnown}
\rho^2\| \J \|^2_{L^2(\R^n_+\cap B_{\rho/2}(y))}  \le  C\ep \rho^2\| \J \|^2_{L^2(\R^n_+\cap B_{\rho}(y))} +
\frac{C}{\ep}  \|\nabla u\|_{L^2(B^+_1)}^2 
\end{equation}
for every ball $B_{\rho}(y)\subset B_{1}$ and $\ep\in(0,1)$. 
Note that $B_{\rho}(y)$ is not necessarily centered at a point on $\partial \R^n_+$. The proof of the lemma will be finished once this is shown. Indeed, \eqref {angoiwnown} and Lemma \ref{lem_abstract}, applied to the subadditive quantity 
$$
\sigma(B):=\| \J \|^2_{L^2(\R^n_+ \cap B)} \quad\text{for } B\subset B_1
$$ 
(note here that the $L^2$ norm is not taken in $B$, but in its intersection with the upper half-space) lead to \eqref{bdry-CalA} with $B^+_{{7/8}}$ replaced by $B^+_{1/2}$. Finally, a standard
covering and scaling argument gives the same estimate with $B^+_{{7/8}}$ in its left-hand side.

It remains to show \eqref{angoiwnown}. To do this, it is easy to check that there is a dimensional number of balls $\{B_{\rho/16} (y_i)\}_i$ and $\{B_{3\rho/16} (z_j)\}_j$ which cover $\R^n_+\cap B_{\rho/2}(y)$ and have the following properties. Each ball $B_{\rho/16} (y_i)$ is interior, in the sense that the ball with twice its radius satisfies $B_{\rho/8} (y_i)\subset \R^n_+\cap B_{\rho}(y)\subset B_1^+$. The part of $\R^n_+\cap B_{\rho/2}(y)$ not covered by any of the previous balls is covered by the union of the half-balls $B^+_{3\rho/16} (z_j)$, where $z_j\in \partial\R^n_+$. In addition, the half-ball with twice its radius satisfies $B^+_{3\rho/8} (z_j)\subset \R^n_+\cap  B_\rho (y)\subset B^+_1$.\footnote{This can be easily seen as follows.
Take a dimensional number of points $y_k\in B_{\rho/2} (y)$ such that $B_{\rho/2} (y) \subset\cup_k B_{\rho/16} (y_k)$. Keep those balls which are interior, in the sense that $B_{\rho/8} (y_i)\subset \R^n_+$. Note that we will also have $B_{\rho/8} (y_i)\subset B_{\rho}(y)$. To cover all of $\R^n_+\cap B_{\rho/2}(y)$, it remains to consider those balls $B_{\rho/16} (y_j)$ which intersect $\R^n_+$ and with $B_{\rho/8} (y_j)\not\subset\R^n_+ $. It follows that there exists a point $z_j\in \partial\R^n_+ \cap B_{\rho/8} (y_j)$. Now, the claimed covering property holds, since $B_{\rho/16} (y_j) \subset B_{3\rho/16} (z_j)$. In addition, $B_{3\rho/8} (z_j)\subset B_\rho (y)$.}

Now, in each interior ball $B_{\rho/16} (y_i)$ we use Lemma~\ref{conseqestab}, rescaled, to deduce
$$
\rho^2\| \J \|^2_{L^2(B_{\rho/16}(y_i))}  \le C \|\nabla u\|_{L^2(B_{\rho/8}(y_i))}^2   \le \frac{C}{\ep} \|\nabla u\|_{L^2(B_1^+)}^2,
$$
since $\ep <1$. Instead, for the balls $\{B_{3\rho/16} (z_j)\}_j$ we use the statement of Step~2 above (after rescaling and a translation to make $z_j\in\partial \R^n_+$ to be the origin) to get
\[\begin{split}
\rho^2\| \J \|^2_{L^2(B^+_{3\rho/16}(z_j))}  & \le  \ep \rho^2\| \J \|^2_{L^2(B^+_{3\rho/8}(z_j))} +
\frac{C}{\ep}  \|\nabla u\|_{L^2(B^+_{3\rho/8}(z_j))}^2 \\
& \le  \ep \rho^2\| \J \|^2_{L^2(\R^n_+\cap B_{\rho}(y))} +
\frac{C}{\ep}  \|\nabla u\|_{L^2(B^+_1)}^2.
\end{split}\]
Adding all these inequalities we obtain \eqref{angoiwnown} ---notice that $\ep$ is multiplied by a dimensional constant $C$ in \eqref{angoiwnown}.
\end{proof}

The following $L^1$ estimates for the full Hessian will be useful in several occasions, not only in this section to control the gradient in $L^2$ by the function in $L^1$, but also in next section for the control of the $L^1$ norm of $u$ by its radial derivative in~$L^1$.
These estimates were not explicitly stated in \cite{CFRS}, but follow from arguments in that paper.

\begin{lemma}\label{lem:Hessbdry}
Let $u\in C^\infty(\overline{B^+_1})$ be a nonnegative stable solution of $-\Delta u=f(u)$ in~$B^+_1\subset\R^n$, with $u=0$ on $\partial^0 B^+_1$.
Assume that $f\in C^1(\R)$ is nonnegative and nondecreasing.

Then,
\begin{equation}\label{bdry-product}
\Vert\, |\nabla u| \, D^2 u \, \Vert_{L^1(B^+_{3/4})} \leq  C \Vert \nabla u \Vert^2_{L^{2}(B^+_{1})}
\end{equation}
and
\begin{equation}\label{massHessbdry}
\|D^2u\|_{L^1({B_{3/4}^+})} \leq  C \|\nabla u\|_{L^2(B^+_1)} 
\end{equation}
for some dimensional constant $C$.
\end{lemma}

\begin{proof}
We first show \eqref{bdry-product}. By taking  $\eta \in C^\infty_c(B_{7/8})$ with $\eta\equiv1$ in $B_{3/4}$, the estimate follows from \eqref{eq:Du D2u}, the Cauchy-Schwarz inequality, and Lemma~\ref{lem:bdry-CalA}.

To prove \eqref{massHessbdry}, we choose a nonnegative function $\psi\in C_c^\infty(B_{7/8})$ with $\psi\equiv 1$ in $B_{3/4}$. Since $-\Delta u \geq 0$, we see that
\begin{equation}\label{massLaplbdry}
\begin{split}
\|\Delta u\|_{L^1({B_{3/4}^+})} & \leq \int_{B_{7/8}^+} -\Delta u \,\psi\, dx=-\int_{\partial^0  B_{7/8}^+} u_\nu \psi \, d\mathcal H^{n-1}+\int_{B_{7/8}^+} \nabla u\cdot \nabla \psi\,dx \\ & \leq C \|\nabla u\|_{L^2(B^+_1)},
\end{split}
\end{equation}
where we have used Lemma \ref{lem:auxbdry2}. 

Recall now that $|D^2u|\leq |\Delta u| + C {\mathcal A}$ a.e.\ in $B_{3/4}^+$ ---since the pointwise inequality~\eqref{HessbyLapl}  also holds for solutions in half-balls instead of full balls.
Now, estimate \eqref{massHessbdry} follows from this inequality, \eqref{massLaplbdry}, and Lemma~\ref{lem:bdry-CalA}.
\end{proof}

With the weighted Hessian estimate \eqref{bdry-product} at hand, we can now control the gradient in~$L^2$ by the solution in $L^1$. By using the new interpolation inequalities of Appendix~\ref{app:interp}, we give a simpler proof of this result than in~\cite{CFRS} ---in particular we do not need to use the higher integrability $L^{2+\gamma}$ for $|\nabla u|$.

\begin{proposition}{\rm (\cite[Proposition 5.5]{CFRS})}\label{prop:bdry-nabla2byL1}
Let $u\in C^\infty(\overline{B^+_{1}})$ be a nonnegative stable solution of $-\Delta u=f(u)$ in~$B^+_1\subset\R^n$, with $u=0$ on $\partial^0 B^+_1$.
Assume that $f\in C^1(\R)$ is nonnegative and nondecreasing.

Then,
\begin{equation*}
 \|\nabla u\|_{L^{2}(B^+_{1/2})}  \le C \| u\|_{L^{1}(B^+_{1})}, 
\end{equation*}
for some dimensional constant $C$.
\end{proposition}

\begin{proof}
We cover $B^+_{1/2}$ (except for a set of measure zero) with a family of disjoint open cubes $Q_j\subset\R^n_+$ of side-length which depends only on $n$ and is small enough such that $Q_j\subset B^+_{3/4}$.
We now combine the interpolation inequalities of  Propositions~\ref{prop5.2} and~\ref{Nash} in each cube $Q_j$ (we rescale them to pass from the unit cube to the cubes~$Q_j$), used with $p=2$ and $\tilde\ep= \ep^{3/2}$ for any given $\ep\in(0,1)$.  We obtain that
\begin{equation*}
\int_{Q_j}\abs{\nabla u}^{2}dx \leq C\varepsilon \int_{Q_j}\abs{\nabla u}\lvert D^2u\rvert\,dx+ C\varepsilon \int_{Q_j}\abs{\nabla u}^2dx+C\varepsilon^{-2-\frac{3n}{2}}\left( \int_{Q_j}\abs{u}\,dx\right)^2.
\end{equation*} 
Now, using $Q_j\subset B^+_{3/4}$ and estimate 
\eqref{bdry-product} from Lemma~\ref{lem:Hessbdry},
we deduce
\begin{equation*}
\int_{Q_j}\abs{\nabla u}^{2}dx \leq C\varepsilon \int_{B^+_1}\abs{\nabla u}^2dx+C\varepsilon^{-2-\frac{3n}{2}}\left( \int_{B^+_1}\abs{u}\,dx\right)^2.
\end{equation*} 
Adding up all these inequalities (note that the number of cubes $Q_j$ depends only on~$n$), we get
\begin{equation*}
\|\nabla u\|_{L^{2}(B^+_{1/2})}^2 \le C\varepsilon  \|\nabla u\|_{L^{2}(B^+_{1})}^2 + C\varepsilon^{-2-\frac{3n}{2}} \|u\|_{L^{1}(B^+_{1})}^2.
\end{equation*}

This estimate, applied to rescaled solutions, yields
\begin{equation}\label{bdry-grad-eps}
\rho^{n+2}\int_{B^+_{\rho/2}}|\nabla u|^2\,dx \leq 
C\varepsilon \rho^{n+2}\int_{B^+_{\rho}}|\nabla u|^2\,dx
+ C\varepsilon^{-2-\frac{3n}{2}}\left(\int_{B^+_{\rho}}|u|\,dx\right)^2
\end{equation}
for all $\rho <1$ and  $\ep\in(0,1)$.

We now claim that, for all balls $B_\rho(y)\subset B_1$ (not necessarily contained in $B_1^+$) and every  $\ep\in(0,1)$, we have  
\begin{equation}\label{all-balls-grad-eps}
\rho^{n+2}\int_{\R^n_+\cap B_{\rho/2}(y)}|\nabla u|^2\,dx \leq C\varepsilon \rho^{n+2}\int_{\R^n_+\cap B_{\rho}(y)}|\nabla u|^2\,dx
+C\varepsilon^{-2-\frac{3n}{2}}\Vert u \Vert^2_{L^1(B_1^+)}.
\end{equation}
To show this, we proceed exactly as in Step~3 of the proof of Lemma~\ref{lem:bdry-CalA}, where we proved \eqref{angoiwnown}. That is, we use a covering by balls which are either centered at $\partial\R^n_+$ or interior to $\R^n_+$. For the first ones we use \eqref{bdry-grad-eps} (with $\rho$ replaced by $3\rho/8$), while for the interior ones we employ Proposition~\ref{prop:2}, rescaled, and use that $\varepsilon <1$. We conclude \eqref{all-balls-grad-eps}.

Finally, by Lemma \ref{lem_abstract} applied to the subadditive quantity
$\sigma(B):=\|\nabla u\|_{L^2(\R^n_+\cap B)}^2$, the result follows.
\end{proof}

The following higher integrability result from \cite{CFRS} is of interest by itself but, in addition, it will be needed in next section to control, following the later paper~\cite{C22quant}, the $L^1$ norm of $u$ in a half-annulus by its radial derivative in~$L^1$.

\begin{lemma}{\rm (\cite[Proposition 5.2]{CFRS})}\label{lem:bdry-gamma-2}
Let $u\in C^\infty(\overline{B^+_{1}})$ be a nonnegative stable solution of $-\Delta u=f(u)$ in~$B^+_1\subset\R^n$, with $u=0$ on $\partial^0 B^+_1$.
Assume that $f\in C^1(\R)$ is nonnegative and nondecreasing.

Then,
\begin{equation*}
 \|\nabla u\|_{L^{2+\gamma}(B^+_{3/4})}  \le C \| \nabla u\|_{L^{2}(B^+_{1})}
\end{equation*}
for some dimensional constants $\gamma>0$ and $C$.
\end{lemma}

\begin{proof}
\vspace{2mm}\noindent
{\it Step 1: We show  that, for almost every $t>0$,} 
\begin{equation}\label{bdry-level}
 \int_{\{u=t\}\cap B_{3/4}^+} |\nabla u|^2 \, d\HH^{n-1}  \leq C\int_{B^+_{1}} |\nabla u|^2\,dx.
\end{equation}

Take  $\eta \in C^\infty_c(B_{7/8})$ with $\eta\equiv1$ in $B_{{4/5}}$. We combine \eqref{hwuighwiu} and \eqref{hwuighwiu2} ---note that these are pointwise relations that also hold for solutions in half-balls instead of full balls--- with Lemma~\ref{lem:bdry-CalA} and \eqref{eq:Du D2u}, to get
\begin{align*}
\int_{B^+_{{4/5}}} \big| {\rm div}(|\nabla u|\, \nabla u) \big|  \,dx &\le C \int_{B^+_{{4/5}}} |\nabla u| \,\J\,dx +\int_{B^+_{{4/5}}} -2|\nabla u|\,\Delta u\,dx  \\
& \hspace{-3cm}\leq C\biggl(\int_{B^+_{{4/5}}} |\nabla u|^2\,dx\biggr)^{1/2} \biggl(\int_{B^+_{{4/5}}} \J^2\,dx\biggr)^{1/2}+ C\int_{B^+_{{7/8}}} \J^2\,dx+C\int_{B^+_{1}} |\nabla u|^2\,dx \\
 &\hspace{-3cm}\le C\int_{B^+_{1}} |\nabla u|^2\,dx.
\end{align*}

Now, let $\zeta \in C^\infty_c(B_{4/5})$ with $\zeta\equiv1$ in $B_{3/4}$. By Sard's theorem, $\{u=t\}$ is a smooth hypersurface for almost every $t>0$. For such values of $t$, noting that $\partial( \{u>t\} \cap B_{4/5}^+)$ does not intersect $\partial^0 B_{4/5}^+$ and using the previous bound we deduce
\begin{equation*}
\begin{split}
 \int_{\{u=t\}\cap B_{3/4}^+} |\nabla u|^2 \, d\HH^{n-1} & \\  & \hspace{-2cm} \leq \int_{\{u=t\}\cap B_{4/5}^+} |\nabla u|^2\zeta^2 \,d\HH^{n-1}
=- \int_{\{u>t\} \cap B_{4/5}^+}  {\rm div}\big(|\nabla u|\, \nabla u \,\zeta^2\big) \,dx \\
 & \hspace{-2cm} \le  C\int_{B^+_{1}} |\nabla u|^2\,dx,
\end{split}
\end{equation*}
as claimed.

\vspace{2mm}\noindent
{\it Step 2: Conclusion.}
\vspace{2mm}

Here it is convenient to assume $\|\nabla u\|_{L^2(B^+_{1})}=1$, which can be achieved  after multiplying $u$ by a constant. 

Setting $h(t)=\max\{1,t\}$, by the Sobolev embedding for functions vanishing on $\partial^0 B^+_{1}$,\footnote{\label{Sob+Poi}This Sobolev inequality follows combining the Sobolev embedding \cite[Theorem 7.26]{GT} applied to $u$ with the Poincar\'e inequality with exponent 2 proved in Footnote~\ref{poincare-0}.} 
we see that
\begin{equation}\label{ashgowobb}
\begin{split}
\int_{\R^+} dt & \int_{\{u=t\}\cap  B^+_{1}\cap \{|\nabla u|\neq0\}} d\HH^{n-1} h(t)^p \,|\nabla u|^{-1} \\
& \qquad\qquad \leq |B^+_{1}\cap \{u<1\}|+ \int_{B^+_{1}} u^p  \, dx\leq C
\end{split}
\end{equation}
for some $p>2$.
Hence, choosing dimensional constants  $q>1$ and $\theta\in (0,1/3)$ such that $p/q = (1-\theta)/\theta$, we see that
\begin{equation*}
\begin{split}
\int_{B^+_{3/4}} |\nabla u|^{3-3\theta}\,dx  &=\int_{\R^+} dt\int_{\{u=t\}\cap B^+_{3/4}\cap \{|\nabla u|\neq0\}} d\HH^{n-1} h(t)^{p\theta -q(1-\theta)}  |\nabla u|^{-\theta + 2(1-\theta)} 
\\
&\le   \left(\int_{\R^+} dt\int_{\{u=t\}\cap B^+_1\cap \{|\nabla u|\neq0\}}d\HH^{n-1}   h(t)^{p}  |\nabla u|^{-1} \right)^\theta 
\\
&\qquad\qquad \cdot \bigg(\int_{\R^+} dt\, h(t)^{-q}  \int_{\{u=t\}\cap B^+_{3/4}}d\HH^{n-1}   |\nabla u|^2   \bigg)^{1-\theta}.
\end{split}
\end{equation*}
 By \eqref{ashgowobb} and \eqref{bdry-level}, this yields
\[
\int_{B^+_{3/4}} |\nabla u|^{3-3\theta}  \,dx \le C,
\]
which concludes the proof.
\end{proof}

We can now give the

\begin{proof}[Proof of Proposition \ref{prop:2bdry}]
By rescaling Lemma \ref{lem:bdry-gamma-2},  we deduce that 
$$
\|\nabla u\|_{L^{2+\gamma}(B^+_{1/4})}  \le  \|\nabla u\|_{L^{2+\gamma}(B^+_{3/8})}  \le C \| \nabla u\|_{L^{2}(B^+_{1/2})}.
$$
 This and Proposition~\ref{prop:bdry-nabla2byL1} conclude the proof.
\end{proof}

For future use, note that putting together Proposition~\ref{prop:bdry-nabla2byL1} with \eqref{massHessbdry} (rescaled), we obtain that
\begin{equation}\label{last-gamma}
\|D^2 u\|_{L^{1}(B^+_{1/4})}  \le  \|D^2 u\|_{L^{1}(B^+_{3/8})} \le  C\|\nabla u\|_{L^{2}(B^+_{1/2})}  \le C \| u\|_{L^{1}(B^+_{1})}.
\end{equation}

When proving Proposition \ref{prop:3bdry} in next section, we will need the following bounds in half-annuli. They follow from a simple covering and scaling argument combined with the boundary estimates that we have just proven and their interior analogues (i.e., those in balls whose double is contained in $\R^n_+$).

\begin{corollary}\label{corol:Deltabdry}
Let $u\in C^\infty(\overline{B^+_1})$ be a nonnegative stable solution of $-\Delta u=f(u)$ in~$B^+_1\subset\R^n$, with $u=0$ on $\partial^0 B^+_1$.
Assume that $f\in C^1(\R)$ is nonnegative and nondecreasing. Let $0<\rho_1<\rho_2<\rho_3<\rho_4\leq 1$.

Then,
\begin{equation} \label{prop5.2CFRS-ann}
\|\nabla u\|_{L^{2+\gamma}(A^+_{\rho_2,\rho_3})}
\le C_{\rho_i} \,  \|u\|_{L^{1}(A^+_{\rho_1,\rho_4})}
\end{equation}
and
\begin{equation} \label{hess-ann}
\|D^2 u\|_{L^{1}(A^+_{\rho_2,\rho_3})}
\le C_{\rho_i} \,  \|u\|_{L^{1}(A^+_{\rho_1,\rho_4})}
\end{equation}
for some dimensional constant $\gamma>0$ and some constant $C_{\rho_i}$ depending only on $n$, $\rho_1$, $\rho_2$, $\rho_3$, and $\rho_4$.
\end{corollary}

\begin{proof}
We first cover $\partial^0 A^+_{\rho_2,\rho_3}$ by a finite number of balls $B_j$ centered at $\{x_n=0\}$ and with a sufficiently small radius~$\delta_0$ such that the balls $4B_j$ (with same centers as the previous ones, and with radius $4\delta_0$) are contained in $A_{\rho_1,\rho_4}$. We now cover $\bar{A^+_{\rho_2,\rho_3}}\setminus \cup_j B_j$ by balls $B_k$ of  a smaller radius $\delta_1$ in such a way that $4 B_k\subset A^+_{\rho_1,\rho_4}$. In this way we obtain a covering of $A^+_{\rho_2,\rho_3}$ made of half-balls $B_j^+$ (centered at $\{x_n=0\}$) and interior balls $B_k$, with $4B_j^+$ and $4B_k$ all contained in $A^+_{\rho_1,\rho_4}$.

Now, to deduce our two estimates, in each half-ball $B_j^+$ we use  Proposition~\ref{prop:2bdry} and \eqref{last-gamma} (both rescaled). Instead, in the interior balls $B_k$ we apply \eqref{eq:W12 L1 int} as well as \eqref{estdivbis} (both rescaled). In this way we control the norms of the quantities on the left-hand sides of \eqref{prop5.2CFRS-ann} and of \eqref{hess-ann} but now integrated in any of the previous balls,  by $C_{\rho_i} \, \|u\|_{L^{1}(A^+_{\rho_1,\rho_4})}$ (for each of the balls). Adding all the inequalities completes the proof.
\end{proof}

\section{The radial derivative controls the function in $L^{1}$ up to the boundary}
\label{sect:boundary-L1radial}

In this section we establish Proposition \ref{prop:3bdry}. This will be much more delicate than the interior case of Proposition~\ref{prop:3}, which holds assuming only that the function~$u$ is superharmonic. Instead, in the boundary setting, superharmonicity is not enough for the proposition to hold, as the following remark shows. Hence, within the proof we will need to use the semilinear equation satisfied by $u$ and, in fact, also the stability of $u$.

\begin{remark}\label{rk:not-superh}
For $n\ge 2$, the estimate \eqref{introbdryradial} of Proposition \ref{prop:3bdry}, controlling the function in~$L^1$  in a half-annulus by its radial derivative in $L^1$ (or even in $L^\infty$), also in a half-annulus, cannot not hold within the class of nonnegative superharmonic functions which are smooth in $\overline{\R^n_+}$ and vanish on $\partial \R^n_+$. 

Indeed, for $\delta \in (0,1)$, consider
$$
u^\delta (x):= \frac{x_n}{|(x',x_n+\delta)|} \qquad\text{ for } x=(x',x_n)\in\R^{n-1}\times\R, \, x_n\geq 0.
$$
The function $u^\delta$ is nonnegative and smooth in $\overline{\R^n_+}$, vanishes on $\{x_n=0\}$, and, as a simple computation shows,\footnote{One may start from $\Delta u^\delta = x_n \Delta \varphi + 2 \varphi_{x_n}$, where $\varphi=(|x|^2+ 2\delta x_n+\delta^2)^{-1/2}$.} is superharmonic in~$\R^n_+$ for $n\geq 2$. At the same time, in the half-annulus it satisfies
$$
|\partial_r u^\delta| = \frac{x_n}{r\, |(x',x_n+\delta)|^3}\, \delta (x_n+\delta) \le \frac{1}{r\,r^{3} }\, \delta (1+\delta)
\le C\delta \quad\text{ in } A^+_{1/2,1}
$$ 
for some constant $C$ independent of $\delta$. By taking $\delta$~small enough, this shows that the estimate  \eqref{introbdryradial} of Proposition \ref{prop:3bdry} cannot hold within this class of functions. 

Note also that in the limiting case $\delta=0$, we are exhibiting a nonnegative  superharmonic function $u^0=x_n/r$ which belongs to $W^{1,2}(B_1^+)$ for $n\geq 3$, vanishes a.e.\ on $\{x_n=0\}$, and is zero homogeneous (i.e., $\partial_r u^0\equiv 0$). 
\end{remark}

To control the $L^1$ norm of a stable solution in a half-annulus by its radial derivative in $L^1$, the starting idea is to use the equation
\begin{equation}\label{continuousbdry}
-2 \Delta u + \Delta (x\cdot\nabla u)  = -f'(u) \, x\cdot\nabla u,
\end{equation}
after multiplying it against a cut-off function and integrating it in the half-annulus. We will see that this easily yields a lower bound for the integral of the left-hand side which is appropriate for our purposes. The difficulty is how to control the integral of the right-hand side, by above, in terms of only $x\cdot\nabla u= ru_r$. As we will explain later in Footnote~\ref{foot-xgrad}, we wish to use the stability condition to deal with the factor $f'(u)$ in \eqref{continuousbdry}. However, through the simplest approach this would force to control the $L^2$ norm of $\nabla (x\cdot\nabla u)$   ---which is not at our hands since we do not have $L^2$ control on the full Hessian of $u$.

To proceed in a similar manner but reducing the number of derivatives falling on~$u$, for $\lambda >0$ we consider the functions 
$$
u_\lambda(x):=u(\lambda x)
$$
and note that $\frac{d}{d\la} u_\la (x)= x \cdot \nabla u (\lambda x)= \lambda^{-1} x \cdot \nabla u_\lambda (x)$.  Using this twice, and noticing that $\lambda^{-2} \Delta u_\lambda = -f(u_\lambda)$, we deduce
\begin{equation}\label{expression}
\begin{split}
-2\lambda^{-3}\Delta u_{\lambda} + \lambda^{-2} \Delta ( \la^{-1} x\cdot \nabla u_{\lambda}) &
= \frac{d}{d\lambda} \left( \lambda^{-2} \Delta u_\lambda\right) \\
& = - \frac{d}{d\lambda} \, f( u_\lambda )= -f'(u_\lambda)\  \lambda^{-1} x\cdot\nabla u_\lambda.
\end{split}
\end{equation}
Evaluating the first and last expressions at $\la=1$, we recover \eqref{continuousbdry}. However, it will be crucial (see the later Footnote~\ref{foot-xgrad}) to use instead the equality between the first and third expressions, after integrating them not only in $x$, but also in~$\lambda$. Integrating in $\lambda$ will be essential to reduce the number of derivatives falling on~$u$. In addition, the monotonicity and convexity of $f$ will allow to use the stability condition appropriately.

This will be understood in all detail going through the following proof. 

\begin{proof}[Proof of Proposition \ref{prop:3bdry}]
We follow \cite{C22quant}. By rescaling, we may suppose that we have a stable solution in $B_6^+$ instead of $B_1^+$.

We choose a nonnegative smooth function $\zeta$ with compact support in the full annulus $A_{4,5}$ and such that $\zeta\equiv 1$ in $A_{4.1,4.9}$. Then, the function $\xi:= x_n \zeta$ satisfies
\begin{equation*}
\begin{split}
 &\xi\ge 0 \text{ in } A_{4,5}^+,  \quad \xi=0 \text{ on }\partial^0 A_{4,5}^+, \\
 & \xi=\xi_{\,\nu} = 0 \text{ on } \partial^+ A_{4,5}^+, \quad \text{and}\quad  \xi=x_n \text{ in } A_{4.1,4.9}^+. 
\end{split}
\end{equation*}

The proof starts from the identity
\begin{equation}\label{derla0-1}
\begin{split}
2\lambda^{-3}\int_{A_{4,5}^+} -\Delta u_{\lambda}\, \xi\, dx + \lambda ^{-2}\int_{A_{4,5}^+}  \Delta ( \la^{-1} x\cdot \nabla u_{\lambda}) \, \xi\, dx &\\
&\hspace{-3cm}= - \frac{d}{d\lambda}  \int_{A_{4,5}^+} f(u_\lambda) \, \xi\, dx,
\end{split}
\end{equation}
which follows from the equality between the first and third expressions in \eqref{expression}.
In a first step, we will bound the left-hand side of \eqref{derla0-1} by below. But the subtle part of the proof, where we use the stability of the solution, will be the second step. It will bound the right-hand side of \eqref{derla0-1} by above, but only after averaging it in~$\la$. For this, we will use that
\begin{equation}\label{derla0-2}
\int_1^{1.1} \big( - \frac{d}{d\lambda} \int_{A_{4,5}^+}    f( u_\lambda)\, \xi\, dx \big) d\lambda 
= \int_{A_{4,5}^+} \big( f(u)-f(u_{1.1}) \big) \xi\, dx ,
\end{equation}
where
$$
u_{1.1}:= u_{11/10}=u\big( (11/10) \cdot\big).
$$

\vspace{2mm}\noindent
{\it Step 1: We prove that}
\begin{equation}\label{bdrystep1}
\begin{split}
2\la^{-3}\int_{A_{4,5}^+} -\Delta u_{\la} \, \xi\, dx + \la ^{-2}\int_{A_{4,5}^+}  \Delta ( \la^{-1} x\cdot \nabla u_{\la}) \, \xi\, dx &\\
& \hspace{-3cm}\ge c \| u\|_{L^1(A_{4.7,4.8}^+)} - C \|u_r\|_{L^1({A_{3,6}^+})}
\end{split}
\end{equation}
{\it for every $\la\in [1,1.1]$, where $c$ and $C$ are positive dimensional constants.}
\vspace{2mm}

To bound by below the second integral in \eqref{bdrystep1} is easy. Since $\xi$ and $\xi_{\,\nu}$ vanish on $\partial^+A^+_{4,5}$, and both $\xi$ and $x\cdot \nabla u_{\la}$ vanish on $\partial^0 A^+_{4,5}$, we see that
\begin{equation}\label{2ndint}
\int_{A_{4,5}^+}  \Delta ( x\cdot \nabla u_{\la}) \, \xi\, dx =\int_{A_{4,5}^+}  x\cdot \nabla u_{\la}\, \Delta \xi\, dx \ge -C \|u_r\|_{L^1({A_{4,5.5}^+})}
\end{equation}
since $x\cdot \nabla u_{\la}(x)= \la x\cdot \nabla u(\la x)= \la r \, u_r(\la x)$.

Next, to control the first integral in \eqref{bdrystep1}, given any $\rho_1\in (4.1,4.2)$ and $\rho_2\in (4.8,4.9)$ we consider the solution $\varphi$ of 
\begin{equation}\label{torsion}
\begin{cases} -\Delta \varphi=1 \quad  \quad &\mbox{in } A_{\rho_1,\rho_2}^+
\\
\varphi=0 &\mbox{on  } \partial^0 A_{\rho_1,\rho_2}^+
\\
\varphi_\nu=0 &\mbox{on  } \partial^+ A_{\rho_1,\rho_2}^+.
\end{cases}
\end{equation}
Note that $\varphi\ge 0$ by the maximum principle. In addition, as shown in Appendix~\ref{app:neumann}, we have $|\nabla \varphi|\le C$ in $ A_{\rho_1,\rho_2}^+$ for some dimensional constant $C$. This bound yields $c\varphi\le x_n=\xi$ in $A_{\rho_1,\rho_2}^+$ for some small dimensional constant $c>0$. Since, in addition, $-\Delta u_{\la}$, $\xi$, and $\varphi$ are all nonnegative, it follows that (for positive constants $c$ and $C$ that may differ from line to line)
 \begin{eqnarray*}
\int_{A_{4,5}^+} -\Delta u_{\la} \, \xi\, dx &\geq& c\int_{A_{\rho_1,\rho_2}^+} -\Delta u_{\la}\, \varphi \, dx
\\ & &\hspace{-2cm} = -c \int_{\partial^+ A_{\rho_1,\rho_2}^+} (u_{\la})_\nu \,\varphi \, d\HH^{n-1} 
+c\int_{A_{\rho_1,\rho_2}^+} u_{\la}\, dx
\\ & &\hspace{-2cm} \geq - C \int_{\partial^+ B_{\rho_1}^+} |(u_{\la})_r| \, d\HH^{n-1}  -C \int_{\partial^+ B_{\rho_2}^+} |(u_{\la})_r| \, d\HH^{n-1} 
+c\| u\|_{L^1(A_{\la\rho_1,\la\rho_2}^+)}
\\ & &\hspace{-2cm} \geq - C \int_{\partial^+ B_{\rho_1}^+} |(u_{\la})_r| \, d\HH^{n-1}  -C \int_{\partial^+ B_{\rho_2}^+} |(u_{\la})_r| \, d\HH^{n-1} 
+c\| u\|_{L^1(A_{4.7,4.8}^+)}
\end{eqnarray*}
since $\la\rho_1\le 1.1\cdot 4.2\le 4.7$ and $4.8\le \la\rho_2$. Finally, integrating first in $\rho_1\in (4.1,4.2)$ and then in $\rho_2\in (4.8,4.9)$, we arrive at
\begin{equation*}
 \int_{A_{4,5}^+} -\Delta u_{\la} \, \xi\, dx 
\ge -C \|u_r\|_{L^1({A_{4.1,6}^+})}
+c\| u\|_{L^1(A_{4.7,4.8}^+)},
\end{equation*}
since $1.1\cdot 4.9\le 6$.

This and \eqref{2ndint} establish the claim of Step 1.

\vspace{2mm}\noindent
{\it Step 2: We prove that, for every $\ep\in (0,1)$,}
\begin{equation}\label{bdrystep2}
 \int_{A_{4,5}^+} \big( f(u)-f( u_{1.1}) \big) \xi\, dx\le C\big( \varepsilon  \|u\|_{L^{1}(A^+_{3,6})}+ \varepsilon^{-1-2\frac{2+\gamma}{\gamma}} \|u_r\|_{L^1({A_{3,6}^+})}\big)
\end{equation}
{\it for some dimensional constants $\gamma>0$ and $C$ ---with $\gamma$ being the exponent in Corollary~\ref{corol:Deltabdry}. }
\vspace{2mm}

By convexity of $f$ we have $f(u)-f( u_{1.1}) \leq f'(u)(u- u_{1.1})$. We now use that $\xi=0$ on $\partial A_{4,5}^+$ and that $u- u_{1.1}=0$ on $\partial^0 A_{3.9,5.1}^+$ in order to take advantage, twice, of the stability of $u$. Taking a function $\phi\in C^\infty_c(A_{3.9,5.1})$ with $\phi=1$ in $A_{4,5}$, and since $\xi$ and $f'(u)$ are nonnegative, we deduce
\begin{align}
 \hspace{0cm} \int_{A_{4,5}^+} \big( f(u)-f( u_{1.1}) \big) \xi\, dx & \leq  \int_{A_{4,5}^+} f'(u)(u-u_{1.1}) \xi\, dx \label{int-fsxi1}
\\ & \hspace{-3cm} 
\leq  \left( \int_{A_{4,5}^+} f'(u)\xi^2\, dx \right)^{1/2}\left( \int_{A_{4,5}^+} f'(u) (u- u_{1.1})^2\, dx \right)^{1/2} \nonumber
\\ & \hspace{-3cm} 
\leq  \left( \int_{A_{4,5}^+} |\nabla \xi|^2\, dx \right)^{1/2} \left( \int_{A_{3.9,5.1}^+} f'(u) \big((u-
 u_{1.1})\phi\big)^2\, dx \right)^{1/2}
\nonumber\\ & \hspace{-3cm} 
\leq  C\left( \int_{A_{3.9,5.1}^+} \left|\nabla \big( (u- u_{1.1})\phi\big)\right|^2\, dx \right)^{1/2}
\nonumber\\ & \hspace{-3cm} 
\leq C\Vert \nabla(u- u_{1.1}) \Vert_{L^2(A_{3.9,5.1}^+)}\label{int-fsxi3},
 \end{align}
where in the last bound we have used Poincar\'e's inequality in $A_{3.9,5.1}^+$ for functions vanishing on $\partial^0 A_{3.9,5.1}^+$ (see Footnote~\ref{poincare-0} for its simple proof). The previous chain of inequalities, which are a crucial part of the proof, have used the stability of the solution twice.\footnote{\label{foot-xgrad}Note first that a simple Cauchy-Schwarz argument to control the right-hand side of \eqref{int-fsxi1} would not work, since we do not have control on the integral of $f'(u)^2$. On the other hand, we would be in trouble if we had performed the above chain of inequalities starting from \eqref{continuousbdry} instead of \eqref{expression}. Indeed, in such case, $f'(u)(u- u_{1.1}) \xi $  in \eqref{int-fsxi1}   would be replaced by $f'(u)\,x\cdot\nabla u\,\xi$. Hence, proceeding as we have done above, the final quantity appearing in \eqref{int-fsxi3} would be the $L^2$ norm of $\nabla (x\cdot\nabla u)$. But recall that we do not have control on the $L^2$ norm of the full Hessian of $u$.}
 
Next, by the $W^{1,2+\gamma}$ estimate \eqref{prop5.2CFRS-ann} (rescaled to hold in $B_6^+$ instead of $B_1^+$), taking $q:=\frac{2(1+\gamma)}{2+\gamma}$, and applying H\"older's inequality to $\int (u-u_{1.1})^{(2+\gamma)/(1+\gamma)}(u-u_{1.1})^{\gamma/(1+\gamma)}$ with exponents $1+\gamma$ and $(1+\gamma)/\gamma$, we obtain 
\begin{align} \label{int-fsxibis}
\|\nabla (u- u_{1.1})\|_{L^{2}(A^+_{3.9,5.1})}& \le \|\nabla (u- u_{1.1})\|_{L^{2+\gamma}(A^+_{3.9,5.1})}^{\frac{1}{q}}\|\nabla (u- u_{1.1})\|_{L^{1}(A^+_{3.9,5.1})}^{\frac{1}{q'}}\nonumber\\
&\hspace{-1.5cm} \le C \|\nabla u\|_{L^{2+\gamma}(A^+_{3.9,5.1\cdot 1.1})}^{\frac{1}{q}}\|\nabla(u- u_{1.1})\|_{L^{1}(A^+_{3.9,5.1})}^{\frac{1}{q'}}  \nonumber \\
&\hspace{-1.5cm} \le C   \|u\|_{L^{1}(A^+_{3,6})}^{\frac{1}{q}}  \|\nabla(u- u_{1.1})\|_{L^{1}(A^+_{3.9,5.1})}^{\frac{1}{q'}}\nonumber\\
&\hspace{-1.5cm} \le   \ep \|u\|_{L^{1}(A^+_{3,6})}
+ C\ep^{-\frac{q'}{q}} \|\nabla(u- u_{1.1})\|_{L^{1}(A^+_{3.9,5.1})} 
 \end{align}
for all $\ep\in (0,1)$.
Now, by the interpolation inequality in cubes of Proposition \ref{prop5.2} applied with $p=1$, we claim that
\begin{align}\label{frominterp}
\|{\nabla (u- u_{1.1})}\|_{L^1({A^+_{3.9,5.1}})}&\nonumber\\ 
&\hspace{-2.5cm} \leq C\varepsilon^{1+\frac{q'}{q}} \|{D^2(u- u_{1.1})}\|_{L^1({A^+_{3.8,5.2}})}+ C \varepsilon^{-1-\frac{q'}{q}}\|u- u_{1.1}\|_{L^1({A^+_{3.8,5.2}})}\nonumber\\ 
&\hspace{-2.5cm} \leq C\varepsilon^{1+\frac{q'}{q}} \|{D^2u}\|_{L^1({A^+_{3.8,5.8}})}+ C \varepsilon^{-1-\frac{q'}{q}}\|u- u_{1.1}\|_{L^1({A^+_{3.8,5.2}})}.
\end{align}
To see the first inequality, one covers the half-annulus $A^+_{3.9,5.1}$ (except for a set of measure zero) by disjoint cubes of sufficiently small side-length to be contained in $A^+_{3.8,5.2}$. One then applies Proposition \ref{prop5.2} with $p=1$ in each of these cubes, after rescaling it and renaming~$\ep$, and finally one adds up all the inequalities ---exactly as we did in the beginning of the proof of Proposition \ref{prop:2} in Section~\ref{sect:W12}.

From \eqref{frominterp} and  the Hessian bound \eqref{hess-ann} (rescaled to hold in $B_6^+$ instead of~$B_1^+$) we conclude
\begin{equation*}
\|{\nabla (u- u_{1.1})}\|_{L^1({A^+_{3.9,5.1}})}\leq  \hspace{-.8mm}C\varepsilon^{1+\frac{q'}{q}} \|u\|_{L^{1}(A^+_{3,6})}+ C\varepsilon^{-1-\frac{q'}{q}} \|u- u_{1.1}\|_{L^1({A^+_{3.8,5.2}})}.
\end{equation*}

Putting together this last bound with \eqref{int-fsxi3} and \eqref{int-fsxibis},  and using again \eqref{hess-ann}, we arrive at 
\begin{equation*}
 \int_{A_{4,5}^+} \big( f(u)-f( u_{1.1}) \big) \xi\, dx\le C\varepsilon  \|u\|_{L^{1}(A^+_{3,6})}+ C\varepsilon^{-1-2\frac{2+\gamma}{\gamma}}\|u- u_{1.1}\|_{L^1({A^+_{3.8,5.2}})}.
\end{equation*}
At the same time, since $\frac{d}{d\lambda} u_\lambda (x)=x\cdot\nabla u (\lambda x)= r\, u_r(\lambda x)$, we have
\begin{eqnarray*}
\|u- u_{1.1}\|_{L^1({A_{3.8,5.2}^+})} &=& \int_{A_{3.8,5.2}^+} dx \left| \int_1^{1.1} d\lambda \ r\, u_r(\lambda x)\right| \\
& \le & C\int_1^{1.1}  d\lambda \int_{A_{3.8,5.2\cdot 1.1}^+} dy \ |u_r(y)| \le C\|u_r\|_{L^1({A_{3,6}^+})}.
\end{eqnarray*}
The last two bounds establish Step 2.

\vspace{2mm}\noindent
{\it Step 3: Conclusion.}
\vspace{2mm}

We integrate \eqref{derla0-1} in $\la\in (1,1.1)$ and use \eqref{derla0-2}, as well as \eqref{bdrystep1} and \eqref{bdrystep2} in the statements of Steps 1 and 2.
We obtain the bound
\begin{equation}\label{finalbdry}
\| u\|_{L^1(A_{4.7,4.8}^+)} \le C\big( \varepsilon  \|u\|_{L^{1}(A^+_{3,6})}+ \varepsilon^{-1-2\frac{2+\gamma}{\gamma}}\|u_r\|_{L^1({A_{3,6}^+})} \big)
\end{equation}
for all $\ep\in (0,1)$. It is now simple to conclude the desired estimate
\begin{equation}\label{finalB6}
\| u\|_{L^1(A_{3,6}^+)} \le C \|u_r\|_{L^1({A_{3,6}^+})}.
\end{equation}

Indeed, we first show that
\begin{equation}\label{annulusB3}
\| u\|_{L^1(A_{3,6}^+)} \le C  \big( \|u\|_{L^{1}(A^+_{4.7,4.8})}+  \|u_r\|_{L^1({A_{3,6}^+})} \big).
\end{equation}
To prove this, use that 
\begin{equation}\label{alongr}
u(s\sigma)=u(t\sigma)-\int_s^t u_r (r\sigma) \,dr
\end{equation}
for $s\in (3,6)$, $t\in (4.7,4.8)$, and $\sigma\in S^{n-1}$. We deduce
$$
(s/6)^{n-1}|u(s\sigma)|\leq (t/4)^{n-1} |u(t\sigma)|+\int_{3}^6 (r/3)^{n-1} |u_r (r\sigma)| \,dr.
$$
Integrating in $\sigma\in S^{n-1}$, and then in $s\in (3,6)$ and in $t\in (4.7,4.8)$, we conclude ~\eqref{annulusB3}.

Now, we use  \eqref{annulusB3} to bound the right-hand side of  \eqref{finalbdry}. In the resulting inequality we choose $\ep\in (0,1)$ small enough so that the constant $C\ep$ multiplying $ \|u\|_{L^{1}(A^+_{4.7,4.8})}$ satisfies $C\ep\le1/2$. We deduce
$$
\| u\|_{L^1(A_{4.7,4.8}^+)} \le C \|u_r\|_{L^1({A_{3,6}^+})},
$$
that, together with \eqref{annulusB3}, yields \eqref{finalB6}. 

Next, to show $\| u\|_{L^1(B_{6}^+)} \le C \|u_r\|_{L^1({B_{6}^+})}$, by \eqref{finalB6} it suffices to show $\| u\|_{L^1(B_{3}^+)} \le C \|u_r\|_{L^1({B_{6}^+})}$. Using again \eqref{finalB6}, it is indeed enough to show 
$$
\| u\|_{L^1(B_{3}^+)} \le C \left(\| u\|_{L^1(A_{3,6}^+)} + \|u_r\|_{L^1({B_{6}^+})}\right).
$$
This is easily shown using \eqref{alongr}, with a similar argument as above or as in the interior case of Proposition~\ref{prop:3}.

The proven bounds hold for a solution in $B_6^+$. By rescaling we conclude the estimates of the proposition for solutions in $B_1^+$.
\end{proof}

For $n\geq 3$, the estimate that we have just proven can be improved as follows.

\begin{remark}\label{rk:bdry-ngeq3}
When $n\geq3$, Proposition \ref{prop:3bdry} also holds when replacing $\| u\|_{L^1(A_{1/2,1}^+)}$, in the left-hand side of~\eqref{introbdryradial}, by $\| u\|_{L^1(B_{1}^+)}$. 

To show this, and given the proposition, it suffices to prove  
$$
\| u\|_{L^1(B_{1/2}^+)} \leq C \| u\|_{L^1(A_{1/2,1}^+)}.
$$ 
But notice that, by the Poincar\' e inequality for functions vanishing on $\partial^0 B^+_{1/2}$ (see Footnote~\ref{poincare-0} for its simple proof), we have that 
$$
\| u\|_{L^1(B_{1/2}^+)}\le C\|\nabla u\|_{L^{2}(B^+_{1/2})}\le C\|\nabla u\|_{L^{2}(B^+_{4/6})}.
$$
Now, using \eqref{bdry-notweighted} (here we need that $n\geq 3$) we can bound this last quantity by $C\|\nabla u\|_{L^{2}(A^+_{4/6,5/6})}$. Finally, by  \eqref{prop5.2CFRS-ann} we know that
$$
\|\nabla u\|_{L^{2}(A^+_{4/6,5/6})}
\le C   \|u\|_{L^{1}(A^+_{1/2,1})}.
$$ 
\end{remark}

\section{Boundary $C^\alpha$ estimate \vspace{.15cm}}
\label{sect:boundary-conclusion}

We have now all the ingredients to give the

\begin{proof}[Proof of Theorem \ref{thm:0bdry}]
The proof of \eqref{higher-bdry} follows from Proposition~\ref{prop:2bdry} and a standard covering and scaling argument.

To prove the H\"older estimate \eqref{holder-bdry},
we may assume (as in the interior case; see the beginning of the proof of Theorem~\ref{thm:0} in Section~\ref{sect:Holder})  that $3\le n\le 9$ by adding superfluous variables. Note that the ``half-ball geometry'' in the statement of the theorem allows, when adding variables, to deduce the result for $n=1$ and $n=2$ from the one for $n=3$; here we may use ``half-cubes'' contained (respectively, containing) each of the two half-balls together with a scaling and covering argument.

Using Proposition~\ref{prop:1bdry} with $\rho=3/5$ and $\lambda=4/3$, we obtain
\begin{equation}\label{eq:11proofbdry}
\int_{B^+_{1/2}}r^{2-n} u_{r}^2\,dx \leq \int_{B^+_{3/5}}r^{2-n} u_{r}^2\,dx 
\leq C\int_{A^+_{3/5,4/5}}|\nabla u|^2\,dx.
\end{equation}
Now, since $1/2<3/5<4/5<1$, \eqref{prop5.2CFRS-ann} in Corollary~\ref{corol:Deltabdry} yields
\begin{equation}\label{eq:12proofbdry}
\int_{A^+_{3/5,4/5}}|\nabla u|^2\,dx\leq  C\Vert u\Vert_{L^1(A^+_{1/2,1})}^2.
\end{equation}
At the same time, by the estimate \eqref{introbdryradial} from Proposition \ref{prop:3bdry} we have
\begin{equation*}
\Vert u\Vert_{L^1(A^+_{1/2,1})}^2\le C\Vert u_r\Vert_{L^1(A^+_{1/2,1})}^2 \le 
 C\Vert u_r\Vert_{L^2(A^+_{1/2,1})}^2\le
C\int_{A^+_{1/2,1}} r^{2-n} u_r^2\,dx.
\end{equation*}

This, together with \eqref{eq:11proofbdry} and \eqref{eq:12proofbdry}, leads to
$$
\int_{B^+_{1/2}}r^{2-n} u_{r}^2\,dx \leq C\int_{A^+_{1/2,1}} r^{2-n} u_r^2\,dx.
$$
We can write this inequality, in an equivalent way, as
$$
\int_{B^+_{1/2}}r^{2-n} u_{r}^2\,dx \leq  \theta \int_{B^+_{1}} r^{2-n} u_r^2\,dx
$$ 
for the dimensional constant $\theta= \frac{C}{1+C}\in (0,1)$.

This estimate (which is rescale invariant), applied to the rescaled stable solutions $u(\rho\, \cdot)$, yields
$$
\int_{B^+_{\rho/2}}r^{2-n} u_{r}^2\,dx \leq \theta\int_{B^+_{\rho}} r^{2-n} u_r^2\,dx
$$ 
for all $\rho \in (0,1)$.
This inequality can be iterated, in balls of radius $2^{-j}$, with $j\ge 1$ an integer, to obtain
$\int_{B^+_{2^{-j}}}r^{2-n} u_{r}^2\,dx \leq \theta^{j-1} \int_{B^+_{1/2}}r^{2-n} u_{r}^2\,dx$.
Since $0<\theta<1$, it follows that, for some dimensional $\alpha\in (0,1)$,
$$
\int_{B^+_\rho} r^{2-n}u_r^2\,dx\leq C\rho^{2\alpha}\int_{B^+_{1/2}} r^{2-n} u_r^2\,dx \le C\rho^{2\alpha}\Vert u\Vert_{L^1(B^+_{1})}^2 \qquad\text{ for all } \rho \leq 1/2,
$$
where we have used \eqref{eq:11proofbdry} and \eqref{eq:12proofbdry}
in the last inequality. In particular, since $\rho^{2-n}\leq r^{2-n}$ in $B^+_\rho$, we conclude that
$$
\int_{B_\rho^+} |u_r|\,dx  \le C\rho^{n-1+\alpha}\Vert u\Vert_{L^1(B^+_{1})}
\qquad\text{ for all } \rho \leq 1/2.
$$

Next, given $y'\in \partial^0 B^+_{1/2}$, we can apply the last estimate (with $\rho$ replaced by $2\rho$) to the function $u_{y'}(x):=u(y'+\frac{x}{2})$, defined for $x\in B_{1}^+$, since $y'+\frac{1}{2}B_{1}^+\subset B_{1}^+$. Using the notation
$u_{r_{y'}}(x):= |x-y'|^{-1}(x-y')\cdot \nabla u (x)$, we get
\begin{equation*}
\int_{B^+_\rho(y')}|u_{r_{y'}}|\,dx\leq C\Vert u\Vert_{L^1(B_1^+)}\,\rho^{n-1+\alpha}\quad \mbox{ for all }\,y' \in \partial^0 B^+_{1/2}\text{ and } \rho \leq 1/4.
\end{equation*}
At this point we can reflect $u$ oddly to get a $C^1$ function in all $B_1$, which we still call $u$ and which satisfies the previous estimate but now with the integral of $|u_{r_{y'}}|$ extended to the whole ball $B_\rho(y')$. Thus, given $0<y_n\leq 1/8$, we can apply  Lemma~\ref{lemmaC1}  in $B_{2y_n}(y')$ (that is, with $y$ replaced by $y'$ and $d=2y_n$) and with $S:=B^+_{2y_n}(y')$ to get (recall that $u\geq 0$)
\begin{equation}\label{bdrycontrolaver}
\frac{1}{(y_n)^n}\Vert u\Vert_{L^1({B^+_{2y_n}(y')})}= C|u_S|=C|u(y')-u_S|\leq C  \Vert u\Vert_{L^1(B_1^+)}\,(y_n)^{\alpha}.
\end{equation}

We are now ready to control the H\"older norm of $u$ up to the boundary. Given $y=(y',y_n)\in B^+_{1/8}$,  by the interior H\"older estimate \eqref{eq:Ca L1 int} of Theorem~\ref{thm:0} (applied in the interior ball $B_{y_n}(y)$ and properly rescaled) and by \eqref{bdrycontrolaver}, we have 
\begin{equation}\label{bdry-interior}
\begin{split}
\Vert u\Vert_{L^\infty(B_{y_n/2}(y))}+ (y_n)^\alpha  [ u]_{C^\alpha(\overline{B}_{y_n/2}(y))}
&\leq \frac{C}{(y_n)^n}  \Vert u\Vert_{L^1(B_{y_n}(y))}\\
& \hspace{-1cm} \leq \frac{C}{(y_n)^n}   \Vert u\Vert_{L^1(B_{2 y_n}^+(y'))}\leq C  \Vert u\Vert_{L^1(B_1^+)}(y_n)^{\alpha}.
\end{split}
\end{equation}
In particular, $|u(y)|\leq C  \Vert u\Vert_{L^1(B_1^+)}$ and thus we have controlled the $L^\infty$ norm of $u$ in~$B^+_{1/8}$. 

Finally, to bound $|u(z)-u(y)|$ for another given point $z\in B^+_{1/8}$, note that, by symmetry, we may assume $z_n\leq y_n$. Now, if $|z-y|\leq y_n/2$ then from \eqref{bdry-interior} we get $|u(z)-u(y)|/|z-y|^\alpha \leq C  \Vert u\Vert_{L^1(B_1^+)}$. Instead, if $|z-y|>y_n/2$ we apply the $L^\infty$ estimate in \eqref{bdry-interior} (also with the point $y$ replaced by $z$) to deduce
\begin{equation*}
\begin{split}
|u(z)-u(y)|
&\leq u(z)+u(y) \leq   C  \Vert u\Vert_{L^1(B_1^+)} \left( (z_n)^{\alpha} + (y_n)^{\alpha}\right) \\
&\leq  C  \Vert u\Vert_{L^1(B_1^+)} (y_n)^{\alpha} \leq C  \Vert u\Vert_{L^1(B_1^+)} |z-y|^{\alpha}.
\end{split}
\end{equation*}

We have thus controlled the $C^\alpha$-seminorm of $u$ in $B^+_{1/8}$ for every  stable solution in $B_1^+$. From this, the H\"older estimate \eqref{holder-bdry} of Theorem \ref{thm:0bdry} follows by a standard covering and scaling argument.
\end{proof}

\appendix

\medskip\bigskip
\centerline{\textsc{ \large Appendices}}
\addtocontents{toc}{\textsc{\hspace{.2cm}  Appendices \vspace{.1cm}}}

\section{Two simple interpolation inequalities in cubes}
\label{app:interp}

In contrast with \cite{CFRS}, in this paper we have not used the $W^{1,2+\gamma}$ bound to prove the interior $W^{1,2}$ regularity result. This has been accomplished thanks to the following interpolation inequality in cubes, which has been used also in our boundary regularity proof. Thus, we have used the interpolation inequality twice ---with $p=2$ in the interior case and with $p=1$ in the boundary case. 

The inequality concerns the quantity $\int |\nabla u|^{p-1}|D^2 u|\, dx$. It is stated in cubes, it does not assume any boundary values for the function~$u$, and it is extremely simple to be established. Indeed, in contrast with other interpolation results for functions with no prescribed boundary values, such as \cite[Theorem~7.28]{GT}, proving our inequality in a cube of~$\R^n$ will be immediate once we prove it in dimension one.\footnote{Instead, the interpolation inequality \cite[Theorem~7.28]{GT} for functions with arbitrary boundary values requires a result on Sobolev seminorms for extension operators (i.e., operators extending functions of~$n$~variables to a larger domain).}

We conceived the inequality recently, in our work \cite[Proposition A.3]{CMS}, in order to extend the results of \cite{CFRS} to the $(p+1)$-Laplacian operator.\footnote{The proof that we give here in cubes is exactly the same as that of the last arXiv version of~\cite[arXiv version]{CMS}. Instead, \cite[printed version]{CMS} claimed the same interpolation inequality in balls instead of cubes, and this made the proof (and perhaps the statement) in~\cite[printed version]{CMS} not to be correct. Indeed, proving inequality (A.5) of \cite[printed version]{CMS} contains a mistake: it fails when, given $\varepsilon \in (0,1)$, one has $|I|<<\ep$ ---as it occurs later in the proof in $\R^n$ when some 1d sections of the ball are very small compared to $\ep$. Despite this, all other results of~\cite[printed version]{CMS} remain correct since, within Step 1 of the proof of Theorem~1.1 in \cite[printed version]{CMS},  the interpolation inequality may be applied in small enough cubes covering the ball $B_{1/2}$ (as we do in the current paper) instead of applying it in the whole ball.} We do not know if the interpolation inequality has appeared somewhere before.

\begin{proposition}{\rm (\cite[Proposition A.2]{C22quant})}\label{prop5.2}
Let $Q=(0,1)^n\subset \R^n$, $p\geq 1$, and $u\in C^{2}(\overline Q)$.

Then, for every $\varepsilon\in (0,1)$,
\begin{equation}\label{5.2}
\int_{Q}\abs{\nabla u}^{p}dx \leq C_p\left( \varepsilon \int_{Q}\abs{\nabla u}^{p-1}\lvert D^2u\rvert\,dx+   \frac{1}{\ep^p} \int_{Q}\abs{u}^{p}dx\right),
\end{equation}
where $C_p$ is a constant depending only on $n$ and $p$.
\end{proposition}

\begin{proof}
We first prove it for $n=1$, in the interval $(0,\delta)$ for a given $\delta\in (0,1)$. Let $u\in C^2([0,\delta])$. Let $x_0\in[0,\delta]$ be such that $\abs{u'(x_0)}= \min_{[0,\delta]}\abs{u'}$.
For $0<y<\frac\delta3<\frac{2\delta}3<z<\delta$, since $(u(z)-u(y))/(z-y)$ is equal to $u'$ at some point, we deduce that $\abs{u'(x_0)}\leq 3\delta^{-1}(|u(y)|+|u(z)|)$. 
Integrating this inequality first in $y$ and later in $z$, we see that
$\abs{u'(x_0)}\leq 9\delta^{-2}\int_{0}^{\delta}\abs{u}\,dx.$
Raising this inequality to the power $p\in [1,\infty)$ we get
\begin{equation}\label{5.2_1}
\abs{u'(x_0)}^{p}\leq 9^{p} \delta^{-p-1}\int_{0}^{\delta}\abs{u}^{p}\,dx.
\end{equation}
Now, for $x\in(0,\delta)$, integrating $\left(\abs{u'}^{p}\right)'$ in the interval with end points $x_0$ and $x$, we deduce
\[
\abs{u'(x)}^{p} \leq p\int_{0}^\delta\abs{u'}^{p-1}\abs{u''}\,dx+\abs{u'(x_0)}^{p}.
\] 
Combining this inequality with \eqref{5.2_1} and integrating in $x\in(0,\delta)$, we obtain
\begin{equation}\label{5.2claim2}
\int_{0}^{\delta}\abs{u'}^{p}\,dx \leq p\,\delta\int_{0}^\delta\abs{u'}^{p-1}\abs{u''}\,dx+9^{p}\delta^{-p}\int_{0}^{\delta}\abs{u}^{p}\,dx.
\end{equation}

We come back now to the statement of the proposition in dimension one and let $u\in C^2([0,1])$. For any given integer $k>1$ we divide $(0,1)$ into $k$ disjoint intervals of length $\delta=1/k$. Since for \eqref{5.2claim2} we did not require any specific boundary values for $u$, we can use the inequality in each of these intervals of length $\delta=1/k$ (instead of in $(0,\delta)$), and then add up all the inequalities, to deduce
\begin{equation}\label{5.2claim3}
\int_{0}^1\abs{u'}^{p}\,dx \leq  \frac{p}{k}\int_{0}^1\abs{u'}^{p-1}\abs{u''}\,dx+(9k)^{p}\int_{0}^1\abs{u}^{p}\,dx.
\end{equation}
Since $0<\ep<1$, there exists an integer $k>1$ such that $\frac{1}{\ep}\leq k<\frac{2}{\ep}$. This and \eqref{5.2claim3} establish the proposition in dimension one.

Finally, for $u\in C^{2}([0,1]^n)$,  denote $x=(x_1,x')\in\R\times\R^{n-1}$. Using \eqref{5.2} with $n=1$ for every $x'$, we get
\begin{equation*}
\begin{split}
\int_{Q}\abs{u_{x_1}}^{p}\,dx 
&=\int_{(0,1)^{n-1}}\,dx'\int_0^1\,dx_1\abs{u_{x_1}(x)}^{p}
\\ 
&\leq C_p\, \varepsilon\int_{(0,1)^{n-1}}\,dx'\int_0^1\,dx_1\abs{u_{x_1}(x)}^{p-1}\abs{u_{x_1x_1}(x)} \\
&\hspace{2cm} +C_p\, \varepsilon^{-p}\int_{(0,1)^{n-1}}\,dx'\int_0^1\,dx_1\abs{u(x)}^{p}
\\
&= C_p\left(\varepsilon\int_{Q}\abs{u_{x_1}(x)}^{p-1}\abs{u_{x_1x_1}(x)}\,dx+ \varepsilon^{-p}\int_{Q}\abs{u(x)}^{p}\,dx\right).
\end{split}
\end{equation*}
Since the same inequality holds for the partial derivatives with respect to each variable $x_k$ instead of $x_1$, adding up the inequalities we conclude \eqref{5.2}. 
\end{proof}

We also use Nash's interpolation inequality, which will allow us to replace the $L^2$ norm of $u$ by the square of its $L^1$ norm in \eqref{5.2}, when $p=2$. We provide an elementary proof of it.

\begin{proposition}\label{Nash}
Let $Q=(0,1)^n\subset \R^n$, $p\ge 1$,  and $u\in C^{2}(\overline Q)$.

Then, for every $\tilde\ep\in (0,1)$,
\begin{equation}\label{5.2bis}
\int_{Q}|u|^{p}\, dx \leq C_p\left( \tilde\ep^{\;p} \int_{Q}\abs{\nabla u}^{p}dx+  \frac{1}{{\tilde\ep}^{\; n(p-1)}}\left( \int_{Q}\abs{u}dx\right)^p\, \right),
\end{equation}
where $C_p$ is a constant depending only on $n$ and $p$.
\end{proposition}

\begin{proof}
Following \cite[Proposition 12.29]{Le}, we first give a very simple proof of Poincar\'e's inequality in a cube:
\begin{equation}\label{poincare}
\int_{Q}|u-u_Q|^{p}\, dx \leq n^p \int_{Q}\abs{\nabla u}^{p}dx,
\end{equation}
where $u_Q$ is the average of $u$ in $Q$. To establish it, for $x$ and $y$ in $Q$, notice that
\begin{eqnarray*}
|u(x)-u(y)| &\leq & |u(x)-u(x_1,\ldots,x_{n-1},y_n)|+\cdots+ |u(x_1, y_2,\ldots,y_n)-u(y)|\\
&\leq & \sum_{i=1}^n \int_0^1 |u_{x_i}(x_1,\ldots, x_{i-1}, t, y_{i+1},\ldots, y_n)|\, dt.
\end{eqnarray*}
Using H\"older's inequality we obtain
\begin{eqnarray*}
|u(x)-u(y)|^p &&\\
& & \hspace{-2.5cm} \leq \left( \sum_{i=1}^n \left( \int_0^1 |\nabla u(x_1,\ldots, x_{i-1}, t, y_{i+1},\ldots, y_n)|^p\, dt\right)^{1/p} \right)^p\\
&&  \hspace{-2.5cm} \leq n^{p-1} \sum_{i=1}^n \int_0^1 |\nabla u(x_1,\ldots, x_{i-1}, t, y_{i+1},\ldots, y_n)|^p\, dt.
\end{eqnarray*}
We conclude Poincar\' e's inequality, as follows:
\begin{eqnarray*}
\int_Q |u(x)-u_Q|^p\, dx &=&\int_Q \left| \int_Q  (u(x)-u(y)) \, dy\right|^p dx \\
&& \hspace{-2.8cm}  \le \int_Q \int_Q  |u(x)-u(y)|^p \, dy\, dx \\
& & \hspace{-2.8cm}  \le n^{p-1} \int_Q \int_Q \sum_{i=1}^n \int_0^1 |\nabla u(x_1,\ldots, x_{i-1}, t, y_{i+1},\ldots, y_n)|^p\, dt\, dy\, dx\\
& & \hspace{-2.8cm}  =n^{p-1} \sum_{i=1}^n \int_Q |\nabla u(x_1,\ldots, x_{i-1}, t, y_{i+1},\ldots, y_n)|^p\, dx_1\, dx_{i-1}\,dt\, dy_{i+1}\ldots dy_n.
\end{eqnarray*}

Now, rescaling \eqref{poincare}, we see that Poincar\'e's inequality for functions in a cube $Q_\delta$ of side-length $\delta$ reads as $\|u-u_{Q_\delta}\|_{L^p (Q_\delta)} \leq n\delta\|\nabla u\|_{L^p (Q_\delta)}$.
By the triangular inequality we deduce 
$$
\|u\|_{L^p (Q_\delta)} \leq n\delta \|\nabla u\|_{L^p (Q_\delta)}+ |Q_\delta|^{1/p}|u_{Q_\delta}|= n\delta \|\nabla u\|_{L^p (Q_\delta)}+ \delta^{-n/p'}\int_{Q_\delta}\abs{u}dx.
$$
Now, for any given integer $k>1$ we divide $Q$ into $k^n$ disjoint cubes $Q_j$ of side-length $\delta=1/k$. 
By the previous inequality we have
$$
\int_{Q_j}|u|^{p}\, dx \leq C_p\left( \frac{1}{k^p} \int_{Q_j}\abs{\nabla u}^{p}dx+  k^{n(p-1)}\left( \int_{Q_j}\abs{u}dx\right)^p\, \right)
$$
on each cube $Q_j$.
Adding up all these inequalities and taking into account that $\sum_j (\int_{Q_j}\abs{u}dx)^p \le ( \sum_j \int_{Q_j}\abs{u}dx)^p= ( \int_{Q}\abs{u}dx)^p$, we deduce \eqref{5.2bis} by choosing an integer $k$ such that $\tilde\ep^{\,-1}\le k< 2 \tilde\ep^{\,-1}$.
\end{proof}

\section{Absorbing errors in larger balls}
\label{app:errors-balls}

The following is a general ``abstract'' lemma  due to L. Simon \cite{Simon}.
We include its proof for completeness. For brevity of the proof we take $\beta\geq 0$ in its statement, although the lemma also holds for all $\beta\in\R$.

\begin{lemma}{\rm (\cite[Lemma, page 398]{Simon})}\label{lem_abstract}
Let $\beta\geq 0$ and $\bar C>0$. Let  $\mathcal B$ be the class of all open balls $B$ contained in the unit ball $B_1$ of $\R^n$ and let $\sigma :\mathcal B \rightarrow [0,+\infty)$ satisfy the following subadditivity property:
\[ \sigma(B)\le \sum_{j=1}^N \sigma(B^j) \quad \mbox{ whenever }  N\in\Z^+, \{B^j\}_{j=1}^N \subset \mathcal B, \text{ and } B \subset \bigcup_{j=1}^N B^j. \]

It follows that there exists a constant $\delta>0$, which depends only on $n$ and $\beta$, such that if
\begin{equation}\label{hp-lem}
 \rho^\beta \sigma\bigl(B_{\rho/2}(y)\bigr) \le \delta \rho^\beta \sigma\bigl(B_\rho(y)\bigr)+ \bar C\quad \mbox{whenever }B_\rho(y)\subset B_1,
 \end{equation}
then
\[ \sigma(B_{1/2}) \le C_\beta \bar C\]
for some constant $C_\beta$ which depends only on $n$ and $\beta$.
\end{lemma}

\begin{proof}
Define
\[
S:= \sup_{B_\rho(y)\subset B_{1}} \rho^\beta \sigma\bigl(B_{\rho/2}(y)\bigr).
\]
We clearly have that $S<\infty$ by subadditivity, since $\beta\geq 0$ and $\sigma(B_1)<\infty$. 

Consider a finite covering $B_{1/2} \subset \bigcup_{i=1}^M B_{1/8}(x_i)$ of $B_{1/2}$ by $M$ balls (with $M$ depending only on $n$) of radius $1/8$ centered at points of $x_i\in B_{1/2}$.

Now, given a ball $B_\rho(y)$ contained in $B_{1}$, we have that
$$
B_{\rho/2} (y)\subset \bigcup_{i=1}^M B_{\rho/8}(y+\rho x_i)
$$
and $B_{\rho/2}(y+\rho x_i) \subset B_{\rho}(y)\subset B_{1}$. Hence, using the subadditivity of $\sigma$, assumption~\eqref{hp-lem}, and the definition of $S$ we deduce
\[
\begin{split}
 \rho^\beta \sigma\bigl(B_{\rho/2}(y)\bigr)& \le    \sum_{i=1}^M  \rho^\beta \sigma\bigl(B_{\rho/8}(y+\rho x_i)\bigr)\\
& \le  4^\beta \sum_{i=1}^M \big\{  \delta  (\rho/4)^\beta   \sigma\big(B_{\rho/4}(y+\rho x_i)\big)+\bar C\big\} \\
&= 2^\beta\delta \sum_{i=1}^M   (\rho/2)^\beta   \sigma\big(B_{\rho/4}(y+\rho x_i)\big)+4^\beta M\bar C\\
& \le  2^\beta \delta M S +  4^\beta M\bar C.
\end{split}
\]

Thus, taking supremum for all balls $B_\rho(y)\subset B_{1}$ in the left-hand side we obtain 
\[S \le 2^\beta\delta M S + 4^\beta M\bar C.\]
Choosing $\delta := 1/(2^{1+\beta}M)$ we deduce $S/2 \le 4^\beta M\bar C$. This clearly leads to the desired bound on $\sigma(B_{1/2})$.
\end{proof}

\section{Morrey's embedding theorem stated with radial derivatives}
\label{app:morrey}

Here we give all details of the proof of Morrey's estimate (see \cite[Section 7.9]{GT}), stated using radial derivatives instead of the usual full gradient ---since it is useful to use this (new) version in our applications. We start with the key lemma. As in Section~\ref{sect:Holder}, we use the notation
$$
u_{r_y}(x):= \frac{x-y}{|x-y|}\cdot \nabla u (x).
$$

\begin{lemma}\label{lemmaC1}
Let $y\in\R^n$, $d>0$, and $u$ be a $C^1$ function in $\overline B_{d}(y)\subset \R^n$. Assume that, for some positive constants~$\alpha$ and $\bar C$,
\begin{equation}\label{radial4}
\int_{B_\rho(y)}|u_{r_y}|\,dx\leq \bar C\,\rho^{n-1+\alpha}\quad \mbox{ for all }\rho \leq d.
\end{equation}

Let $S$ be any measurable subset of $B_{d}(y)$ and $u_S:= \frac{1}{|S|} \int_S u\, dx$.
Then,
$$
|u(y)-u_S| \leq C_\alpha \bar C \, \frac{d^{\, n}}{|S|}\, d^{\,\alpha}
$$
for some constant $C_\alpha$ depending only on $n$ and $\alpha$.
\end{lemma}

\begin{proof}
We follow \cite[Chapter 7]{GT}. Given $z\in S$, we have 
\begin{eqnarray*}
u(z)-u(y) &=& \int_0^{|z-y|}\frac{z -y}{|z-y|} \cdot \nabla u \left( y+\rho\frac{z -y}{|z-y|} \right)\,d\rho\\
&=&\int_0^{|z-y|}u_{r_y} \left(y+\rho\frac{z -y}{|z-y|}\right)\,d\rho.
\end{eqnarray*}
Averaging in $z\in S$, taking absolute values, and using spherical coordinates $z=\bar\rho\,\omega$ centered at~$y$ (hence, $\bar\rho=|z-y|$), we deduce
\begin{eqnarray*}
|u(y)-u_{S}| &=& \left| \frac{1}{|S|} \int_{S} dz \int_0^{|z-y|} d\rho\,\, u_{r_y} \left(y+\rho\frac{z -y}{|z-y|}\right) \right|
\\
&\leq & \frac{1}{|S|} \int_{B_{d}(y)} dz \int_0^{d} d\rho\,\, \left| u_{r_y} \left(y+\rho\frac{z -y}{|z-y|}\right) \right|\\
&= & \frac{1}{|S|} \int_0^{d} d\bar\rho \,\, \bar\rho^{\, n-1} \int_{S^{n-1}} d\omega \int_0^{d} d\rho\,\, | u_{r_y} (y+\rho\omega)|
 \\
&= & \frac{d^{\,n}}{n|S|}\int_{S^{n-1}} d\omega \int_0^{d} d\rho\,\, | u_{r_y} (y+\rho\omega)|
\\
&= & \frac{d^{\,n}}{n|S|} \int_{B_{d}(y)} r_y^{1-n} | u_{r_y}|\, dx.
\end{eqnarray*}

We now control this last integral. For this, we define
$$
\varphi(\rho):=\int_{B_\rho(y)} | u_{r_y}|\, dx
$$
and recall that $\varphi'(\rho)=\int_{\partial B_\rho(y)} | u_{r_y}|\, d{\mathcal H}^{n-1}$. We now use \eqref{radial4} to get
\begin{eqnarray*}
 \int_{B_{d}(y)} r_y^{1-n} | u_{r_y}|\, dx &=& \int_0^{d} \rho^{1-n} \varphi'(\rho)\, d\rho = d^{\,{1-n}}\varphi(d)
 +(n-1) \int_0^{d} \rho^{-n} \varphi(\rho)\, d\rho \\
 &\leq & \bar Cd^{\,{1-n}} d^{\,n-1+\alpha}
 +(n-1) \bar C \int_0^{d} \rho^{-n} \rho^{n-1+\alpha}\, d\rho \leq C_\alpha \bar C d^{\,\alpha},
\end{eqnarray*}
which concludes the proof.
\end{proof}

\begin{theorem}\label{thmC2}
Let $u$ be a $C^1$ function in $B_{1}\subset \R^n$ such that, for some positive constants~$\alpha$ and $\bar C$,
\begin{equation*}
\int_{B_\rho(y)}|u_{r_y}|\,dx\leq \bar C\,\rho^{n-1+\alpha}\quad \mbox{ for all }y \in \bar B_{1/4} \text{ and } \rho \leq 1/2.
\end{equation*}

Then,
$$
\|u\|_{C^\alpha(\overline B_{1/4})}\leq C_\alpha \left( \bar C + \Vert u\Vert_{L^1(B_{1})} \right)
$$
for some constant $C_\alpha$ depending only on $n$ and $\alpha$.
\end{theorem}

\begin{proof}
Given $y\in \overline B_{1/4}$ and $\bar y\in \overline B_{1/4}$, define
$$
d:=|y-\bar y|\le \frac{1}{2} \quad\text{and} \quad S:=B_{d}(y)\cap B_d(\bar y).
$$ 
We use Lemma~\ref{lemmaC1} both in $B_d(y)$ and in $B_d(\bar y)$ (hence, the second time with $y$ replaced by $\bar y$) and, since $|S|=c(n) d^{\,n}$ for some positive dimensional constant $c(n)$, we conclude that $|u(y)-u(\bar y)|\le |u(y)-u_S|+|u_S-u(\bar y)|\leq C_\alpha \bar C d^{\,\alpha}= C_\alpha \bar C |y-\bar y|^\alpha$. We have thus controlled the $C^\alpha$-seminorm of $u$ in $\overline B_{1/4}$. 

To bound the $L^\infty$ norm, given $y\in \overline B_{1/4}$ we choose a point $\bar y\in \overline B_{1/4}$ such that $d:=|y-\bar y|\geq 1/4$. We now use the previous bounds to obtain $|u(y)|\le |u(y)-u_S|+|u_S|\leq C_\alpha \bar C d^{\,\alpha} + C d^{-n} \Vert u\Vert_{L^1(B_{1})}$. Since $1/4\leq d\leq 1/2$, this concludes the proof.
\end{proof}

\section{Gradient estimate for a mixed boundary value problem}
\label{app:neumann}

Here we follow \cite[Appendix B]{C22quant}. We show that the solution to \eqref{torsion}, i.e.,
\begin{equation}\label{torsion-app}
\begin{cases} -\Delta \varphi=1 \quad  \quad &\mbox{in } A_{\rho_1,\rho_2}^+
\\
\varphi=0 &\mbox{on  } \partial^0 A_{\rho_1,\rho_2}^+
\\
\varphi_\nu=0 &\mbox{on  } \partial^+ A_{\rho_1,\rho_2}^+,
\end{cases}
\end{equation}
with $\rho_1\in (4.1,4.2)$ and $\rho_2\in (4.8,4.9)$ ---that we used in Step 1 of the proof of Proposition \ref{prop:3bdry}--- satisfies $|\nabla \varphi|\le C$ in $A_{\rho_1,\rho_2}^+$ for some dimensional constant $C$. Note that the equation is posed in a nonsmooth domain and that this is a mixed Dirichlet-Neumann problem. 

To establish this gradient estimate, we first need to control the $L^\infty$ norm of $\varphi$. For this, recall that, by the maximum principle, $\varphi\ge 0$. Multiplying the equation in~\eqref{torsion-app} by $\varphi^\beta$, with $\beta\ge 1$, and integrating by parts, we get 
\begin{equation}\label{moser}
\int_{A_{\rho_1,\rho_2}^+} |\nabla (\varphi^{(\beta+1)/2})|^2\, dx = \frac{(\beta+1)^2}{4\beta} \int_{A_{\rho_1,\rho_2}^+} \varphi^\beta\, dx.
\end{equation}
We now proceed with a standard iteration procedure, as in the proof of~\cite[Theorem~8.15]{GT}.  To bound by below the left-hand side of \eqref{moser}, we use the Sobolev inequality for functions vanishing on $\partial^0 A_{\rho_1,\rho_2}^+$ (see Footnote~\ref{Sob+Poi}). Since $\beta +1\geq\beta$, H\"older's inequality yields $\Vert\varphi\Vert_{L^{\chi\cdot\beta}}\leq (C\beta)^{1/\beta}\Vert\varphi\Vert_{L^{\beta}}$ for some dimensional exponent $\chi>1$. Iterating this bound, as in the proof of~\cite[Theorem~8.15]{GT}, we control the $L^\infty$ bound of $\varphi$ by $C\Vert\varphi\Vert_{L^1}$. Finally, notice that
$\Vert\nabla\varphi\Vert_{L^2}^2=\Vert\varphi\Vert_{L^1}$ by \eqref{moser} and that $\Vert\varphi\Vert_{L^1}\leq C \Vert\nabla\varphi\Vert_{L^2}$ by Poincar\'e's inequality for functions vanishing on $\partial^0 A_{\rho_1,\rho_2}^+$ (see Footnote~\ref{poincare-0}). It follows that $\Vert\varphi\Vert_{L^\infty}\leq C\Vert\varphi\Vert_{L^1}\leq C$ for
a dimensional constant $C$.

Now, a simple way to prove the gradient estimate consists of considering the function $\tilde\varphi:=\varphi + x_n^2/2$, which is harmonic and vanishes on the flat boundary. Reflecting it oddly, it becomes a harmonic function in the full annulus $A_{\rho_1,\rho_2}$ (a smooth domain) and has $\tilde\varphi_r= |x_n|x_n/r$ as Neumann data on its boundary. Since this data is a $C^{1,\alpha}$ function on $\partial A_{\rho_1,\rho_2}$, we can now apply global Schauder's regularity theory for the Neumann problem. We use Theorems 6.30 and 6.31 in \cite{GT} to get a $C^{2,\alpha}$ estimate for $\tilde\varphi$ up to the boundary ---and, in particular, a global gradient estimate. The $C^{2,\alpha}$ bound depends only on the $C^{1,\alpha}$ norm of the Neumann data, the $L^\infty$ norm of $\tilde\varphi$ (that we already controlled), and the smoothness of the annulus $A_{\rho_1,\rho_2}$ (which is uniform in $\rho_1$ and $\rho_2$ due to the constraints we placed on these parameters). 

\section{Singular solutions satisfying the weighted radial derivative estimate}
\label{app:example}

In this appendix we answer the following non obvious Real Analysis question that remained open for some years. The question originated when discovering Proposition~\ref{prop:1} around 2013.  We use the notation
$$
u_{r_y}(x):= \frac{x-y}{|x-y|}\cdot \nabla u (x).
$$

\vspace{4pt}

\noindent {\bf Question.} 
{\it For $n\geq 2$, given a function $u$ defined in some ball of $\R^n$ centered at the origin, assume that the bounds
\begin{equation}\label{eq:11}
\int_{B_\rho(y)}r_y^{2-n} u_{r_y}^2\,dx \leq C\rho^{2-n}\int_{B_{3\rho/2}(y)\setminus B_\rho(y)}|\nabla u|^2\,dx < +\infty,
\end{equation} 
hold for all $y$ in a neighborhood of the origin and all $\rho$ small enough. Do they lead to an~$L^\infty$ bound for~$u$ around the origin?}

\vspace{4pt}

The answer is clearly affirmative for $n=2$. Indeed, in this dimension the weight $r_y^{2-n}$ in the left-hand side of \eqref{eq:11} is identically 1 and one can average the inequality with respect to $y$ (taking $y$ and $\rho$ in appropriate sets, as we did in Step 2 of the proof of Theorem 1.2 in \cite{CFRS}) to get the estimate
$$
\int_{B_\rho}|\nabla u|^2\,dx \leq C\int_{B_{\lambda\rho}\setminus B_\rho}|\nabla u|^2\,dx
$$
for all $\rho$ small enough and for some $\lambda >1$. Now, as in previous sections, this (when used with balls centered at any other given point instead of the origin) and the hole-filling technique lead to the power decay in $\rho$ of the previous quantity. From this, one gets interior H\"older regularity by Theorem~\ref{thmC2}, using the Cauchy-Schwarz inequality and that $n=2$.

Instead, for $n\ge 3$, we (as well as some other experts) thought that estimate \eqref{eq:11} could perhaps suffice to get interior boundedness.  Here we show that this is not the case. Note first that we cannot take $u$ to be a positive radially decreasing function of $n$~variables. Indeed, in such case, \eqref{eq:11} used with $y=0$ can be iterated (as in previous sections) since $|\nabla u|^2=u_r^2$, leading to power decay in $\rho$ for \eqref{eq:11}. From this, one gets the
boundedness of $u$ around the origin, by Lemma~\ref{lemmaC1} and the Cauchy-Schwarz inequality. 

The following is the precise fact that we prove and the class of functions that provide a counterexample. 

\vspace{4pt}

\noindent {\bf Claim.} {\em For $n\geq 3$, there exist $W^{1,2}$ weak solutions $u$  of equations of the form $-\Delta u=f(u)$ in a ball of $\R^n$ around $0$, with $f$ smooth, nonnegative, increasing, and convex, that satisfy \eqref{eq:11} but  are unbounded in any neighborhood of the origin. In view of Theorem~\ref{thm:0} (and the approximation procedure described in Remark~\ref{monsters}), such solutions are unstable in dimensions $3\leq n\leq 9$. 

Some concrete examples are provided by certain functions $u$ of two Euclidean variables, but placed in $\R^n$, such as 
\begin{equation}\label{counter-exp}
u(x)=u(x_1,\ldots,x_n)=\log \left(-\log \ (x_1^2+x_2^2)^{1/2}\right),\qquad x\in B_{1/e} \subset \R^n, n\geq 3.
\end{equation}
A larger class of unbounded functions for which  \eqref{eq:11} holds are those satisfying \eqref{gen-counter} below ---these might not be solutions of semilinear equations, though.
}

\vspace{4pt}

One can easily check that the function~$u$ in \eqref{counter-exp} is an unbounded $W^{1,2}(B_{1/e})$ positive weak solution of $-\Delta u= e^{2(e^u-u)}$ in $B_{1/e}\subset \R^n$, where $f(u)=e^{2(e^u-u)}$ is a nonnegative, increasing, and convex function of $u>0$. On the other hand, we will next show that $u$, as in \eqref{counter-exp}, satisfies \eqref{eq:11}.

Indeed, using the notation
$$
B_\rho^{n}:=B_\rho\subset \R^n
$$
for balls in $\R^n$, we claim that, for $n\ge 3$,  \eqref{eq:11} is satisfied by any function of two Euclidean variables placed in $\R^n$,
\begin{equation}\label{gen-counter}
u(x)= u(x_1,\ldots,x_n)=\tilde u(x_1,x_2), \quad\text{ with $\tilde u\in W^{1,2}(B_1^{2})\setminus L^\infty(B_1^2)$.}
\end{equation}
Finally, recall that in dimension 2, $W^{1,2}(B_1^{2})$ functions may be unbounded ---\eqref{counter-exp} being a concrete example.

To verify \eqref{eq:11} we may assume that $y=0$  ---note that we are not assuming $\tilde u$ to be radially symmetric with respect to a point. Now, let
$$
x=(x',x'')\in \R^2\times \R^{n-2}, \quad r'=|x'|, \quad\text{and}\quad r''=|x''|.
$$
Since $u(x)=\tilde u(x')$, we have that
$ru_{r}=x\cdot \nabla u=r' \,\tilde u_{r'}$
and
$$
\int_{B_{\rho}^{n}}  r^{2-n}u_{r}^{2}\, dx=\int_{B_{\rho}^{2}} dx'\, r'^{\,2} \, \tilde u_{r'}^{\, 2}\int_{B^{n-2}_{\sqrt{\rho^{2}-r'^{\,2}}}}dx''\left(r'^{\,2} +r''^{\,2}\right)^{-n/2}.
$$
Now, since, for every $r'$,
$$
\int_{B^{n-2}_{\sqrt{\rho^{2}-r'^{\,2}}}}dx''\left( r'^{\,2}+r''^{\,2}\right)^{-n/2} 
=|S^{n-3}|\int_{0}^{\sqrt{\left(\frac{\rho}{r'}\right)^{2}-1}}\frac{t^{n-3}}{r'^{\, 2}\left(1+t^{2}\right)^{n/2}} \, dt \le  \frac{C}{r'^{\, 2}}
$$
for some dimensional constant $C$, we deduce that
\begin{equation}\label{upper-2d}
\int_{B_{\rho}^{n}}  r^{2-n}u_{r}^{2}\, dx\le C \int_{B_{\rho}^{2}} \tilde u_{r'}^{\,2}\, dx'.
\end{equation}

Next, to control the integral on an annulus in the right-hand side of \eqref{eq:11}, we use that
\begin{equation}\label{lower-2d} 
\int_{B_{\rho}^{n}} |\nabla u|^{2}\, dx=|B^{n-2}_1| \int_{B_{\rho}^{2}}\left|\nabla \tilde u\right|^{2}\left(\rho^{2}-r'^{\,2}\right)^{\frac{n-2}{2}} dx',
\end{equation}
both for the radius $\rho$ and for $3\rho/2$. We bound by below the integral on $B^n_{3\rho/2}$ by integrating the right-hand side of \eqref{lower-2d} (corresponding to $3\rho/2$ instead of $\rho$) only on $B^2_{\rho}$ (instead of on $B^2_{3\rho/2}$). In this way, we see that
\begin{equation*}
\begin{split}
\rho^{2-n}\int_{B_{3p/2}^{n}\setminus B_{\rho}^{n}}\left|\nabla u\right|^{2}\, dx &\\
&\hspace{-2cm} \ge 
|B^{n-2}_1|\int_{B_{\rho}^{2}} \left|\nabla \tilde u\right|^{2}  \left\{ \left( (3/2)^{2}-(r'/\rho)^{2}\right)^{\frac{n-2}{2}}-\left(1-(r'/\rho)^{2}\right)^{\frac{n-2}{2}}\right\}\, dx'.
\end{split}
\end{equation*}
Now, calling $a:=r'/\rho\in (0,1)$ and differentiating the function
$\left(t^{2}-a^{2}\right)^{\frac{n-2}{2}}$ with respect to $t$, we see that
$$
\left( (3/2)^{2}-(r'/\rho)^{2}\right)^{\frac{n-2}{2}}-\left(1-(r'/\rho)^{2}\right)^{\frac{n-2}{2}}=\left(n-2\right) \int_{1}^{3/2}\left(t^{2}-a^{2}\right)^{\frac{n-4}{2}}t\, dt.
$$
Note that $\left(t^{2}-a^{2}\right)^{\frac{n-4}{2}}\ge \left(t^{2}-1\right)^{\frac{n-4}{2}}$ if $n\geq 4$, and that $\left(t^{2}-a^{2}\right)^{\frac{n-4}{2}}\ge (3/2)^{-1}$ if $n=3$. We deduce that
$$
\rho^{2-n}\int_{B_{3p/2}^{n}\setminus B_{\rho}^{n}}\left|\nabla u\right|^{2} dx \geq c 
\int_{B_{\rho}^{2}} \left|\nabla \tilde u\right|^{2}  dx'
$$
for some positive dimensional constant $c$. This, together with \eqref{upper-2d}, establishes that~\eqref{eq:11} holds for this class of functions.

\end{document}